\documentclass[leqno,10pt]{amsart}
\usepackage{amssymb,amsmath,amsfonts,amsthm,mathtools,latexsym}
\usepackage{graphicx}
\usepackage{undertilde}
\usepackage{amsxtra}
\usepackage{xspace}
\usepackage{url}
\usepackage[english]{babel} 
\usepackage{enumitem,tikz,stmaryrd}
\usepackage[justification=centering]{caption}
\usepackage{wrapfig}
\usetikzlibrary{decorations.pathmorphing,decorations.shapes}
\usepgflibrary{arrows.meta}

\usepackage{etoolbox}

\newcommand{\kla}[1]{ {\langle #1 \rangle} }

\newcommand{\goedel}[1]{{\prec}#1{\succ}}
\newcommand{\st}{\;|\;}

\newcommand{\dom}{ {\rm dom} }
\newcommand{\ran}{ {\rm ran} }
\newcommand{\otp}{ {\rm otp} }

\newcommand{\crit}{ {\rm crit} }
\newcommand{\unt}[1]{\underline{#1}}

\newcommand{\sub}{\subseteq}

\newfont{\ssi}{cmssi12 at 12pt}

\newcommand{\eins}{ {1{\rm\hspace{-0.5ex}l}} }
\newcommand{\rest}{{\restriction}}
\newcommand{\cf}{ {\rm cf} }
\newcommand{\On}{ {\rm On} }

\newcommand{\verl}{{{}^\frown}}
\newcommand{\leer}{\emptyset}
\newcommand{\ohne}{\setminus}

\newcommand{\id}{ {\rm id} }

\newenvironment{ea*}{\begin{eqnarray*}}{\end{eqnarray*}}

\newcounter{claimnumber}
\AtEndEnvironment{proof}{\setcounter{claimnumber}{0}}
\newenvironment{NumberedClaim}%
{\begin{enumerate}[label=(\arabic*)]
\refstepcounter{claimnumber}
\setcounter{enumi}{\value{claimnumber}-1}
\item}%
{\end{enumerate}}

\setbox0=\hbox{$\longrightarrow$}
\newcommand{\To}{\longrightarrow}
\newcommand{\emb}[1]{\longrightarrow_{#1} }

\newcommand{\power}{{\mathcal{P}}}

\newcommand{\calM}{\mathcal{M}}

\newcommand{\bC}{{\bar{C}}}

\newcommand{\bG}{{\bar{G}}}

\newcommand{\bM}{{\bar{M}}}
\newcommand{\bP}{{\bar{\P}}}

\newcommand{\bS}{{\bar{S}}}

\newcommand{\balpha}{{\bar{\alpha}}}

\newcommand{\bgamma}{{\bar{\gamma}}}
\newcommand{\bdelta}{{\bar{\delta}}}

\newcommand{\bkappa}{{\bar{\kappa}}}

\newcommand{\blambda}{{\bar{\lambda}}}
\newcommand{\btau}{{\bar{\tau}}}
\newcommand{\baeta}{{\bar{\eta}}}
\newcommand{\btheta}{{\bar{\theta}}}
\newcommand{\bpi}{{\bar{\pi}}}

\newcommand{\barf}{{\bar{f}}}

\newcommand{\bN}{{\bar{N}}}

\newcommand{\tdelta}{{\tilde{\delta}}}
\newcommand{\tkappa}{{\tilde{\kappa}}}

\newcommand{\tsigma}{{\tilde{\sigma}}}

\newcommand{\tM}{{\tilde{M}}}
\newcommand{\tN}{{\tilde{N}}}

\newcommand{\vx}{{\vec{x}}}

\newcommand{\vA}{{\vec{A}}}

\newcommand{\vD}{{\vec{D}}}

\newcommand{\vM}{{\vec{M}}}

\newcommand{\vS}{{\vec{S}}}

\newcommand{\seq}[2]{{\langle#1\;|\;}\linebreak[0]{#2\rangle}}

\renewcommand{\phi}{\varphi}
\newcommand{\card}[1]{\overline{\overline{#1}}}

\newcommand{\ZFC}{\ensuremath{\mathsf{ZFC}}\xspace}

\newcommand{\ZFCm}{\ensuremath{{\ZFC}^-}\xspace}

\newcommand{\KP}{\ensuremath{\mathsf{KP}}\xspace}

\newcommand{\SCH}{\ensuremath{\mathsf{SCH}}}
\newcommand{\V}{\ensuremath{\mathrm{V}}\xspace}

\newcommand{\forces}{\Vdash}

\def\<#1>{\langle#1\rangle}

\newcommand{\B}{{\mathord{\mathbb{B}}}}
\renewcommand{\P}{{\mathord{\mathbb P}}}
\newcommand{\Q}{{\mathord{\mathbb Q}}}

\newcommand{\bbT}{\mathord{\mathbb{T}}}

\newcommand{\compose}{\circ}

\newcommand{\SC}{\ensuremath{\mathsf{SC}}\xspace}

\newcommand{\MA}{\ensuremath{\mathsf{MA}}}
\newcommand{\MM}{\ensuremath{\mathsf{MM}}\xspace}

\newcommand{\FA}{\ensuremath{\mathsf{FA}}}
\newcommand{\SCFA}{\ensuremath{\mathsf{SCFA}}\xspace}

\newcommand{\BFA}{\ensuremath{\mathsf{BFA}}}
\newcommand{\BPFA}{\ensuremath{\mathsf{BPFA}}\xspace}

\newcommand{\MP}{\ensuremath{\mathsf{MP}}}

\newcommand{\ColNothing}{\mathrm{Col}}
\newcommand{\Col}[1]{\ColNothing(#1)}

\newcommand{\MPColNothing}[1]{\MP_{\Col{\dot{\kappa}}}}

\newcommand{\isomorphic}{\cong}

\newcommand{\GCH}{\ensuremath{\mathsf{GCH}}\xspace}
\newcommand{\CH}{\ensuremath{\mathsf{CH}}\xspace}

\newcommand{\Hull}{\mathrm{Hull}} 

\newcommand{\Def}{\mathsf{Def}}

\newcommand{\Refl}{\ensuremath{\mathsf{Refl}}}

\newcommand{\FP}[1]{\ensuremath{\mathsf{FP}_{#1}}}
\newcommand{\SFP}[1]{\ensuremath{\mathsf{SFP}_{#1}}}
\newcommand{\OSR}[1]{\ensuremath{\mathsf{OSR}_{#1}}}
\newcommand{\DSR}[1]{\ensuremath{\mathsf{DSR}{(#1)}}}
\newcommand{\SRP}{\ensuremath{\mathsf{SRP}}\xspace}

\newcommand{\wBFA}{\ensuremath{\mathsf{wBFA}}}

\newcommand{\BSCFA}{\ensuremath{\mathsf{BSCFA}}\xspace}

\newcommand{\Todorcevic}{Todor\v{c}evi\'{c}\xspace}

\newtheorem{thm}{Theorem}[section]
\newtheorem*{thm*}{Theorem} 
\newtheorem{cor}[thm]{Corollary}
\newtheorem{lem}[thm]{Lemma}
\newtheorem{obs}[thm]{Observation}

\newtheorem{fact}[thm]{Fact}

\theoremstyle{definition}
\newtheorem{defn}[thm]{Definition}
\newtheorem{question}[thm]{Question}

\newcommand{\thistheoremname}{} 
\newtheorem{genericthm}[thm]{\thistheoremname} 
\newenvironment{namedthm}[1]
  {\renewcommand{\thistheoremname}{#1}%
   \begin{genericthm}}
  {\end{genericthm}}

\theoremstyle{remark}
\newtheorem{remark}[thm]{Remark}
\newtheorem{example}[thm]{Example}

\newcommand{\SSP}{\ensuremath{\mathsf{SSP}}\xspace}
\newcommand{\infSC}{\ensuremath{\mathsf{\infty\text{-}SC}}\xspace}
\newcommand{\Proper}{\ensuremath{\mathsf{Proper}}\xspace}
\newcommand{\Semiproper}{\ensuremath{\mathsf{Semiproper}}\xspace}

\newcommand{\lifting}{\mathsf{lifting}}
\renewcommand{\DSR}{\ensuremath{\mathsf{DSR}}}
\newcommand{\uDSR}{\ensuremath{\mathsf{uDSR}}}
\newcommand{\sDSR}{\ensuremath{\mathsf{sDSR}}}
\newcommand{\refl}{\ensuremath{\mathsf{Refl}}}
\newcommand{\projectdown}{\mathbin{\downarrow}}
\newcommand{\projectup}{\mathbin{\uparrow}}
\renewcommand{\card}[1]{|#1|}
\DeclareMathOperator{\Tr}{\mathsf{Tr}}
\DeclareMathOperator{\eTr}{\mathsf{eTr}}
\newcommand{\BV}[1]{\llbracket#1\rrbracket}

\newcommand{\eps}{\varepsilon}

\begin{document}

\title{Canonical fragments of the strong reflection principle}
\author{Gunter Fuchs}
\address{The College of Staten Island (CUNY)\\2800 Victory Blvd.~\\Staten Island, NY 10314}
\address{The Graduate Center (CUNY)\\365 5th Avenue, New York, NY 10016}
\email{gunter.fuchs@csi.cuny.edu}
\urladdr{www.math.csi.cuny.edu/~fuchs}
\thanks{Support for this project was provided by PSC-CUNY Award \# 63516-00 51, jointly funded by The Professional Staff Congress and The City University of New York.}
\subjclass[2020]{03E35 03E40 03E50 03E55 03E57 03E65}
\keywords{Strong reflection principle, forcing axioms, Martin's Maximum, stationary reflection, subcomplete forcing}
\date{\today}

\begin{abstract}
For an arbitrary forcing class $\Gamma$, the $\Gamma$-fragment of \Todorcevic's strong reflection principle \SRP  is isolated in such a way that (1) the forcing axiom for $\Gamma$ implies the $\Gamma$-fragment of \SRP, (2) the stationary set preserving fragment of \SRP is the full principle \SRP, and (3) the subcomplete fragment of \SRP implies the major consequences of the subcomplete forcing axiom. Along the way, some hitherto unknown effects of (the subcomplete fragment of) \SRP on mutual stationarity are explored, and some limitations to the extent to which fragments of \SRP may capture the effects of their corresponding forcing axioms are established.
\end{abstract}

\maketitle

\tableofcontents

\section{Introduction}

The strong reflection principle, \SRP, introduced by \Todorcevic (see \cite[P.~57]{Bekkali:TopicsInST}), follows from Martin's Maximum, and encompasses many of the major consequences of Martin's Maximum: the singular cardinal hypothesis, that $2^\omega=\omega_2$, that the nonstationary ideal on $\omega_1$ is $\omega_2$-saturated, and many others; see \cite[Chapter 37]{ST3} for an overview. In \cite{Fuchs:HierarchiesOfForcingAxioms}, I began a detailed study of the consequences of \SCFA, the forcing axiom for subcomplete forcing, with an eye to its relationship to Martin's Maximum. Subcomplete forcing was introduced by Jensen \cite{Jensen:SPSCF}, \cite{Jensen2014:SubcompleteAndLForcingSingapore}, and shown to be iterable with revised countable support. Since subcomplete forcing notions cannot add reals, \SCFA is compatible with \CH, which sets it apart from Martin's Maximum. In fact, Jensen \cite{Jensen:FAandCH} showed that \SCFA is even compatible with $\diamondsuit$, and hence does not imply that the nonstationary ideal on $\omega_1$ is $\omega_2$-saturated. On the other hand, \SCFA does have many of the major consequences of Martin's Maximum, such as the singular cardinal hypothesis, as mentioned above.

While my quest to deduce consequences of Martin's Maximum from \SCFA (or some related forcing principles for subcomplete forcing) has been fairly successful in many respects, such as the failure of (weak) square principles and the reflection of stationary sets of ordinals \cite{Fuchs:DiagonalReflection}, and even the existence of well-orders of $\power(\omega_1)$ \cite{Fuchs:SCprinciplesAndDefWO}, it remained unclear until recently how to find an analog of \SRP that relates to \SCFA like \SRP relates to Martin's Maximum. Thus, I was looking for a version of \SRP that follows from \SCFA and that, in turn, implies the major consequences of \SCFA.

The objective of the present article is to provide such a principle. In fact, I canonically assign to any forcing class $\Gamma$ its fragment of the strong reflection principle, which I call $\Gamma$-\SRP, in such a way that
\begin{enumerate}[label=(\arabic*)]
	\item
	\label{item:FollowsFromFA}
	the forcing axiom for $\Gamma$, $\FA(\Gamma)$, implies $\Gamma$-\SRP,
	\item
	\label{item:GeneralizesSRP}
	letting $\SSP$ be the class of all stationary set preserving forcing notions, $\SRP$ is equivalent to $\SSP$-$\SRP$,
	\item
	\label{item:ImpliesMajorConsequences}
	letting $\SC$ be the class of all subcomplete forcing notions, $\SC$-$\SRP$ captures many of the major consequences of $\SCFA$.
\end{enumerate}
While points \ref{item:FollowsFromFA} and \ref{item:GeneralizesSRP} are beyond dispute, point \ref{item:ImpliesMajorConsequences} is a little vague, and I will give some more details on which consequences of \SCFA the principle $\SC$-\SRP captures and which it cannot. The situation will turn out to be similar to the \SRP vs.~\MM comparison.



For the most part, I will be working with a technical simplification of the notion of subcompleteness, called $\infty$-subcompleteness and introduced in Fuchs-Switzer \cite{FuchsSwitzer:IterationTheorems}. This leads to a simplification of the adaptation of projective stationarity to the context of this version of subcompleteness. Working with the original notion of subcompleteness would add some technicalities, but would not change much otherwise.

The paper is organized as follows. In Section \ref{sec:Gamma-ProjectiveStationarity}, I give some background on the strong reflection principle and on generalized stationarity. Given a forcing class $\Gamma$, I introduce the notion of $\Gamma$-projective stationarity and the $\Gamma$-fragment of \SRP, and I show that the forcing axiom for $\Gamma$ implies the $\Gamma$-fragment of \SRP, as planned. I then review the definitions of subcomplete and $\infty$-subcomplete forcing and characterize ($\infty$-)subcomplete projective stationarity combinatorially (as being ``(fully) spread out'').

Section \ref{sec:SRPandConsequences} is concerned with consequences of \SRP, and mainly of the subcomplete fragment of \SRP. The main task here is to show that certain stationary sets are not only projective stationary, but even spread out. To this end, I give some background on Barwise theory and prove a technical lemma in Subsection \ref{subsec:BarwiseTheoAndTechnicalLemma}. Subsection \ref{subsec:FPandSCH} then uses this in order to establish that some of the major consequences of \SRP already follow from the subcomplete fragment of \SRP: Friedman's problem, the failure of square and the singular cardinal hypothesis. In Subsection \ref{subsec:MutualStationarity}, I derive some consequences of \SRP, and its subcomplete fragment, on mutual stationarity (which as far as I know are new even as consequences of the full \SRP).

Section \ref{sec:SRPLimitations} deals with limitations to the effects of the subcomplete fragment of \SRP on certain diagonal reflection principles for stationary sets of ordinals, and with the task of separating it from \SCFA. In Subsection \ref{subsec:GeneralSetting}, I show that the $\infty$-subcomplete fragment of \SRP is consistent with a failure of reflection at $\omega_2$, assuming the consistency of an indestructible version of \SRP. I also observe that, assuming the consistency of an indestructible version of \SRP (that follows from \MM), \infSC-\SRP does not imply \SCFA. This follows rather directly from prior results. However these results don't separate the subcomplete fragment of \SRP, together with $\CH$, from \SCFA. The last two subsections contain partial results in this direction. Subsection \ref{subsec:CHsetting-LessCanonical} shows that, assuming the consistency of an indestructible version of the fragment of \SRP used to deduce the consequences in Section \ref{sec:SRPandConsequences}, together with \CH, that fragment with \CH  fails to imply a rather weak diagonal reflection principle at $\omega_2$ which does follow from $\SCFA+\CH$, thus separating this fragment of \SRP in the presence of \CH from \SCFA. Finally, in Section \ref{subsec:CHsetting:CanonicalButOnlyUpToOmega2}, I show that if the subcomplete forcing axiom up to $\omega_2$ (I denote this $\BSCFA({\le}\omega_2)$) is consistent, then it is consistent that \CH and the subcomplete fragment of $\SRP(\omega_2)$ hold, but $\BSCFA({\le}\omega_2)$ fails. In this last result, it matters in the proof that I deal with subcompleteness, not $\infty$-subcompleteness. Along the way, I show that subcomplete forcing that preserves uncountable cofinalities is iterable with countable support.

Finally, Section \ref{sec:Questions} lists some questions and open problems.

\section{$\Gamma$-projective stationarity and the $\Gamma$-fragment of \SRP}
\label{sec:Gamma-ProjectiveStationarity}

In this section, I will give a brief introduction to the strong reflection principle, motivate how I arrive, for a forcing class $\Gamma$, at the $\Gamma$-fragment of the strong reflection principle, consider a couple of examples, and then focus on the subcomplete fragment of \SRP.

\subsection{Some background and motivation for \SRP}

Recall Friedman's problem from \cite{HFriedman:OnClosedSetsOfOrdinals}:

\begin{defn}
\label{def:FPkappa}
If $\gamma<\kappa$, then I write $S^\kappa_\gamma$ for the set of ordinals $\xi<\kappa$ with $\cf(\xi)=\gamma$. Now let $\kappa\ge\omega_2$ be an uncountable regular cardinal. Then \emph{Friedman's Problem} at $\kappa$, denoted $\FP{\kappa}$, says that whenever $S\sub S^\kappa_\omega$ is stationary, then there is a normal (that is, increasing and continuous) function $f:\omega_1\To S$. In other words, $S$ contains a closed set of order type $\omega_1$.
\end{defn}

The strong reflection principle \SRP, introduced by \Todorcevic (see \cite{Bekkali:TopicsInST} or \cite{Todorcevic:NotesOnForcingAxioms} for its original formulation), can be viewed as a version of Friedman's problem, but adapted to the generalization of stationarity due to Jech, see \cite{Jech:StationarySetsHST} for an overview article. The role of a closed set of order type $\omega_1$ is taken over by the obvious analog for the context of generalized stationarity: a continuous $\in$-chain of length $\omega_1$. I deviate slightly from the common way of presenting this.

\begin{defn}
\label{def:Chains}
Let $\kappa$ be a regular uncountable cardinal, and let $S\sub[H_\kappa]^\omega$ be stationary. A \emph{continuous $\in$-chain through $S$ of length $\lambda$} is a sequence $\seq{X_i}{i<\lambda}$ of members of $S$, increasing with respect to $\in$, such that for every limit $j<\lambda$, $X_j=\bigcup_{i<j}X_i$.
\end{defn}

Feng \& Jech \cite{FengJech:ProjectiveStationarityAndSRP} found an equivalent way to express \Todorcevic's original principle that is more amenable to generalization than its original formulation. They introduced the following concept.

\begin{defn}
\label{def:ProjectiveStationary}
Let $\kappa$ be an uncountable regular cardinal. $S\sub[H_\kappa]^\omega$ is \emph{projective stationary} if for every stationary set $T\sub\omega_1$, the set $\{X\in S\st X\cap\omega_1\in T\}$ is stationary.
\end{defn}

While I'm at it, let me introduce some terminology around generalized stationarity.

\begin{defn}
\label{defn:ProjectionsAndLiftings}
Let $\kappa$ be a regular cardinal, and let $A\sub\kappa$ be unbounded. Let $\kappa\sub X$. Then
\[\lifting(A,[X]^\omega)=\{x\in[X]^\omega\st\sup(x\cap\kappa)\in A\}\]
is the \emph{lifting} of $A$ to $[X]^\omega$.
Now let $S\sub[X]^\omega$ be stationary. If $W\subseteq X\subseteq Y$, then we define the projections of $S$ to $[Y]^\omega$ and $[W]^\omega$ by
\[S\projectup[Y]^\omega=\{y\in[Y]^\omega\st y\cap X\in S\}\]
and
\[S\projectdown[W]^\omega=\{x\cap W\st x\in S\}.\]
\end{defn}

Thus, using this notation, and letting $\kappa$ be an uncountable regular cardinal, a set $S\sub[H_\kappa]^\omega$ is projective stationary iff for every stationary $T\sub\omega_1$, $S\cap(T\projectup[H_\kappa]^\omega)$ is stationary. It is well-known that in the notation of the previous definition, $S\projectup[Y]^\omega$ and $S\projectdown[W]^\omega$ are stationary.

Following is the characterization of \SRP, due to Feng and Jech, which I take as the official definition.

\begin{defn}
\label{def:SRP}
Let $\kappa\ge\omega_2$ be regular. Then the \emph{strong reflection principle at $\kappa$,} denoted $\SRP(\kappa)$, states that whenever $S\sub[H_\kappa]^\omega$ is projective stationary, then there is a continuous 
$\in$-chain of length $\omega_1$ through $S$. The \emph{strong reflection principle} \SRP states that $\SRP(\kappa)$ holds for every regular $\kappa\ge\omega_2$.
\end{defn}

Usually, the strong reflection principle is formulated so as to assert the existence of an elementary chain of length $\omega_1$ through $S$. We'll briefly convince ourselves that this version of the principle, made precise, follows from the one stated.

\begin{obs}
\label{obs:ProjectiveStationarityPreservedUnderIntersectionWithClubs}
Projective stationarity is preserved by intersections with clubs:
if $S\sub[H_\kappa]^\omega$ is projective stationary, where $\kappa$ is regular and uncountable, then for any club $C\sub[H_\kappa]^\omega$, $S\cap C$ is also projective stationary.
\end{obs}

\begin{proof}
Let us fix $S$ and $C$. Let $T\sub\omega_1$ be stationary and $D\sub[H_\kappa]^\omega$ be club. Then, since $S$ is projective stationary, it follows that
\[\{X\in S\st X\cap\omega_1\in T\}\cap(C\cap D)\neq\leer.\]
But the set on the left is equal to
\[\{X\in S\cap C\st X\cap\omega_1\in T\}\cap D,\]
and this shows that $S\cap C$ is projective stationary.
\end{proof}

In the following, for a model $\calM$ of a first order language, I write $|\calM|$ for its universe. Further, if $X\sub|\calM|$, then $\calM|X$ is the restriction of $\calM$ to $X$.

\begin{cor}
\label{cor:SRPgivesElementaryChains}
Assume $\SRP(\kappa)$.
Let $S\sub[H_\kappa]^\omega$ be projective stationary. Let $\calM=\kla{H_\kappa,\in\rest H_\kappa,\ldots}$ be a first order structure of a finite language.
\footnote{Since $H_\kappa$ is closed under ordered pairs, it is easy to code any countable language into a finite language, and it will sometimes be convenient to assume that the language at hand is finite, so I will usually make that assumption.}
Then there is a continuous elementary chain $\seq{\calM_i}{i<\omega_1}$ of elementary submodels of $\calM$ through $S$, meaning that, for all $i<\omega_1$, $|\calM_i|\in S$, $\calM_i\in|\calM_{i+1}|$, $\calM_i\prec\calM_{i+1}$ and if $i$ is a limit ordinal, then $|\calM_i|=\bigcup_{j<i}|\calM_j|$.
\end{cor}

\begin{proof}
This follows immediately from the previous observation, since $C=\{X\in[H_\kappa]^\omega\st(\calM|X)\prec\calM\}$ contains a club. Any continuous chain $\seq{X_i}{i<\omega_1}$ through $S\cap C$ gives rise to a continuous elementary chain of models by setting $\calM_i=\calM|X_i$. Note that, for $i<\omega_1$, since $X_i\in X_{i+1}$, it follows that $\calM_{i+1}$ sees that $X_i$ is countable, as $\calM_{i+1}\prec\calM$, and hence, $X_i\sub X_{i+1}$ and $\calM_i\prec\calM_{i+1}$. Also, since $\calM_i$ is definable from $X_i\in|\calM_{i+1}|$, as the language of $\calM$ is finite, it follows that $\calM_i\in\calM_{i+1}$.
\end{proof}

The key to showing that Martin's Maximum implies \SRP is that the canonical forcing to shoot a continuous $\in$-chain of length $\omega_1$ through a projective stationary set, described in the following definition, preserves stationary subsets of $\omega_1$.

\begin{defn}
\label{def:P_S}
$\P_S$ is the forcing notion consisting of continuous $\in$-chains through $S$ of countable successor length, ordered by end-extension. For $p\in\P_S$, I write $p=\seq{M^p_i}{i\le\ell^p}$.
\end{defn}

The following fact is essentially contained in Feng \& Jech \cite{FengJech:ProjectiveStationarityAndSRP}, even though it is not explicitly stated.

\begin{fact}
\label{fact:P_SisCountablyDistributiveAndExhaustive}
Let $\kappa$ be an uncountable regular cardinal, and let $S\sub[H_\kappa]^\omega$ be stationary. Then
\begin{enumerate}[label=(\arabic*)]
  \item
  \label{item:GenericSequenceHasLengthOmega1}
  for every countable ordinal $\alpha$, the set of conditions $p$ with $\ell^p\ge\alpha$ is dense in $\P_S$,
  \item
  \label{item:UnionOfGenericIsExhaustive}
  for every $a\in H_\kappa$, the set of conditions $p$ such that there is an $i<\ell^p$ with $a\in M^p_i$ is dense in $\P_S$,
  \item
  \label{item:CountablyDistributive}
  $\P_S$ is countably distributive.
\end{enumerate}
\end{fact}

\begin{proof}
We prove clauses \ref{item:GenericSequenceHasLengthOmega1} and \ref{item:UnionOfGenericIsExhaustive} simultaneously. Let $p\in\P_S$, $\alpha<\omega_1$ and $a\in H_\kappa$ be given. We use
Lemma 1.2 of \cite{FengJech:ProjectiveStationarityAndSRP}, which states that if $T\sub[H_\kappa]^\omega$ is stationary, then, for every countable ordinal $i$, there is a continuous $\in$-chain through $T$ of length at least $i+1$. Let $T=\{M\in S\st a, M^p_{\ell^p}\in M\}$. Clearly, $T$ is stationary, so by the lemma, let $\seq{N_i}{i\le\beta}$ be a continuous elementary chain through $T$, with $\beta\ge\alpha$. Let $\gamma=\ell^p+1+\beta+1$, and define a condition $q=\seq{M^q_i}{i<\gamma}$ by setting $M^q_i=M^p_i$ for $i\le\ell^p$ and $M^q_{\ell^p+1+j}=N_j$ for $j\le\beta$. Then $q\le p$ has length at least $\alpha+1$ and eventually contains $a$, as wished.

In order to prove clause \ref{item:CountablyDistributive}, we have to show that, given a sequence $\vec{D}=\seq{D_n}{n<\omega}$ of dense open subsets of $\P_S$, the intersection $\Delta=\bigcap_{n<\omega}D_n$ is dense in $\P_S$. So, fixing a condition $p\in\P_S$, we have to find a $q\le p$ in $\Delta$. To this end, let $\lambda$ be a regular cardinal much greater than $\kappa$, say $\lambda>2^{2^{\card{\P_S}}}$, and consider the model $\mathcal{N}=\kla{H_\lambda,\in,<^*,S,\P_S,\vec{D},p}$, where $<^*$ is a well-ordering of $H_\lambda$. Let $\calM\prec\mathcal{N}$ be a countable elementary submodel with $|\calM|\cap H_\kappa\in S$. Since $\calM$ is countable, we can pick a filter $G$ which is $\calM$-generic for $\P_S$ and contains $p$. Let $\bar{q}=\bigcup G$. Using the density facts proved in \ref{item:GenericSequenceHasLengthOmega1} and \ref{item:UnionOfGenericIsExhaustive}, it follows that $\delta:=\dom(\bar{q})=M\cap\omega_1$, and that $\bigcup_{i<\delta}\bar{q}(i)=M\cap\kappa\in S$. Thus, if we define the sequence $q$ of length $\delta+1$ by setting $q(i)=\bar{q}(i)$ for $i<\delta$ and $q(\delta)=M\cap\kappa$, then $q\in\P_S$, and $q$ extends every condition in $G$. Moreover, since $D_n\in M$, for each $n<\omega$, it follows that $G$ meets each $D_n$, and hence that $p\ge q\in\Delta$, as desired.
\end{proof}

\begin{fact}[Feng \& Jech]
\label{fact:SisProjStatIffP_SisSSP}
Let $\kappa\ge\omega_2$ be an uncountable regular cardinal. Then a stationary set $S\sub[H_\kappa]^\omega$ is projective stationary iff $\P_S$ preserves stationary subsets of $\omega_1$.
\end{fact}

For the proof of this fact, see Feng \& Jech \cite{FengJech:ProjectiveStationarityAndSRP} -- one direction is given by the proof of Theorem 1.1, and the converse is outlined in the paragraph after the proof, on page 275.

\subsection{Relativizing to a forcing class}
\label{subsec:RelativizingToForcingClass}

\begin{defn}
\label{def:SSP}
I write $\SSP$ for the class of all forcing notions that preserve stationary subsets of $\omega_1$.
\end{defn}

With hindsight, the results in the previous subsection show that the strong reflection principle can be formulated as follows.

\bigskip

\begin{quote}
  \emph{Whenever $\kappa\ge\omega_2$ is regular, $S\sub[H_\kappa]^\omega$ is stationary, and the forcing $\P_S$ to shoot a continuous elementary chain through $S$ is in \SSP, then $S$ already contains a continuous $\in$-chain of length $\omega_1$.}
\end{quote}

\bigskip

The advantage of this formulation is that it generalizes easily to arbitrary forcing classes. First, let me generalize the concept of projective stationarity.

\begin{defn}
\label{def:GammaProjectiveStationary}
Let $\Gamma$ be a forcing class. Then a stationary subset $S$ of $H_\kappa$, where $\kappa\ge\omega_2$ is regular, is \emph{$\Gamma$-projective stationary} iff $\P_S\in\Gamma$.
\end{defn}

Thus, the Feng-Jech notion of projective stationarity is the same thing as \SSP-projective stationarity. Generalizing the above formulation of \SRP, we arrive at:

\begin{defn}
\label{def:Gamma-SRP}
Let $\Gamma$ be a forcing class. Let $\kappa\ge\omega_2$ be regular. The \emph{$\Gamma$-fragment of the strong reflection principle at $\kappa$}, denoted $\Gamma$-$\SRP(\kappa)$, states that whenever $S\sub[H_\kappa]^\omega$ is $\Gamma$-projective stationary, then $S$ contains a continuous chain of length $\omega_1$. The \emph{$\Gamma$-fragment of the strong reflection principle,} $\Gamma$-\SRP, states that $\Gamma$-$\SRP(\kappa)$ holds for every $\kappa\ge\omega_2$.
\end{defn}

The idea is that the collection of the $\Gamma$-projective stationary sets captures exactly those sets whose instance of the strong reflection principle follows from the forcing axiom for $\Gamma$ using the simplest possible argument, namely that $\P_S$ is in $\Gamma$. Let me make this precise. First, by the forcing axiom for $\Gamma$, I mean the version of Martin's Axiom $\MA_{\omega_1}$ for $\Gamma$ rather than the collection of all ccc partial orders.

\begin{defn}
\label{def:FA_Gamma}
Let $\Gamma$ be a class of forcing notions. The \emph{forcing axiom for $\Gamma$}, denoted $\FA(\Gamma)$, states that whenever $\P$ is a forcing notion in $\Gamma$ and $\seq{D_i}{i<\omega_1}$ is a sequence of dense subsets of $\P$, there is a filter $F\sub\P$ such that for all $i<\omega_1$, $F\cap D_i\neq\leer$.
\end{defn}

It is now easy to check that $\Gamma$-\SRP behaves as claimed in the introduction.

\begin{obs}
Let $\Gamma$ be a forcing class. Then $\FA(\Gamma)$ implies $\Gamma$-\SRP.
\end{obs}

\begin{proof}
Let $\kappa\ge\omega_2$ be regular, and let $S\sub[H_\kappa]^\omega$ be $\Gamma$-projective stationary. Then $\P_S\in\Gamma$, and, for $i<\omega_1$, we can let $D_i$ be the set of conditions in $\P_S$ of length at least $i$. By clause \ref{item:GenericSequenceHasLengthOmega1} of Fact \ref{fact:P_SisCountablyDistributiveAndExhaustive}, $D_i$ is a dense subset of $\P_S$. So by $\FA(\Gamma)$, there is a filter $F$ meeting each $D_i$. But then $\bigcup F$ is a continuous $\in$-chain through $S$.
\end{proof}

The utility of \SRP is, of course, that it encapsulates many of the consequences of the forcing axiom for stationary set preserving forcing without mentioning forcing. Thus, in order to arrive at a similarly useful version of it for other forcing classes, it will be crucial to express $\Gamma$-projective stationarity in a purely combinatorial way that does not mention $\Gamma$ explicitly.

As an illustration, let's look at two examples.

\begin{example}
\label{example:ProperSRPtrivial}
Let \Proper be the class of all proper forcing notions, and let's consider the notion of projective stationarity associated to that class. It is then not hard to see that:

\begin{obs}
Let $\kappa\ge\omega_2$ be regular. Then a stationary set
$S\sub[H_\kappa]^\omega$ is \Proper-projective stationary iff $S$ contains a club.
\end{obs}

\begin{proof}
For the forward direction, suppose that $S$ is \Proper-projective stationary, that is, $\P_S$ is proper. One of the many characterizations of properness is the preservation of stationary subsets of $[X]^\omega$, for any uncountable $X$. Now $\P_S$ shoots a club through $S$, and this means that the complement $[H_\kappa]^\omega\ohne S$ could not have been stationary in $\V$, since its stationarity would be killed by $\P_S$. But this means that $S$ contains a club.

For the converse, suppose that $S$ contains a club $C\sub[H_\kappa]^\omega$. Let $\theta$ be sufficiently large, and let $M$ be a countable elementary submodel of $\kla{H_\kappa,\in,<^*}$, with $\P_S,C\in M$. Let $p\in\P_S\cap M$. Let $N=M^p_{\ell^q}$. Since $M$ believes that $C$ is club in $H_\kappa$, it is now easy to construct an $\in$-chain $\seq{N_i}{i<\omega}$ so that $N\in N_0$, $N_i\in M$ and $\bigcup_{i<\omega}N_i=M\cap H_\kappa$. It is then routine to verify that the condition $q=p\verl\vec{N}\verl(M\cap H_\kappa)$ is $(M,\P_S)$-generic.
\end{proof}

Since any club contains a continuous $\in$-chain of length $\omega_1$, the proper fragment of \SRP is thus provable in \ZFC.
\end{example}

\begin{example}
\label{example:SemiproperSRPisSRP}
For an example going in the other extreme, let $\Semiproper$ be the class of semiproper partial orders. In \cite{FengJechZapletal:StructureOfStationarySets}, a set $S\sub[\kappa]^\omega$ is defined to be \emph{spanning} if for every $\lambda\ge\kappa$ and every club $C\sub[\lambda]^\omega$, there is a club $D\sub[\lambda]^\omega$ such that for every $x\in D$, there is a $y\in C$ such that $x\sub y$ and $x\cap\omega_1=y\cap\omega_1$ and $y\cap\kappa\in S$. It is shown in \cite[Theorem 4.4]{FengJechZapletal:StructureOfStationarySets} that $S$ is spanning iff $\P_S$ is semiproper, that is, using our terminology, $S$ is spanning iff it is $\Semiproper$-projective stationary. However, \cite[Cor.~5.4]{FengJechZapletal:StructureOfStationarySets} can be expressed as saying that $\Semiproper$-\SRP implies \SRP, so the semiproper fragment of \SRP is equivalent to the full principle \SRP.
\end{example}

\subsection{The subcomplete fragment of \SRP}
\label{subsec:SCfragment}

In the previous subsection, we have seen that the class of all proper forcing notions is too small to be of interest, in the sense that $\Proper$-\SRP is provable in \ZFC, and the class of all semiproper forcing notions is too large to be of interest, in the sense that $\Semiproper$-\SRP is equivalent to the full principle \SRP, and so is nothing new. So let us now get ready to define when a forcing notion is subcomplete, so that we can turn to the subcomplete fragment of \SRP.

\begin{defn}[Jensen]
\label{def:Full}
A transitive model $N$ of \ZFCm is \emph{full} if there is an ordinal $\gamma>0$ such that $L_\gamma(N)\models\ZFCm$ and $N$ is regular in $L_\gamma(N)$, meaning that if $a\in N$, $f:a\To N$ and $f\in L_\gamma(N)$, then $\ran(f)\in N$. A possibly nontransitive but well-founded model of \ZFCm is full if its transitive isomorph is full.
\end{defn}

The notion of fullness is central to the theory of subcomplete forcing, and so, it seems worthwhile to elaborate on it a little bit, since it is somewhat subtle. First off, when I say that $N$ is a transitive model of \ZFCm, I mean that $N$ is a model of a countable language which may extend the language of set theory, in which the symbol $\dot{\in}$ is interpreted as the actual $\in$ relation, restricted to $N$, and that $N$ satisfies the usual axioms of \ZFCm, with respect to its language, that is, the formulas in the axiom schemes are allowed to contain the additional symbols available in the language of $N$. There is a subtlety in the concept of fullness, then, since whether or not a model $N$ is full depends on the way it is represented. For a simple example, let's assume that $N$ is a countable full model of \ZFCm in the language of set theory. Now let us consider $N_0$ to be like $N$, except that $N_0$ has a constant symbol $c_a$ for every $a\in N_0$, so that $c_a^{N_0}=a$. Clearly, $N_0$ is also a model of \ZFCm, and $N_0$ is also full. Now let $N_1$ be like $N$, but equipped with constant symbols $d_0, d_1, \ldots$, interpreted as $d_n^{N_1}=f(n)$, where $f:\omega\To N$ is a bijection. In a model-theoretic sense, $N_0$ and $N_1$ are essentially the same, it is just that their constant symbols are different. However, $N_0$ is full, while $N_1$ is not, since in $L_\gamma(N_1)$, the function $n\mapsto d_n^{N_1}=f(n)$ is available, and hence the fact that $N_1$ is countable is revealed. Thus, in order to make sense of the definition of fullness, one has to view the model $N$ as the triple $\kla{|N|,\mathcal{L},I}$, where $|N|$ is the universe of $N$, $\mathcal{L}$ is the language of $N$ in an explicitly given G\"odelization (since this example shows that it is important what the symbols in the language are), and $I$ is the function assigning each element of $\mathcal{L}$ its interpretation in $N$. In the context of subcomplete forcing, the model $N$ in question will always be a model of a language with just one additional predicate symbol (which avoids the complications just mentioned). In fact, it will always be the result of constructing relative to some set. The notation I use for relative constructibility follows Jensen's conventions: for a class $A$, define recursively:
\begin{itemize}
\item $L_0[A]=\leer$, $L_0^A=\kla{\leer,\leer,\leer}$,
\item $L_{\alpha+1}[A]=\Def(L_\alpha^A)$,\footnote{Here, in the case $\alpha=0$, $L_0[A]$ is not technically a model, because its universe is empty, so we have to set $\Def(\kla{\leer,\leer,\leer})=\{\leer\}$ to make literal sense of this definition.}
    $L_{\alpha+1}^A=\kla{L_{\alpha+1}[A],\in\rest L_{\alpha+1}[A],A\cap L_{\alpha+1}[A]}$,
\item for limit $\lambda$, $L_\lambda[A]=\bigcup_{\alpha<\lambda}L_\alpha[A]$ and $L_\lambda^A=\kla{L_\lambda[A],\in\rest L_\lambda[A],A\cap L_\lambda[A]}$.
\end{itemize}

\begin{defn}
\label{def:Density}
The \emph{density} of a poset $\P$, denoted $\delta(\P)$, is the least cardinal $\delta$ such that there is a dense subset of $\P$ of size $\delta$.
\end{defn}

I can now define Jensen's notion of subcompleteness and its simplification, $\infty$-subcompleteness, introduced in \cite{FuchsSwitzer:IterationTheorems}.

\begin{defn}
\label{def:(ininifty-)subcompleteness}
A forcing notion $\P$ is \emph{subcomplete} if every sufficiently large cardinal $\theta$ \emph{verifies} the subcompleteness of $\P$, which means that $\P\in H_\theta$, and for any \ZFCm{} model $N=L_\tau^A$ with $\theta<\tau$ and $H_\theta\sub N$, any $\sigma:\bN\prec N$ such that $\bN$ is countable, transitive and full and such that $\P,\theta\in\ran(\sigma)$, any $\bar{G}\sub\bar{\P}$ which is $\bar{\P}$-generic over $\bN$, any $\bar{s}\in\bN$, and any ordinals $\blambda_0,\ldots,\blambda_{n-1}$ such that $\blambda_0=\On\cap\bN$ and $\blambda_1,\ldots,\blambda_{n-1}$ are regular in $\bN$ and greater than $\delta(\bar{\P})^\bN$, the following holds. Letting $\sigma(\kla{\bar{\theta},\bar{\P}})=\kla{\theta,\P}$, and setting $\bar{S}=\kla{\bar{s},\bar{\theta},\bar{\P}}$, there is a condition $p\in\P$ such that whenever $G\sub\P$ is $\P$-generic over $\V$ with $p\in G$, there is in $\V[G]$ a $\sigma':\bN\prec N$ such that
	\begin{enumerate}[label=(\arabic*)]
		\item
        \label{item:AgreementOnParameters}
        $\sigma'(\bar{S})=\sigma(\bar{S})$,
		\item
        \label{item:LiftingCondition}
        $(\sigma')``\bar{G}\sub G$,
		\item
		\label{item:SupremumCondition}
        $\sup\sigma``\blambda_i=\sup\sigma'``\blambda_i$ for each $i<n$.
	\end{enumerate}
$\P$ is \emph{$\infty$-subcomplete} iff the above holds with \ref{item:SupremumCondition} removed.

I denote the classes of subcomplete and $\infty$-subcomplete forcing notions by $\SC$ and $\infSC$, respectively.
\end{defn}

It should be pointed out that full models as in the previous definition are abundant. For example, suppose that $H_\theta\sub L[A]$, where $A\sub L_\beta[A]$, and let $\tau<\tau'$ be successive cardinals in $L[A]$, say, with $\beta<\tau$. Then whenever $X'\prec L_{\tau'}^A$ and $X=X'\cap L_\tau[A]$, it follows that $L_\tau^A|X$ is full.

The following easy fact can be used in order to further simplify the definitions of subcompleteness/$\infty$-subcompleteness.

\begin{fact}
\label{fact:AbsorbingParameters}
Let $L_\tau^A$ be a model of \ZFCm, and let $s\in L_\tau[A]$. Then there is a $B$ such that $L_\tau[B]=L_\tau[A]$, $L_\tau^B\models\ZFCm$ and such that $s$ is definable (without parameters) in $L_\tau^B$. Moreover, $B$ is definable in $L_\tau^A$ and $A\cap L_\tau[A]$ is definable in $L_\tau^B$. In particular, $L_\tau^A$ is full iff $L_\tau^B$ is.
\end{fact}

\begin{proof}
First, by replacing $A$ with $A\cap L_\tau[A]$ if necessary, we may assume that $A\sub L_\tau[A]$. Second, we may assume that $A\sub\tau$. That is, we may construct a set $A'\sub\tau$ such that $L_\tau[A]=L_\tau[A']$, $A$ is definable in $L_\tau^{A'}$ and $A'$ is definable in $L_\tau^A$. Namely, $L_\tau^A$ has a definable well-order of its universe, and since it is a model of \ZFCm, the monotone enumeration of $L_\tau[A]$ according to this well-order is definable in $L_\tau^A$, and its domain is $\tau$. Let's call it $F:\tau\To L_\tau[A]$. Let $R=\{\kla{\alpha,\beta}\st F(\alpha)\in F(\beta)\}$. Then $F$ is the Mostowski-collapse of the structure $\kla{\tau,R}$. Now it is easy to encode $R$ and $A$ as a set of ordinals, using G\"{o}del pairs, for example, say
\[A'=\{\goedel{0,\alpha,\beta}\st F(\alpha)\in F(\beta)\}\cup\{\goedel{1,\alpha}\st F(\alpha)\in A\}.\]
Since $A'$ is a definable class in the $\ZFCm$-model $L_\tau^A$, it follows that $L_\tau^{A'}=(L^{A'})^{L_\tau^A}$ is also a model of $\ZFCm$, and since $A'$ codes $F$, it follows that $L_\tau[A']=L_\tau[A]$. Moreover, $A$ is definable in $L_\tau^{A'}$, by design.

It is now easy to prove the fact: we can define $B=\{\goedel{0,\alpha}\st\alpha\in A\}\cup\{\goedel{1,\gamma}\}$, where $s$ is the $\gamma$-th element of $L_\tau[A]$ in the canonical well-order.
\end{proof}

Of course, if any one element of $L_\tau[A]$ can be made definable by changing $A$ as in the previous fact, then any finitely many elements can be made definable by applying the same method to a finite sequence listing these elements. A consequence of this fact, or rather, its proof, is that in Definition \ref{def:(ininifty-)subcompleteness}, condition \ref{item:AgreementOnParameters} is vacuous, because if $\P$ satisfies this simplified definition, in the notation of that definition, one can modify $A$ to $A'$ in such a way that the desired parameters in $\bar{S}$ become definable in $N'=L_\tau^{A'}$. Letting $\bN=L_\tau^{\bar{A}}$, and $\bN'=L_\btau^{\bar{A}'}$ (where $\bar{A}'=\sigma^{-1}``A$ is constructed from $\bar{A}$ the same way that $A'$ is constructed from $A$), it then follows that $\bN'$ is full and $\sigma:\bN'\prec N'$. Thus, since $\P$ satisfies the simplified version of subcompleteness, there is a condition in $\P$ forcing the existence of an elementary embedding $\sigma':\bN'\prec N$ such that $\sigma'``\bG\sub\dot{G}$, where $\dot{G}$ is the canonical name for the generic filter. But then, $\sigma':\bN\prec N$ as well, and $\sigma'$ must move the desired parameters the same way $\sigma$ did, since they/their preimages are definable in $N'$/$\bN'$. This means, in particular, that only condition \ref{item:LiftingCondition} is really needed in the definition of $\infty$-subcompleteness.

%
%

The following definition is designed to capture the concept of $\infSC$-projective stationarity. If $N$ is a model and $X$ is a subset of $|N|$, the universe of $N$, then I write $N|X$ for restriction of $N$ to $X$.

\begin{defn}
\label{def:SpreadOut}
Let $\kappa$ be an uncountable regular cardinal. A stationary set $S\sub[H_\kappa]^\omega$ is \emph{spread out} if for every sufficiently large cardinal $\theta$, whenever $\tau$, $A$, $X$ and $a$ are such that $H_\theta\sub L_\tau^A=N\models\ZFCm$, $S,a,\theta\in X$, $N|X\prec N$, and $N|X$ is countable and full, then there are a $Y$ such that $N|Y\prec N$ and an isomorphism $\pi:N|X\To N|Y$ such that $\pi(a)=a$  and $Y\cap H_\kappa\in S$.
\end{defn}

Using Fact \ref{fact:AbsorbingParameters} as before, one can see that the definition of being spread out can be simplified by dropping any reference to $a$, since any desired parameter, or even any finite list of such parameters, can be made definable by modifying $A$ while preserving fullness. So  $S\sub[H_\kappa]^\omega$ is spread out if for all sufficiently large $\theta$, whenever $H_\theta\sub L_\tau^A=N\models\ZFCm$, $S\in X$, $N|X\prec N$, and $N|X$ is countable and full, then there is a $Y\in S\projectup[N]^\omega$ such that $N|X\isomorphic N|Y\prec N$.

Thus, in the situation of the previous definition, the stationarity of $S$ guarantees the existence of \emph{some} elementary submodel of $N$ in $S\projectup[N]^\omega$, but if $S$ is spread out, then \emph{every} elementary submodel of $N$ has an isomorphic copy in $S\projectup[N]^\omega$, as long as it is full.

The following theorem is the analog of Fact \ref{fact:SisProjStatIffP_SisSSP} for $\infty$-subcompleteness, providing a combinatorial characterization of $\infSC$-projective stationarity.

\begin{thm}
\label{thm:SpreadOutIffinfSCProjectiveStationary}
Let $\kappa$ be an uncountable regular cardinal, and let $S\sub[H_\kappa]^\omega$. Then $S$ is spread out iff $S$ is \infSC-projective stationary.
\end{thm}

\begin{proof}
For the direction from left to right, suppose $S$ is spread out. We have to show that $\P_S\in\infSC$. To that end, let $\theta$ be large enough for Definition \ref{def:SpreadOut} to apply. Let $N=L_\tau^A\models\ZFCm$ with $H_\theta\sub N$, and let $\P_S\in X$, $N|X\prec N$, $X$ countable and full. Let $a$ be some member of $X$, and let $\sigma:|\bN|\To X$ be the inverse of the Mostowski collapse of $X$, $|\bN|$ transitive, and $\sigma:\bN\prec N$. Let $\bP_\bS=\sigma^{-1}(\P_S)$, $\bar{a}=\sigma^{-1}(a)$, and let $\bG\sub\bP_\bS$ be $\bN$-generic. Note that since $\P_S\in X$, $S, H_\kappa\in X$. Let $\bkappa=\sigma^{-1}(\kappa)$. It follows from Fact \ref{fact:P_SisCountablyDistributiveAndExhaustive} that $\bigcup\bG$ is of the form $\seq{\bM_i}{i<\omega_1^\bN}$ and $\bigcup_{i<\omega_1^\bN}\bM_i=H_\bkappa^\bN$. Now, since $S$ is spread out, let $\pi:\kla{X,\in}\To\kla{Y,\in}$ be an isomorphism that fixes $a,\P_S$, with $Y\cap H_\kappa\in S$. Let $\sigma'=\pi\circ\sigma:\bN\prec Y$. Let $q=\seq{\sigma'(\bM_i)}{i<\omega_1^\bN}\verl(Y\cap H_\kappa)$. Since $Y\cap H_\kappa\in S$, it follows that $q\in\P_S$, and whenever $G\ni q$ is $\P_S$-generic over $\V$, then $\sigma'``\bG\sub G$. Since $\sigma'(\bar{a})=a$, this shows that $\P_S$ is $\infty$-subcomplete.

For the converse, suppose that $S$ is $\infSC$-projective stationary, that is, that $\P_S$ is $\infty$-subcomplete. Let $\theta$ witness that $\P_S$ is $\infty$-subcomplete. Let $N=L_\tau^A$, $X$, $a$ be as in Definition \ref{def:SpreadOut}. Since $S\in X$, it follows that $\kappa,\P_S\in X$ as well. Let $\sigma:\bN\To X$ be the inverse of the Mostowski collapse of $X$. Thus, $\sigma:\bN\prec N$, and as usual, let $\bar{a}$, $\bS$, $\bkappa$, $\bP_\bS$ denote the preimages of $a$, $S$, $\kappa$, $\P_S$ under $\sigma$. Let $\bG\sub\bP_\bS$ be an arbitrary $\bN$-generic filter. By $\infty$-subcompleteness of $\P_S$, let $p\in\P_S$ force the existence of an elementary embedding $\sigma':\bN\To N$ with $\sigma'(\bar{a})=a$ and $(\sigma')``\bG\sub\dot{G}$ ($\dot{G}$ being the canonical $\P_S$-name for the generic filter). Since $\P_S$ is countably distributive by Fact \ref{fact:P_SisCountablyDistributiveAndExhaustive}, it follows that there is such a $\sigma'\in\V$. Let $Y=\ran(\sigma')$, and let $G$ be $\P_S$-generic over $\V$ with $p\in G$. Let $\delta=\omega_1^\bN=\omega_1\cap X$. As before, $\bigcup\bG$ is of the form $\seq{\bM_i}{i<\delta}$. By Fact \ref{fact:P_SisCountablyDistributiveAndExhaustive}, we have that $\bM^*=\bigcup_{i<\delta}\bM_i=H_{\bkappa}^\bN$. For $i<\delta$, let $M_i=\sigma'(\bM_i)$ - note that this is the same as $(\sigma')``\bM_i$, as $\bM_i$ is countable in $\bN$. Since $G$ contains a condition of length $\delta+1$, letting $M_\delta=\bigcup_{i<\delta}M_i$, the sequence $q=\seq{M_i}{i\le\delta}$ is in $G$. It follows that $M_\delta\in S$. Moreover, \[M_\delta=\bigcup_{i<\delta}\sigma'``\bM_i=\sigma'``\bM^*=\sigma'``H_\bkappa^\bN=Y\cap H_\kappa^N\]
and so, $Y\cap H_\kappa\in S$. Letting $\pi=\sigma'\compose\sigma^{-1}$, one sees that $\pi:X\To Y$ is an isomorphism that fixes $a$, thus verifying that $S$ is spread out.
\end{proof}

Having a characterization of $\infSC$-projective stationarity of course gives a characterization of the $\infSC$-fragement of \SRP.

\begin{thm}
For an uncountable regular cardinal $\kappa$, the principle $\infSC$-$\SRP(\kappa)$ holds iff every spread out subset of $[H_\kappa]^\omega$ contains a continuous $\in$-chain of length $\omega_1$.
\end{thm}

It will often be useful to work with the following seemingly weaker form of the notion ``spread out.'' It will turn out to be equivalent, but it will sometimes be easier to verify.

\begin{defn}
\label{def:WeaklySpreadOut}
Let $\kappa$ be an uncountable regular cardinal. A stationary set $S\sub[H_\kappa]^\omega$ is \emph{weakly spread out} if there is a set $b$ such that for all sufficiently large $\theta$, the condition described in Definition \ref{def:SpreadOut} is true of all $X$ with $b\in X$.
\end{defn}

\begin{obs}
\label{obs:WeaklySpreadOutImpliesSpreadOut}
Let $\kappa$ be an uncountable regular cardinal. A stationary set $S\sub[H_\kappa]^\omega$ is spread out iff it is weakly spread out.
\end{obs}

\begin{proof}
Of course, if $S$ is spread out, it is also weakly spread out. For the converse, suppose $S$ is weakly spread out, as witnessed by the set $b$, say. Then the proof of Theorem \ref{thm:SpreadOutIffinfSCProjectiveStationary} shows that $\P_S$ satisfies the definition of $\infty$-subcompleteness, Definition \ref{def:(ininifty-)subcompleteness}, under the extra condition that $b\in\ran(\sigma)$, in the notation of that definition. But this implies that $\P_S$ is $\infty$-subcomplete, see the arguments in Jensen \cite[P.~115f., in particular Lemma 2.5]{Jensen2014:SubcompleteAndLForcingSingapore}. But if $\P_S$ is $\infty$-subcomplete, then $S$ is spread out, by Theorem \ref{thm:SpreadOutIffinfSCProjectiveStationary}.
\end{proof}

I would now like to make a few simple observations on the structure of the spread out sets and their relation to other notions of largeness of subsets of $[H_\kappa]^\omega$. First, of course, being spread out is a strengthening of projective stationarity.

\begin{obs}
\label{obs:SpreadOutImpliesProjectiveStationary}
If a stationary set $S\sub[H_\kappa]^\omega$ is spread out, then $S$ is projective stationary.
\end{obs}

\begin{proof}
This is because $\infSC\sub\SSP$.
\end{proof}

In particular, spread out sets are stationary. In fact, being spread out is preserved by intersecting with a club; this is the analog of Fact \ref{obs:ProjectiveStationarityPreservedUnderIntersectionWithClubs}.

\begin{obs}
\label{obs:SpreadOutSetsClosedUnderIntersectionWithClubs}
Let $\kappa$ be an uncountable regular cardinal, let $S\sub[H_\kappa]^\omega$ be spread out, and let $C\sub[H_\kappa]^\omega$ be club. Then $S\cap C$ is spread out.
\end{obs}

\begin{proof}
Let $f:H_\kappa^{{<}\omega}\To H_\kappa$ be such that every $a\in H_\kappa$ closed under $f$ is in $C$.
By Observation \ref{obs:WeaklySpreadOutImpliesSpreadOut}, it suffices to show that $S\cap C$ is weakly spread out. Thus, it will be enough to show that the condition described in Definition \ref{def:SpreadOut} is satisfied for all sufficiently large $\theta$, assuming that $f\in X$, using the notation in the definition. Since $S$ is spread out, there is a $Y$ such that $Y\cap H_\kappa$ in $S$ and such that there is an isomorphism $\pi:\kla{X,\in\cap X^2}\To\kla{Y,\in\cap Y^2}$ that fixes a given $a$, and also $f$, that is, $\pi(\kla{f,a})=\kla{f,a}$. Since $f\in Y$, it follows that $Y$ is closed under $f$, and hence that $Y\cap H_\kappa\in S\cap C$.
\end{proof}

Thus, $\infSC$-\SRP guarantees the existence of \emph{elementary} chains through spread out sets. Thirdly, being spread out is preserved by projections, analogous to the situation with stationarity and projective stationarity.

\begin{obs}
\label{obs:SpreadOutPreservedUnderProjections}
Let $A\sub B\sub C$, and let $S\sub[B]^\omega$ be spread out, with $\bigcup S=B$. Then:
\begin{enumerate}[label=(\arabic*)]
\item
\label{item:ProjectingUpwardPreservesBeingSpreadOut}
$S\projectup C$ is spread out.
\item
\label{item:ProjectingDownwardPreservesBeingSpreadOut}
$S\projectdown A$ is spread out.
\end{enumerate}
\end{obs}

\begin{proof}
We prove \ref{item:ProjectingUpwardPreservesBeingSpreadOut} and \ref{item:ProjectingDownwardPreservesBeingSpreadOut} simultaneously. Let $\theta$ be sufficiently large, and let $X\prec N=L_\tau^D\models\ZFCm$ be countable and full, with $S,a,A,C\in N$ (as usual, we may require some additional parameter to be in $X$). Since $S$ is spread out, there are a $Y$ with $N|Y\prec N$ and an isomorphism $\pi:N|X\To N|Y$ that fixes $a$, and such that $Y\cap B\in S$. But this means that $Y\cap C\in S\projectup C$ and that $Y\cap A\in S\projectdown A$, as wished.
\end{proof}

\begin{obs}
\label{obs:ContainingAClubImpliesBeingSpreadOut}
If $S\sub[H_\kappa]^\omega$ contains a club, then $S$ is spread out.
\end{obs}

\begin{proof}
Let $f:[H_\kappa]^{{<}\omega}\To H_\kappa$ be such that every $x\in C$ is closed under $f$, that is, $f``[x]^{{<}\omega}\sub x$. Let $\theta$ be sufficiently large that $[H_\kappa]^\omega\in H_\theta$. I claim that this $\theta$, together with the function $f$, witnesses that $S$ is weakly spread out, which implies that $S$ is spread out by Observation \ref{obs:WeaklySpreadOutImpliesSpreadOut}. To see this, suppose $H_\theta\sub L_\tau^A\models\ZFCm$ and $S,f\in L_\tau[A]$. Let $X\sub L_\tau[A]$ be countable, $N|X\prec L_\tau^A$, and $f\in X$ (and $N|X$ is full). Since $f\in X$, $X$ is closed under $f$, and hence, so is $X\cap H_\kappa$. Thus, $X\cap H_\kappa\in S$ (so we can choose $Y=X$ in the notation of Definition \ref{def:(ininifty-)subcompleteness}).
\end{proof}

Finally, an elementary argument shows that in the situation of Definition \ref{def:(ininifty-)subcompleteness}, necessarily, $\sigma\rest(2^\omega)^{\bN}=\sigma'\rest(2^\omega)^\bN$, see \cite[Obs.~4.2]{Fuchs:ATP}, or \cite[Proof of Lemma 3.22]{Fuchs:HierarchiesOfRA}. That argument, adapted to the present context, has the following consequence. Notice the parallel to Example \ref{example:ProperSRPtrivial}.

\begin{obs}
\label{obs:SCSRPtrivialuptocontinuum}
Let $\kappa\le 2^\omega$ be an uncountable regular cardinal, and let $S\sub[H_\kappa]^\omega$. Then $S$ is spread out iff $S$ contains a club. Hence, $\infSC$-$\SRP(\kappa)$ holds trivially.
\end{obs}

\begin{proof}
Suppose $S\sub[H_\kappa]^\omega$ is spread out. Let $\theta$ be sufficiently large, and let $N=L_\tau^A\models\ZFCm$ with $H_\theta\sub L_\tau[A]$, and suppose $\tau$ is an $L[A]$-cardinal. Let $\tau'=(\tau^+)^{L[A]}$. Let $N'=L_{\tau'}^A$. It is then easy to see that whenever $X$ is countable and $N'|X\prec N'$, then, letting $\tilde{X}=X\cap L_\tau[A]$, $N|\tilde{X}\prec N$ and $N|\tilde{X}$ is full. Now let
\[C=\{X\in[L_{\tau'}^\omega]\st N'|X\prec N\}. \]
I claim that $\bC=C\projectdown[H_\kappa]^\omega\sub S$. To see this, let $X\in\bC$. Then $X=Y\cap H_\kappa$, for some $Y\in C$. Let $\tilde{Y}=Y\cap L_\tau[A]$, so $N|Y\prec N$ is full. Since $S$ is spread out, there is a $Z$ such that $N|Z\prec N$ so that $N|Y$ is isomorphic to $N|Z$ and $Z\cap H_\kappa\in S$. Let $\pi:N|Y\To N|Z$ be this isomorphism. First, observe that $\pi\rest 2^\omega=\id$. To see this, first note that $\pi\rest\power(\omega)=\id$, and hence that $\power(\omega)\cap Y=\power(\omega)\cap Z$. Let $f$ be the $<_{L_\tau^A}$-least bijection between $\power(\omega)$ and $2^\omega$. Clearly, $f\in Y\cap Z$, and $\pi(f)=f$. It follows that $\pi\rest 2^\omega=\id$ and hence that $2^\omega\cap Y=2^\omega\cap Z$, because if $\alpha<2^\omega$ and $\alpha\in Y$, then letting $a=f^{-1}(\alpha)$, we have that $\pi(\alpha)=\pi(f(a))=\pi(f)(\pi(a))=f(a)=\alpha$. But then, it follows that $\pi\rest H_{2^\omega}=\id$, and hence that $H_{2^\omega}\cap Y=H_{2^\omega}\cap Z$. This is because if $a\in H_{2^\omega}\cap X$, then $a$ can be coded by a bounded subset $A$ of $2^\omega$. Then $\pi(A)=A$ codes the same object, and so, $\pi(a)$ is the set coded by $\pi(A)=A$, which is $a$. In particular, $X=Y\cap H_\kappa=Z\cap H_\kappa\in S$. Thus, $\bC=C\projectdown[H_\kappa]^\omega\sub S$, as claimed. Since $C$ is a club, $C\projectdown[H_\kappa]^\omega$ contains a club, as wished.

The converse is trivial, by Observation \ref{obs:ContainingAClubImpliesBeingSpreadOut}.

This proves the equivalence claimed, and of course, it is easy to construct an continuous $\in$-chain of length $\omega_1$ through $S$ if $S$ contains a club, so that $\infSC$-$\SRP(\kappa)$ holds.
\end{proof}

In ending this section, for completeness, let me mention an obvious modification to the concept of being spread out that corresponds to subcomplete (not $\infty$-subcomplete) projective stationary, as follows.

\begin{defn}
\label{def:FullySpreadOut}
Let $\kappa$ be an uncountable regular cardinal. A stationary set $S\sub[H_\kappa]^\omega$ is \emph{fully spread out} if for all sufficiently large $\theta$, whenever $\tau$, $A$ are such that $H_\theta\sub L_\tau^A\models\ZFC$, $S,\theta\in X\prec L_\tau^A$ is countable and full, if $a\in X$, $\lambda_0,\ldots,\lambda_n$ are such that $\lambda_n=\sup(X\cap\On)$ and for every $i<n$, $\lambda_i$ is a regular cardinal in the interval $(2^{2^{{<}\kappa}},\lambda_n)$, then there exist $\pi$, $Y$ such that $N|Y\prec N$ and $\pi:\kla{X,\in\cap X^2}\To\kla{Y,\in\cap Y^2}$ is an isomorphism such that $\pi(a)=a$ and for $i\le n$, $\sup(X\cap\lambda_i)=\sup(Y\cap\lambda_i)$, and $Y\cap H_\kappa\in S$.
\end{defn}

A repeat of the proof of Theorem \ref{thm:SpreadOutIffinfSCProjectiveStationary} shows that being fully spread captures subcomplete projective stationarity:

\begin{thm}
\label{thm:FullySpreadOutIffSCProjectiveStationary}
Let $\kappa$ be an uncountable regular cardinal, and let $S\sub[H_\kappa]^\omega$. Then $S$ is fully spread out iff $S$ is \SC-projective stationary.
\end{thm}

From here on out, to make things slightly less technical, I will for the most part focus on spread out sets. Everything I do would also go through for fully spread out sets, unless I explicitly say otherwise.

\section{Consequences}
\label{sec:SRPandConsequences}

In order to carry over consequences of \SRP to its subcomplete fragment, I will need to know that certain sets are not only projective stationary, but in fact spread out. Proving that a set is spread out is generally not an easy task, but fortunately, all I need will follow from one technical lemma, which I will prove in the following subsection. In later subsections, I will use this in order to derive consequences concerning Friedman's problem, the failure of square, the singular cardinal hypothesis and mutual stationarity.

\subsection{Barwise theory and a technical lemma}
\label{subsec:BarwiseTheoAndTechnicalLemma}

The proof of the main technical, but very useful lemma will emply methods of Barwise, so I will summarize what I will need very briefly. I follow Jensen's excellent presentation of this material in \cite[p.~102 ff]{Jensen2014:SubcompleteAndLForcingSingapore}. For a more detailed treatment, see Barwise \cite{ASS}, or Jensen's set of handwritten notes \cite{Jensen:AdmissibleSets}.

Recall that a structure $\kla{M,A_1,\ldots,A_n}$ is \emph{admissible} if it is transitive and satisfies $\KP$ (which I take to include the axiom of infinity), using the predicates $A_1,\ldots,A_n$. For admissible $M$, Barwise developed an infinitary logic where the infinitary formulas are (coded by) elements of $M$. Thus, infinitary conjunctions and disjunctions are allowed, as long as they are in $M$, but only finite strings of quantifiers may occur, and all predicate symbols are finitary. Let $A$ be a $\Sigma_1(M)$ set of such infinitary formulas. Thus, the set of formulas $A$ itself can be defined by a finitary first order formula over $M$ that is $\Sigma_1$. The set $A$ may well contain formulas that are not $\Sigma_1$, so it is not a $\Sigma_1$-theory in the usual model theoretic sense. The intuition is that elements of $M$ behave like finite sets in finitary logic (and hence, they are called ``$M$-finite''), and $\Sigma_1(M)$ sets behave like recursively enumerable ones. The logic comes with a proof theory and a model theory whose main features are:
\begin{enumerate}[label=(\arabic*)]
\item \emph{The $M$-finiteness lemma:}
if a formula $\phi$ is provable from $A$, then there is a $u\in M$ such that $u\sub A$ and $\phi$ is provable from $u$.
\item \emph{The correctness theorem:} if there is a model $\mathfrak{A}$ with $\mathfrak{A}\models A$, then $A$ is consistent.
\item \emph{The Barwise completeness theorem:} if $M$ is \emph{countable} and $A$ is consistent, then there is a model $\mathfrak{A}$ with $\mathfrak{A}\models A$.
\end{enumerate}

\begin{defn}
\label{def:EpsilonTheory}
Let $M$ be admissible. If $A$ consists of infinitary formulas in $M$, then $A$ is a \emph{theory on $M$}. $A$ is an \emph{$\in$-theory} on $M$ if the language it is formulated in contains the symbol $\in$, a constant symbol $\unt{x}$, for every $x\in M$, and if the theory contains the extensionality axiom, as well as the \emph{basic axiom}
\[\forall y\quad(y\in\unt{x} \iff \bigvee_{z\in x} y=\unt{z})\]
for every $x\in M$.
It is a $\ZFCm$-theory on $M$ if it is an $\in$-theory on $M$ that contains the $\ZFCm$ axioms (viewed as a set of finitary formulas, which are also in $M$).
\end{defn}

If $A$ is an $\in$-theory on $M$ and $\mathfrak{A}$ is a model for $A$ whose well-founded part is transitive, then automatically, $\unt{x}^{\mathfrak{A}}=x$, which is why I won't specify the interpretation of these constants by such a model. The following property is a slight weakening of fullness, see Definition \ref{def:Full}.

\begin{defn}
\label{def:AlmostFull}
A transitive model $N$ of \ZFCm{} is \emph{almost full} if there is a model $\mathfrak{A}$ of \ZFCm{} whose well-founded part is transitive, $N$ is an element of the well-founded part of $\mathfrak{A}$ and $N$ is \emph{regular} in $\mathfrak{A}$, i.e., if $x\in N$, $f\in\mathfrak{A}$, and $f:x\To N$, then $\ran(f)\in N$.
\end{defn}

The same comments made after Definition \ref{def:Full} apply here as well. In applications, the model $N$ will be of the form $L_\tau^A$, so that no complications arise.

\begin{defn}
If $N$ is a transitive set, then I write $\alpha(N)$ for the least $\alpha>0$ such that $L_{\alpha}(N)\models\KP$.
\end{defn}

The next lemma will be used crucially in the proof of Lemma \ref{lem:OneStep}.

\begin{namedthm}{Transfer Lemma}[{\cite[p.~123, Lemma 4.5]{Jensen2014:SubcompleteAndLForcingSingapore}}]
\label{lem:ConsistencyGoesUp}
Let $\bN$ and $N$ be transitive \ZFCm-models. Let $\bN$ be almost full and $\sigma:\bN\emb{\Sigma_0} N$ be cofinal, that is, $N=\bigcup\ran(\sigma)$. Then $N$ is almost full. Further, let $\bar{\mathcal{L}}$ be a theory in an infinitary language on $L_{\alpha(\bN)}(\bN)$ that has a $\Sigma_1$-definition in $L_{\alpha(\bN)}(\bN)$ in the parameters $\bN$ and  $p_1,\ldots,p_n\in\bN$. Let $\mathcal{L}$ be the infinitary theory on $L_{\alpha(N)}(N)$ defined over $L_{\alpha(N)}(N)$ by the same $\Sigma_1$-formula, using the parameters $N$, $\sigma(p_1),\ldots,\sigma(p_n)$. If $\bar{\mathcal{L}}$ is consistent, then so is $\mathcal{L}$. I will denote $\mathcal{L}$ by $\sigma(\mathcal{L})$.
\end{namedthm}

Coming up is the technical lemma I need, a general version of Jensen's \cite[Lemma 6.3]{Jensen2014:SubcompleteAndLForcingSingapore}. The present lemma differs from Jensen's version in several respects.

First, the formulation is different. Jensen's lemma states that if $\kappa>\omega_1$ is regular and $A\sub\kappa$ is a stationary set consisting of ordinals of countable cofinality, then the usual forcing to shoot a club through $A$ with countable conditions is subcomplete. The present version of the lemma implies this, under the additional assumption that $\kappa>2^\omega$.

More importantly, the original lemma assumes that $\kappa>\omega_1$, while I need $\kappa>2^\omega$. It was first observed by Sean Cox that there is a step in Jensen's proof that seems to only go through if $\kappa>2^\omega$, and I thank him sincerely for pointing this out to me. The assumption is actually needed for the lemma, as I observe in the note following the statement of the lemma, and it is also needed for the original lemma, as was originally noticed by Hiroshi Sakai (as communicated by Corey Switzer). The assumption is used in Claim \ref{Claim:BarNinN1} of the proof.

On the other hand, while Jensen's lemma was missing a needed assumption, the its proof made an assumption that is not needed, namely that $A$ belongs to the range of the embedding $\sigma$. It will play a role later on, in Lemma \ref{lem:OmegaStep}, that this is unnecessary. In fact, the present version of the lemma does not mention $A$.

In the article containing the original lemma, Jensen works with a variant of subcompleteness \cite[Def.~on p.~114]{Jensen2014:SubcompleteAndLForcingSingapore} in which the ``suprema condition'' \ref{item:SupremumCondition} of Definition \ref{def:(ininifty-)subcompleteness} is replaced with a ``hull condition'' which implies the suprema condition. Accordingly, the proof of his lemma establishes a potentially stronger property than \ref{item:SameSupremaAbove}
and \ref{item:HeightMovesCorrectly}. But the proof of the present version of the lemma establishes that condition as well, as I point out in Remark \ref{remark:HullCondition}, after the proof.

Finally, Jensen's lemma does not mention clause \ref{item:SameBehaviorBelow}, and this clause will also be important in the aforementioned application.

I will carry the proof out in considerable detail, because it is a subtle argument in which it is easy to overlook problems (and this has happened in the past, as I expained).

\begin{lem}
\label{lem:OneStep}
Let $\kappa>2^\omega$ be a regular cardinal, $\theta>2^{2^\kappa}$ regular, $N$ a transitive model of $\ZFCm$ (in a finite language) with a definable well-ordering of its universe and $H_\theta\sub N$, $\sigma:\bN\prec N$, where $\bN$ is countable and full and $\kappa\in\ran(\sigma)$. Let $\eta\in\kappa\cap\ran(\sigma)$ be such that $\eta^\omega<\kappa$. Let $\bar{\kappa},\bar{\eta}$ be the preimages of $\kappa,\eta$ under $\sigma$, respectively. Let $\seq{\blambda_i}{i<n}$ be regular cardinals in $\bN$, each greater than $\bkappa$. Let $\bar{a}$ be some element of $\bN$.

Then there is an $\omega$-club of ordinals $\kappa_0<\kappa$ (i.e., the set of such $\kappa_0$ is unbounded in $\kappa$ and closed under limits of countable cofinality) for which there is an embedding $\sigma':\bN\prec N$ with the following properties:
\begin{enumerate}[label=\textnormal{(\alph*)}]
  \item
  \label{item:MoveBasicParametersTheSameWay}
  Letting $p=\{\bar{a},\bar{\kappa},\bar{\eta},\bar{\lambda}_0,\ldots,\bar{\lambda}_{n-1}\}$, we have that $\sigma\rest p=\sigma'\rest p$.
  \item
  \label{item:SameSupremaAbove}
  For $i<n$, $\sup\sigma``\blambda_i=\sup\sigma'``\blambda_i$.
  \item
  \label{item:SameBehaviorBelow}
  $\sigma\rest\bar{\eta}=\sigma'\rest\bar{\eta}$.
  \item
  \label{item:CorrectSupremum}
  $\sup\sigma'``\bkappa=\kappa_0$.
  \item
  \label{item:HeightMovesCorrectly}
  $\sup\sigma``(\On\cap\bN)=\sup\sigma'``(\On\cap\bN)$.
\end{enumerate}
\end{lem}

\noindent\emph{Note:}
\begin{enumerate}
\item Instead of a single $\bar{a}$, one can choose finitely many members of $\bN$, say $\bar{a}_0,\ldots,\bar{a}_{n-1}$, and find a $\sigma'$ as in the lemma that moves each of these elements the same way $\sigma$ does, because one can apply the lemma to the sequence $\kla{\bar{a}_0,\ldots,\bar{a}_{n-1}}$.
\item The assumption that $\kappa>2^\omega$ is necessary, because if $\sigma':\bN\prec N$, then $\sigma'\rest((2^\omega)^\bN+1)=\sigma\rest((2^\omega)^{\bN}+1)$. Clearly, $\sigma((2^\omega)^\bN)=2^\omega=\sigma'((2^\omega)^{\bN})$. And if $f$ is the $N$-least bijection from $\power(\omega)$ to $2^\omega$, and $\barf$ is the $\bN$-least bijection from $(\power(\omega))^{\bN}$ to $(2^\omega)^{\bN}$, then $\sigma(\barf)=f=\sigma'(\barf)$. It follows that for any ordinal $\xi<(2^\omega)^\bN$, letting $x=\barf^{-1}(\xi)$, then $\sigma(\xi)=\sigma(\barf(x))=\sigma(\barf)(\sigma(x))=\sigma'(\barf)(\sigma'(x))=\sigma'(\barf(x))=\sigma'(\xi)$.
\end{enumerate}

\noindent{\emph{Proof.}}
Let me define $a=\sigma(\bar{a})$, $\lambda_i=\sigma(\blambda_i)$, for $i<n$, and $\eta=\sigma(\baeta)$.

Since $N$ has a definable well-order of its universe, for every subset $X$ of the universe of $N$, there is a minimal (with respect to inclusion) subset $Y$ of the universe of $N$ that contains $X$ such that $N|Y\prec N$. I denote this $Y$ by $\Hull^N(X)$, and I will use this notation for other models that have a definable well-order as well. Clearly, the set
\[C=\{\alpha<\kappa\st\Hull^N(\alpha\cup\ran(\sigma))\cap\kappa=\alpha\}\]
is club in $\kappa$. I claim that whenever $\kappa_0\in C$ has countable cofinality, then there is a $\sigma'$ as described, proving the lemma.

For the purpose of the proof, let me define that for transitive models $\bM$ and $M$ of $\ZFCm$, an embedding $j:\bM\prec M$ is \emph{cofinal} if for every $x\in M$ there is a $y\in\bM$ such that $x\in j(y)$. If $\bdelta$ is a cardinal in $\bM$, then $j:\bM\prec M$ is $\bdelta$-cofinal if for every $x\in M$ there is a $y\in\bM$ of $\bM$-cardinality less than $\bdelta$ with $x\in j(y)$. I will use the following facts several times throughout the proof.
\begin{enumerate}[label=(\Alph*)]
\item
\label{item:HullsAndCofinalEmbeddings}
If $j:\bM\prec M$, $\bdelta$ is a cardinal in $\bM$, $\delta=j(\bdelta)$ and $M=\Hull^M(\ran(j)\cup\delta)$, then $j$ is $(\bdelta^+)^\bM$-cofinal.
\item
\label{item:ContinuityOfCofinalEmbeddings}
If $j:\bM\prec M$ is $\bdelta$-cofinal, then $j$ is continuous at every $\blambda\in\bM$ with $\cf^\bM(\blambda)\ge\bdelta$, that is, $j(\blambda)=\sup j``\blambda$.
\end{enumerate}

\begin{proof}[Proof of \ref{item:HullsAndCofinalEmbeddings} \& \ref{item:ContinuityOfCofinalEmbeddings}.]
To see \ref{item:HullsAndCofinalEmbeddings}, let $x\in M$ be given.
By assumption, $x$ is definable in $M$ from some ordinal $\alpha<\delta$ and some elements $a_0,\ldots,a_{n-1}$ of $\ran(j)$. Let's say $x$ is the unique $z$ such that $M\models\phi(z,\alpha,\vec{a})$. Let $\bar{a}_0,\ldots,\bar{a}_{n-1}$ be the preimages of $a_0,\ldots,a_{n-1}$ under $j$, and consider the function $\barf:\bdelta\To\bM$ defined in $\bM$ by letting $\barf(\balpha)$ be the unique $z$ such that $\phi(z,\balpha,\vec{\bar{a}})$ holds, if such a $z$ exists, and let $\barf(\balpha)=0$ otherwise. Clearly, if we set $y=\ran(\barf)$, then $x\in j(y)$, and $y$ has $\bM$-cardinality at most $\bdelta$.

To see \ref{item:ContinuityOfCofinalEmbeddings}, fix a $\blambda$ as stated. To show that $j(\blambda)=\sup j``\blambda$,
note that the right hand side is obviously less than or equal to the left hand side. For the converse, suppose $\alpha<j(\blambda)$. By $\bdelta$-cofinality, let $y\in\bN$ have $\bN$-cardinality less than $\bdelta$, with $\alpha\in j(y)$. We may assume that $y\sub\blambda$ (by intersecting it with $\blambda$ if necessary). But then $y$ is bounded in $\blambda$, since the $\bN$-cofinality of $\blambda$ is at least $\bdelta$, say $y$ is bounded by $\xi<\blambda$. Then $\alpha\in j(y)\sub j(\xi)<\sup j``\blambda$.
\end{proof}

Now, let me fix $\kappa_0\in C$ with $\cf(\kappa_0)=\omega$, let $N_0$ be the transitive collapse of $\Hull^N(\kappa_0\cup\ran(\sigma))$, and let $k_0$ be the inverse of the collapse. We have:
\[k_0:N_0\prec N,\ \crit(k_0)=\kappa_0,\ k_0(\kappa_0)=\kappa.\]
Note that since $\eta,2^\omega\in\ran(\sigma)\cap\kappa$, it follows that $\kappa_0>\eta,2^\omega$.

\tikzset{paint/.style={ draw=#1!50!black},
    decorate with/.style=
    {decorate,decoration={shape backgrounds,shape=#1,shape size=1mm}}}

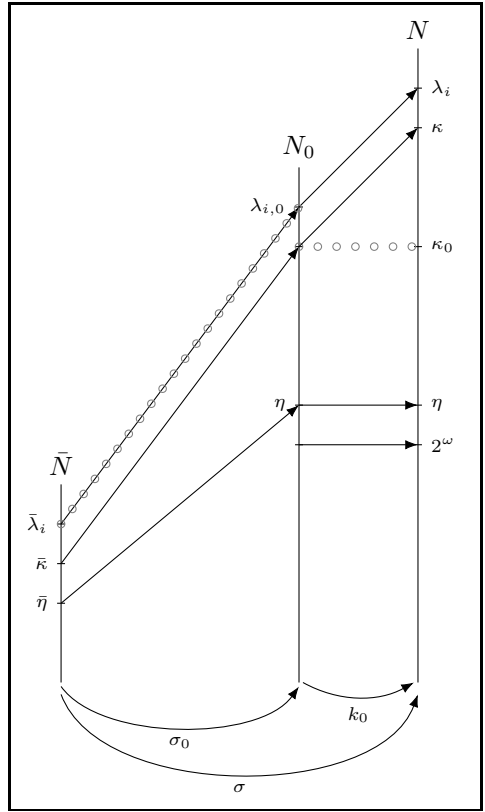
\begin{wrapfigure}[26]{r}{.5\textwidth}
\fbox{

\begin{tikzpicture}[scale=.5\textwidth/12cm,>=Latex]
\draw (1,0) -- (1,5) node[anchor=south] {$\bar{N}$};
\draw (1.1,2) -- (.9,2) node[anchor=east] {$\scriptstyle{\bar{\eta}}$};
\draw (1.1,3) -- (.9,3) node[anchor=east] {$\scriptstyle{\bar{\kappa}}$};
\draw (1.1,4) -- (.9,4) node[anchor=east] {$\scriptstyle{\bar{\lambda}_i}$};
\draw (10,0) -- (10,16) node[anchor=south] {$N$};
\draw (9.9,6) -- (10.1,6) node[anchor=west] {$\scriptstyle{2^\omega}$};
\draw (9.9,7) -- (10.1,7) node[anchor=west] {$\scriptstyle{\eta}$};
\draw (9.9,11) -- (10.1,11) node[anchor=west] {$\scriptstyle{\kappa_0}$};
\draw (9.9,14) -- (10.1,14) node[anchor=west] {$\scriptstyle{\kappa}$};
\draw (9.9,15) -- (10.1,15) node[anchor=west] {$\scriptstyle{\lambda_i}$};
\draw (7,0) -- (7,13) node[anchor=south] {$N_0$};
\draw (7.1,6) --(6.9,6);
\draw (7.1,7) -- (6.9,7) node[anchor=east] {$\scriptstyle{\eta}$};
\draw (7.1,11) -- (6.9,11);
\draw (7.1,12) -- (6.9,12) node[anchor=east] {$\scriptstyle{\lambda_{i,0}}$};
\draw[->] (7,6) -- (10,6);
\draw[->] (7,7) -- (10,7);
\draw[decorate with=circle, paint=white, ->] (7,11) -- (10,11);
\draw[->] (7,11) -- (10,14);
\draw[->] (7,12) -- (10,15);
\draw[->] (7.1,0) .. controls (8,-.5) and (9,-.5) .. (9.9,0)
    node[midway, sloped, below] {$\scriptstyle{k_0}$};
\draw[->] (1,2) -- (7,7);
\draw[->] (1,3) -- (7,11);
\draw[->] (1,4) -- (7,12);
\draw[decorate with=circle, paint=white, ->] (1,4) -- (7,12);


\draw[->] (1,-.3) .. controls (2,-3) and (9,-3) .. (10,-.3)
    node[midway, sloped, below] {$\scriptstyle{\sigma}$};


\draw[->] (1,-.1) .. controls (2,-1.5) and (6,-1.5) .. (7,-.1)
    node[midway, sloped, below] {$\scriptstyle{\sigma_0}$};






\end{tikzpicture} }
\caption{The first interpolation.}
\label{fig:FirstInterpolation}
\end{wrapfigure}
Since $\ran(\sigma)\sub\ran(k_0)$, there is an elementary embedding
\begin{enumerate}
\item[] $\sigma_0:\bN\prec N_0$
\end{enumerate}
defined by $\sigma_0=k_0^{-1}\compose\sigma$. Then we have
\begin{itemize}
\item $\sigma_0:\bN\prec N_0$,
\item $\sigma_0(\bkappa)=\kappa_0$,
\item $k_0\compose\sigma_0=\sigma$.
\end{itemize}%
Let me define $a_0=\sigma_0(\bar{a})$, $\eta_0=\sigma_0(\bar{\eta})$ and $\lambda_{i,0}=\sigma_0(\blambda_i)$, for $i<n$.

\hspace*{\parindent} By \ref{item:HullsAndCofinalEmbeddings}, it follows that

\begin{NumberedClaim}
\label{claim:Sigma0isCofinal}
$\sigma_0:\bN\prec N_0$ is a $(\bkappa^+)^\bN$-cofinal embedding.
\end{NumberedClaim}

Figure \ref{fig:FirstInterpolation} on the right summarizes the situation so far. The circles indicate the cofinal image, that is, the function $\alpha\mapsto\sup f``\alpha$. An arrow with superimposed circles indicates that the function at hand is continuous at this point, that is, that the point is mapped to its cofinal image by the function.

Next, I will define an intermediate model in between $\bN$ and $N_0$. The construction is similar to forming an extender ultrapower of $\bN$ by an extender derived from $\sigma_0\rest\bkappa$.

More precisely, $N_1$ is the Mostowski collapse of the set \[H=\{\sigma_0(\barf)(\alpha)\st\exists\beta<\bkappa\quad \barf:\beta\To\bN,\ \barf\in\bN,\ \alpha<\sigma_0(\beta)\}\]
and, letting $k_1$ be the inverse of the collapsing isomorphism, $\sigma_1=k_1^{-1}\compose\sigma_0$. It is not hard to check that $H\prec N_0$, so that this makes sense. Then:
\[\sigma_1:\bN\prec N_1,\ k_1:N_1\prec N_0,\ \sigma_0=k_1\compose\sigma_1.\]
Let me set $a_1=\sigma_1(\bar{a})$, $\eta_1=\sigma_1(\bar{\eta})$, $\kappa_1=\sigma_1(\bkappa)$ and $\lambda_{i,1}=\sigma_1(\blambda_i)$, for $i<n$. The situation is illustrated in Figure \ref{fig:SecondInterpolation}

It is easy to see from the definition of $N_1$ that
\begin{NumberedClaim}
$\sigma_1:\bN\prec N_1$ is $\bkappa$-cofinal.
\end{NumberedClaim}

Namely, given $y\in N_1$, it follows that $k_1(y)$ is of the form $\sigma_0(\barf)(\alpha)$, for some function $\barf:\beta\To\bN$ in $\bN$, where $\beta<\bkappa$ and $\alpha<\sigma_0(\beta)$. Thus, letting $x=\ran(\barf)$, we have that $x$ has cardinality less than $\bkappa$ in $\bN$, and  $k_1(y)\in\sigma_0(x)=k_1(\sigma_1(x))$, and so, pulling back via $k_1^{-1}$, we have that $y\in\sigma_1(x)$.

By \ref{item:ContinuityOfCofinalEmbeddings}, we have:

\begin{NumberedClaim}
\label{claim:Continuity}
For any $\blambda\in\bN$ of $\bN$-cofinality at least $\bkappa$, it follows that $\sup\sigma_1``\blambda=\sigma_1(\blambda)$.
\end{NumberedClaim}

It is also clear that

\begin{NumberedClaim}
$k_1\rest\kappa_1=\id$. As a result, $\sigma_1\rest\bkappa=\sigma_0\rest\bkappa$.
\end{NumberedClaim}

This is because $\kappa_1\sub H$: suppose $\alpha<\kappa_1$. Then by \ref{claim:Continuity}, $\alpha<\sigma_1(\bkappa)=\sup\sigma_1``\bkappa$. So let $\beta<\bkappa$ be such that $\alpha<\sigma_1(\beta)\le\sigma_0(\beta)$. Then $\alpha=\sigma_0(\id\rest\beta)(\alpha)\in H$.

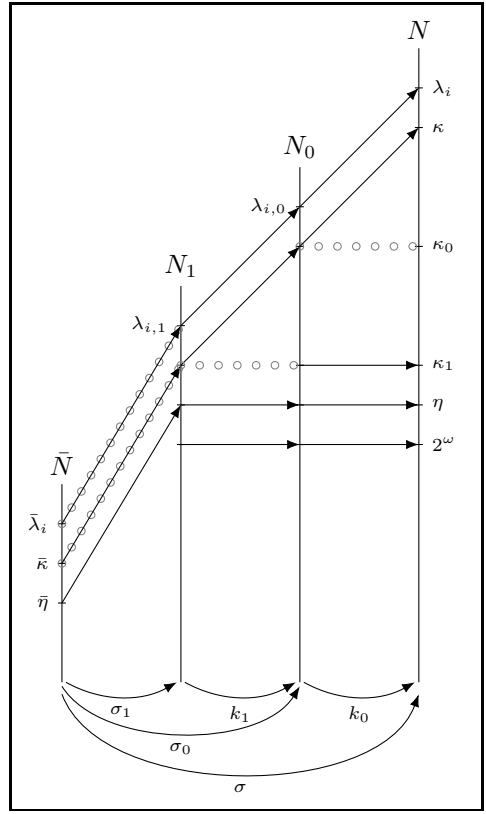
\begin{wrapfigure}[28]{r}{0pt}
\fbox{\begin{tikzpicture}[scale=.5\textwidth/12cm,>=Latex]
\draw (1,0) -- (1,5) node[anchor=south] {$\bar{N}$};
\draw (1.1,2) -- (.9,2) node[anchor=east] {$\scriptstyle{\bar{\eta}}$};
\draw (1.1,3) -- (.9,3) node[anchor=east] {$\scriptstyle{\bar{\kappa}}$};
\draw (1.1,4) -- (.9,4) node[anchor=east] {$\scriptstyle{\bar{\lambda}_i}$};
\draw (10,0) -- (10,16) node[anchor=south] {$N$};
\draw (9.9,6) -- (10.1,6) node[anchor=west] {$\scriptstyle{2^\omega}$};
\draw (9.9,7) -- (10.1,7) node[anchor=west] {$\scriptstyle{\eta}$};
\draw (9.9,8) -- (10.1,8) node[anchor=west] {$\scriptstyle{\kappa_1}$};
\draw (9.9,11) -- (10.1,11) node[anchor=west] {$\scriptstyle{\kappa_0}$};
\draw (9.9,14) -- (10.1,14) node[anchor=west] {$\scriptstyle{\kappa}$};
\draw (9.9,15) -- (10.1,15) node[anchor=west] {$\scriptstyle{\lambda_i}$};
\draw (7,0) -- (7,13) node[anchor=south] {$N_0$};
\draw (7.1,6) --(6.9,6);
\draw (7.1,7) -- (6.9,7);
\draw (7.1,8) -- (6.9,8);
\draw (7.1,11) -- (6.9,11);
\draw (7.1,12) -- (6.9,12) node[anchor=east] {$\scriptstyle{\lambda_{i,0}}$};
\draw[->] (7,6) -- (10,6);
\draw[->] (7,7) -- (10,7);
\draw[decorate with=circle, paint=white, ->] (7,11) -- (10,11);
\draw[->] (7,8) -- (10,8);
\draw[->] (7,11) -- (10,14);
\draw[->] (7,12) -- (10,15);
\draw[->] (7.1,0) .. controls (8,-.5) and (9,-.5) .. (9.9,0)
    node[midway, sloped, below] {$\scriptstyle{k_0}$};


\draw[->] (1,-.3) .. controls (2,-3) and (9,-3) .. (10,-.3)
    node[midway, sloped, below] {$\scriptstyle{\sigma}$};

\draw (4,0) -- (4,10) node[anchor=south] {$N_1$};
\draw (4.1,6) -- (3.9,6);
\draw (4.1,7) -- (3.9,7);
\draw (4.1,8) -- (3.9,8);
\draw (4.1,9) -- (3.9,9) node[anchor=east] {$\scriptstyle{\lambda_{i,1}}$};

\draw[->] (1,-.1) .. controls (2,-1.7) and (6,-1.7) .. (7,-.1)
    node[midway, sloped, below] {$\scriptstyle{\sigma_0}$};

\draw[->] (1,2) -- (4,7);
\draw[decorate with=circle, paint=white, ->] (1,3) -- (4,8);
\draw[->] (1,3) -- (4,8);
\draw[decorate with=circle, paint=white, ->] (1,4) -- (4,9);
\draw[->] (1,4) -- (4,9);

\draw[->] (4,6) -- (7,6);
\draw[->] (4,7) -- (7,7);
\draw[->] (4,8) -- (7,11);
\draw[->] (4,9) -- (7,12);
\draw[decorate with=circle, paint=white, ->] (4,8) -- (7,8);

\draw[->] (1.1,0) .. controls (2,-.5) and (3,-.5) .. (3.9,0)
    node[midway, sloped, below] {$\scriptstyle{\sigma_1}$};
\draw[->] (4.1,0) .. controls (5,-.5) and (6,-.5) .. (6.9,0)
    node[midway, sloped, below] {$\scriptstyle{k_1}$};
\end{tikzpicture} }
\caption{The second interpolation.}
\label{fig:SecondInterpolation}
\end{wrapfigure}
\smallskip


The next step in the proof is crucial. It is where the assumption that $\kappa>2^\omega$ is used, and it is the verification of the next claim that is missing in Jensen's \cite[Lemma 6.3]{Jensen2014:SubcompleteAndLForcingSingapore}.

\begin{NumberedClaim}
\label{Claim:BarNinN1}
$\bN\in N_1$.
\end{NumberedClaim}

To see this, let us view $\bN$ as a subset of $\omega$ temporarily; it can easily be coded that way.
Note that by elementarity, $\power(\omega)\in N$, and also $\power(\omega)\in N_0$. The cardinality of $\power(\omega)$ is $2^\omega$, which is less than $\kappa$, by assumption. Thus, since $\kappa_0\in C$, $2^\omega<\kappa_0$ as well. Since $\crit(k_0)=\kappa_0$, it follows that $(2^\omega)^{N_0}=2^\omega$. By elementarity, there is in $N_0$ a bijection $g:2^\omega\To\power(\omega)$. Since $\crit(k_0)=\kappa_0>2^\omega$, it follows that $k_0(g)=g$, and hence that $g$ is actually a bijection between $2^\omega$ and $\power(\omega)$. Thus, $\power(\omega)\sub N_0$, since $2^\omega\sub\kappa_0\sub N_0$. Moreover, $\power(\omega)\sub H$. This is because if we let $\bar{c}=(2^\omega)^\bN$, then $\bar{c}<\bkappa$, and so, letting $\bar{g}:\bar{c}\To\power(\omega)^\bN$ be a bijection, it follows that $\sigma_0(\bar{g}):\sigma_0(\bar{c})\To\power(\omega)$ is a bijection and for every $\xi<\sigma_0(\bar{c})$, $\sigma_0(\bar{g})(\xi)\in H$. Every subset of $\omega$ is of this form. Thus, $\bN\in H$, and so, $\bN=k_1^{-1}(\bN)\in N_1$.

\bigskip

Notice that clause \ref{item:SameBehaviorBelow} of the lemma can be equivalently expressed as $\sigma``\baeta=\sigma'``\baeta$. This is why the following point will be relevant.

\begin{NumberedClaim}
\label{Claim:PointwiseImageOfEtaBarInN1}
$\sigma_0``\bar{\eta}=\sigma_1``\bar{\eta}\in N_1$.
\end{NumberedClaim}

The reasoning is much like the argument for the previous claim. First, though, since $k_1\rest\kappa_1=\id$, it follows that $\sigma_0``\bar{\eta}=\sigma_1``\bar{\eta}$, as $\bar{\eta}<\bkappa$, and it also follows that $\eta=\eta_0=\eta_1$, since $\kappa_1\le k_0(\kappa_1)=\kappa_0$ and $k_0\rest\kappa_0=\id$.

Clearly, $\eta<\kappa_0$, and since $\eta^\omega<\kappa$, it also follows that $\eta^\omega<\kappa_0$, as $\kappa_0\in C$. It follows as before that $[\eta]^\omega\sub\Hull^N(\ran(\sigma)\cup\kappa_0)$, and since $\crit(k_0)=\kappa_0$, it follows that $[\eta]^\omega\sub N_0$. Further, since in $\bN$, $\bar{\eta}^\omega<\bkappa$, it follows as before that $[\eta]^\omega\sub H$. Hence, $k_1^{-1}``[\eta]^\omega=[\eta]^\omega\sub N_1$. In particular, $\sigma_1``\baeta\in N_1$.

\bigskip

Now let $\alpha(N_1)$ be the least $\beta>0$ such that $L_\beta(N_1)$ is admissible, and let $N_1^+=L_{\alpha(N_1)}(N_1)$. By Claims \ref{Claim:BarNinN1} and \ref{Claim:PointwiseImageOfEtaBarInN1}, $\bar{N}$ and $\sigma_0``\baeta$ are elements of $N_1$, which allows us to define the \ZFCm-theory $\mathcal{L}_1$ on $N_1^+$ that has an extra constant symbol $\dot{\sigma}$ and the following additional axioms:
\begin{itemize}
  \item $\dot{\sigma}:\unt{\bN}\prec\unt{N_1}$ $\unt{\bkappa}$-cofinally.
  \item $\dot{\sigma}(\unt{\bkappa},\unt{\bar{a}},\unt{\bar{\eta}},\unt{\blambda_0},\ldots,\unt{\blambda_{n-1}})=
      \unt{\kappa_1},\unt{a_1},\unt{\eta_1},\unt{\lambda_{0,1}},\ldots,\unt{\lambda_{{n-1},1}}$.
  \item $\dot{\sigma}``\unt{\bar{\eta}}=\unt{\sigma_1``{\bar{\eta}}}$.
\end{itemize}
It is crucial here that $\bN,\sigma_1``\baeta\in N_1$, so that the first and third item of the theory make sense.

Clearly, this theory is consistent, since $\kla{H_\kappa,\in,\sigma_1}$ is a model, for example.

Since $\sigma_0=k_1\compose\sigma_1:\bN\prec N_0$ is cofinal, by Claim \ref{claim:Sigma0isCofinal}, it follows that $k_1:N_1\prec N_0$ is cofinal as well, and hence, we know by the Transfer Lemma \ref{lem:ConsistencyGoesUp}, which is applicable since $\bN$ is full, that the theory $\mathcal{L}_0=$``$k_1(\mathcal{L}_1)$'' on $N_0^+=L_{\alpha(N_0)}(N_0)$ is also consistent. In more detail, $\mathcal{L}_0$ is the \ZFCm-theory on $N_0^+$ with the extra constant symbol $\dot{\sigma}$ and the additional axioms
\begin{itemize}
  \item $\dot{\sigma}:\unt{\bN}\prec\unt{N_0}$ $\unt{\bkappa}$-cofinally.
  \item $\dot{\sigma}(\unt{\bkappa},\unt{\bar{a}},\unt{\bar{\eta}},\unt{\blambda_0},\ldots,\unt{\blambda_{n-1}})=
      \unt{\kappa_0},\unt{a_0},\unt{\eta_0},\unt{\lambda_{0,0}},\ldots,\unt{\lambda_{{n-1},0}}$.
  \item $\dot{\sigma}``\unt{\bar{\eta}}=\unt{\sigma_0``{\bar{\eta}}}$.
\end{itemize}
Notice here that $k_1(\bN)=\bN$, since by elementarity, $N_1$ sees that $\bN$ is coded by a real, and that real is not moved by $k_1$. And $k_1(\sigma_1``\baeta)=\sigma_1``\baeta=\sigma_0``\baeta$ since $k_1\rest\kappa_1=\id$ and $\kappa_1>\sigma_1(\baeta)=\eta$.

In the last step of the proof, I would like to use Barwise completeness, ideally to find an elementary embedding as described in $\mathcal{L}_0$. But Barwise completeness only applies to countable theories, so the idea is to put all the relevant information inside a sufficiently rich model, and take a countable elementary substructure. Thus, let $M=\kla{H_\kappa,\in,N_0,\kappa_0,\sigma_0,a_0,\eta_0,\vec{\lambda}_0}$. Let $\pi:\tilde{M}\prec M$ be such that $\tilde{M}$ is countable and transitive.
Note that $\bN$ is definable from $\sigma_0$ and $N_0$, and so are $\bkappa,\baeta$ and $\vec{\blambda}$, and so, all of these objects are in the range of $\pi$. Note further that $\pi^{-1}(\bN)=\bN$, since $\bN$ is coded by a real number in $M$, and moreover, $\pi\rest\bN=\id$. I will write $\tilde{x}$ for $\pi^{-1}(x)$ when $x\in\ran(\pi)$. So $\tsigma_0:\bN\prec\tN_0$, and it follows that $\sigma_0=\pi\compose\tsigma_0$, because for $x\in\bN$, $\sigma_0(x)=\pi(\tsigma_0)(x)=\pi(\tsigma_0)(\pi(x))=\pi(\tsigma_0(x))$.

Clearly, $\mathcal{L}_0$ is definable in $M$, and hence in the range of $\pi$. Its preimage is $\tilde{\mathcal{L}}_0=\pi^{-1}(\mathcal{L}_0)$. $\tilde{\mathcal{L}}_0$ is then a consistent language on the structure $\tilde{N}^+=\pi^{-1}(N_0^+)$. Since this structure is countable (in $\V$), it has a model, say $\mathcal{A}$, whose well-founded part may be chosen to be transitive. Let $\tsigma'=\dot{\sigma}^{\mathcal{A}}$. Note that we do not know that $\tsigma'\in\tM$. Set
\[\sigma'=k_0\compose\pi\compose\tsigma'.\]
I claim that $\sigma'$ has the desired properties. To see this, it will be useful to write out what the axioms in $\tilde{\mathcal{L}}$ express:
\begin{itemize}
  \item the basic axioms and \ZFCm.
  \item $\dot{\sigma}:\unt{\bN}\prec\unt{\tilde{N}_0}$ $\unt{\bkappa}$-cofinally.
  \item $\dot{\sigma}(\unt{\bkappa},\unt{\bar{a}},\unt{\bar{\eta}},\unt{\blambda_0},\ldots,\unt{\blambda_{n-1}})=
      \unt{\tilde{\kappa}_0},\unt{\tilde{a}_0},\unt{\tilde{\eta}_0},\unt{\tilde{\lambda}_{0,0}},\ldots,\unt{\tilde{\lambda}_{{n-1},0}}$.
  \item $\dot{\sigma}``\unt{\bar{\eta}}=\unt{\tilde{\sigma}_0``{\bar{\eta}}}$.
\end{itemize}
Since $\mathcal{A}$ is a model of this theory, we have that
\begin{itemize}
\item $\tsigma':\bN\prec\tilde{N}_0$ is $\bkappa$-cofinal,
\item $\tsigma'(\bkappa,\bar{a},\bar{\eta},\blambda_0,\ldots,\blambda_{n-1})=
\tilde{\kappa}_0,\tilde{a}_0,\tilde{\eta}_0,\tilde{\lambda}_{0,0},\ldots,\tilde{\lambda}_{{n-1},0}$, and
\item $(\tsigma')``\baeta=\tsigma_0``\baeta$.
\end{itemize}
Composing with $\pi$, and writing $\hat{\sigma}=\pi\compose\tsigma'$, this translates to:
\begin{itemize}
\item $\hat{\sigma}:\bN\prec N_0$,
\item $\hat{\sigma}(\bkappa,\bar{a},\bar{\eta},\blambda_0,\ldots,\blambda_{n-1})=
\kappa_0,a_0,\eta_0,\lambda_{0,0},\ldots,\lambda_{{n-1},0}$, and
\item $\hat{\sigma}``\baeta=\sigma_0``\baeta$.
\end{itemize}
Remembering that $k_0\rest\eta_0=\id$, composing with $k_0$ results in:

\begin{itemize}
\item $\sigma':\bN\prec N$,
\item $\sigma'(\bkappa,\bar{a},\bar{\eta},\blambda_0,\ldots,\blambda_{n-1})=
\kappa,a,\eta,\lambda_0,\ldots,\lambda_{{n-1}}$, and
\item $(\sigma')``\baeta=\sigma_0``\baeta$.
\end{itemize}

In particular, clauses \ref{item:MoveBasicParametersTheSameWay} and \ref{item:SameBehaviorBelow} of the lemma are satisfied. For the remaining clauses, it will be useful to analyze $\hat{\sigma}$ in more detail.
The fact that $\tsigma':\bN\prec\tilde{N}_0$ is $\bkappa$-cofinal gives some more information about this embedding. It is this argument in which I will make use of the fact that $\kappa_0$ has countable cofinality.

\begin{NumberedClaim}
\label{Claim:CofinalityOfSigmaHat}
$\hat{\sigma}:\bN\prec N_0$ is $\bkappa$-cofinal.
\end{NumberedClaim}

To see this, let $a\in N_0$. Since $\sigma_0:\bN\prec N_0$ is $\bkappa^+$-cofinal, there is a $b\in\bN$ of $\bN$-cardinality $\bkappa$ with $a\in\sigma_0(b)$. Let $f:\bkappa\To b$ be surjective, $f\in\bN$, and let $\xi<\kappa_0$ be such that $a=\sigma_0(f)(\xi)$. Since $M$ sees that $\kappa_0$ has countable cofinality, $\tM$ sees that $\tkappa_0$ has countable cofinality, and it follows that $\pi``\tkappa_0$ is unbounded in $\kappa_0$. So let $\beta<\tkappa_0$ be such that $\xi<\pi(\beta)$. Let $b'=\tsigma_0(f)``\beta$, so that $a\in\pi(\beta')$, since $\pi(b')=\sigma_0(f)``\pi(\beta)\ni\sigma_0(f)(\xi)=a$. Since $\tsigma':\bN\prec\tN_0$ is $\bkappa$-cofinal, there is a $c\in\bN$ of $\bN$-cardinality less than $\bkappa$ and such that $b'\in\tsigma'(c)$. Since the $\tN_0$-cardinality of $b'$ is less than $\tkappa_0$, we may assume that every element of $c$ has size less than $\bkappa$ in $\bN$, by shrinking $c$ if necessary. Now, since $\bkappa$ is regular in $\bN$, it follows that $\bigcup c$ has $\bN$-cardinality less than $\bkappa$, and $a\in\pi(b')\sub\pi(\tsigma'(\bigcup c))=\hat{\sigma}(\bigcup c)$.

%
%
%

\begin{NumberedClaim}
If $\cf^\bN(\blambda)\ge\bkappa$, then
$\hat{\sigma}(\blambda)=\sup\hat{\sigma}``\blambda$, and if $\cf^\bN(\blambda)>\bkappa$, then
$\sigma_0(\blambda)=\sup\sigma_0``\blambda$.
\end{NumberedClaim}

This follows from the previous claim, \ref{claim:Sigma0isCofinal} and \ref{item:ContinuityOfCofinalEmbeddings}.

It is now obvious that clause \ref{item:CorrectSupremum} holds, that is, that $\sup\sigma'``\bkappa=\kappa_0$. This is because $\sup\sigma'``\bkappa=\sup k_0``\hat{\sigma}``\bkappa=\sup k_0``\hat{\sigma}(\bkappa)=\sup k_0``\kappa_0=\kappa_0$.

The next claim shows that $\sigma'$ satisfies clause \ref{item:SameSupremaAbove} of the lemma.

\begin{NumberedClaim}
For $i<n$,
$\sup\sigma_0``\blambda_i=\sup\hat{\sigma}``\blambda_i$ and
$\sup\sigma'``\blambda_i=\sup\sigma``\blambda_i$.
\end{NumberedClaim}

The first part follows from the previous claim, because
\[\sup\sigma_0``\blambda_i=\sigma_0(\blambda_i)=\lambda_{i,0}=\hat{\sigma}(\blambda_i)=\sup\hat{\sigma}``\blambda_i.\]
The second part follows from the first, since $\sigma'=k_0\compose\hat{\sigma}$ and $\sigma=k_0\compose\sigma_0$.

Finally, let me check that clause \ref{item:HeightMovesCorrectly} is satisfied, that is, that $\sup\sigma``(\On\cap\bN)=\sup\sigma'``(\On\cap\bN)$. This is easy to see: Both $\sigma_0$ and $\hat{\sigma}$ cofinal, and hence,
\begin{ea*}
\sup\sigma``(\On\cap\bN)\quad=&\sup k_0``\sup\sigma_0``(\On\cap\bN)&=\quad\sup k_0``(\On\cap N_0)\\
=&\sup k_0``\sup\hat{\sigma}``(\On\cap\bN)&=\quad\sup\sigma'``(\On\cap\bN).
\end{ea*}
%
%
\qed

I will not use the following remark, but I would like to state and prove it anyway, since in some of Jensen's writings, he defines subcompleteness by requiring that, in the notation of Definition \ref{def:(ininifty-)subcompleteness}, the embedding $\sigma'$ satisfy the ``hull condition'' that $\Hull^N(\ran(\sigma)\cup\delta)=\Hull^N(\ran(\sigma')\cup\delta)$, where $\delta=\delta(\P)$ is the density of the forcing in question, instead of the ``suprema condition,'' that is, condition \ref{item:SupremumCondition} of that Definition.

\begin{remark}
\label{remark:HullCondition}
In the notation of the previous lemma, the embedding $\sigma'$ can be guaranteed to have the property that
\[\Hull^N(\ran(\sigma)\cup\kappa_0)=\Hull^N(\ran(\sigma')\cup\kappa_0).\]
\end{remark}

\begin{proof}
The embedding constructed in the proof has this property. To see this, let me freely use notation from the proof.
By construction, $N_0=\Hull^{N_0}(\ran(\sigma_0)\cup\kappa_0)$, and so, it suffices to show that $\Hull^{N_0}(\ran(\hat{\sigma})\cup\kappa_0)=N_0$ as well, since $\sigma$/$\sigma'$ result from composing $\sigma_0$/$\hat{\sigma}$ with $k_0$. But I showed that $\hat{\sigma}:\bN\prec N_0$ is $\bkappa$-cofinal, which immediately implies this.
\end{proof}

It is maybe worth mentioning that, still in the notation of Lemma \ref{lem:OneStep}, if $\sigma'$ has the property stated in the previous remark, and actually it is enough that $\Hull^N(\ran(\sigma)\cup\kappa)=\Hull^N(\ran(\sigma')\cup\kappa)$, which is a weaker condition, then for any $\blambda>\bkappa$ of $\bN$-cofinality greater than $\bkappa$, if $\sigma(\blambda)=\sigma'(\blambda)$, then $\sup\sigma``\blambda=\sup\sigma'``\blambda$. So this condition is a strong form of clause \ref{item:SameSupremaAbove} of the lemma. For a proof, see \cite[Fact 1.6]{Fuchs:ParametricSubcompleteness}.

\subsection{Friedman's problem, the failure of square, and \SCH}
\label{subsec:FPandSCH}

In Section \ref{sec:Gamma-ProjectiveStationarity}, I relativized the strong reflection principle to a forcing class $\Gamma$ by saying that every $\Gamma$-projective stationary set contains a continuous elementary chain of length $\omega_1$. I would now like to derive some consequences of this $\Gamma$-fragment of \SRP, and the reason why these consequences arise is that certain sets are $\Gamma$-projective stationary. In order to be able to keep track of the sets that are responsible for these consequences, it will be useful to name them.

\begin{defn}
\label{def:SRP(S)}

For an uncountable regular cardinal $\kappa$, let
\[
\mathcal{S}_{\lifting}(\kappa)=\{\lifting(A,[H_\kappa]^\omega)\cap C\st A\sub S^\kappa_\omega\ \text{is stationary in $\kappa$}\ \text{and}\ C\sub[H_\kappa]^\omega\ \text{is club}\}.
\]

Given a collection $\mathcal{S}$ of stationary subsets of $H_\theta$, the \emph{$\mathcal{S}$-fragment of the strong reflection principle}, $\SRP(\mathcal{S})$, asserts that
if $S\in\mathcal{S}$, 
then there is a continuous $\in$-chain of length $\omega_1$ through $S$.
\end{defn}

So $\mathcal{S}_{\lifting}(\kappa)$ consists of all the liftings of stationary subsets of $S^\kappa_\omega$ to $[H_\kappa]^\omega$, and their intersections with clubs; see Definition \ref{defn:ProjectionsAndLiftings}.
The reason why I isolated the class $\mathcal{S}_{\lifting}$ is that $\SRP(\mathcal{S}_{\lifting}(\kappa))$ has some interesting consequences hinted at in the title of the present subsection, and follows from $\Gamma$-\SRP, for the classes $\Gamma$ of interest. The following fact has been known for a long time.

\begin{fact}[{Feng \& Jech \cite[Example 2.2]{FengJech:ProjectiveStationarityAndSRP}}]
\label{fact:OldNews}
Let $\kappa>\omega_1$ be a regular cardinal, and let $A\sub S^\kappa_\omega$ be stationary. Then the set
\[S=\{X\in[H_\kappa]^\omega\st \sup(X\cap\kappa)\in A\}=\lifting(A,[H_\kappa]^\omega)\]
is projective stationary.
\end{fact}

By Observation \ref{obs:ProjectiveStationarityPreservedUnderIntersectionWithClubs}, this can be restated by saying that for regular $\kappa>\omega_1$, every set in $\mathcal{S}_\lifting(\kappa)$ is projective stationary. With Lemma \ref{lem:OneStep} at my disposal, I am now ready to prove the corresponding fact about spread out sets. But I do need that $\kappa>2^\omega$.

\begin{lem}
\label{lem:LiftingsAreSpreadOut}
Let $\kappa>2^\omega$ be regular. Then $\mathcal{S}_\lifting(\kappa)$ consists of spread out sets.
%
%
\end{lem}

\begin{proof}
By Observation \ref{obs:SpreadOutSetsClosedUnderIntersectionWithClubs}, it suffices to show that if $B\sub S^\kappa_\omega$ is stationary, then $S=\{X\in[H_\kappa]^\omega\st \sup(X\cap\kappa)\in B\}$ is spread out.
To this end, let $X\prec N=L_\tau^A$, $\sigma:\bN\To X$ the inverse of the collapse of $X$, $S\in X$, $\bN$ countable and full, $H_\theta\sub N$, where $\theta$ is sufficiently large. Let $a=\sigma(\bar{a})$ be fixed, and assume that $\kappa=\sigma(\bkappa)\in\ran(\sigma)$. By Lemma \ref{lem:OneStep}, there is an $\omega$-club of ordinals $\kappa_0$ less than $\kappa$ such that there is a $\sigma':\bN\prec N$ with $\sigma'(\bar{a})=a$, $\sigma'(\bkappa)=\kappa$ and $\sup\sigma'``\bkappa=\kappa_0$. Since $B\sub S^\kappa_\omega$ is stationary, there is such a $\kappa_0\in B$. Let $\sigma'$ be the corresponding embedding, and set $Y=\ran(\sigma')$. Since $Y\cap\kappa=\sigma'``\bkappa$, we have that $\sup(Y\cap\kappa)=\sup\sigma'``\bkappa=\kappa_0\in B$, so $Y\cap H_\kappa\in S$. Thus, since
$\pi=\sigma'\compose\sigma^{-1}:X\To Y$ is an isomorphism fixing $a$, $S$ is spread out.
\end{proof}

Clauses \ref{item:SameSupremaAbove} and \ref{item:HeightMovesCorrectly} of Lemma \ref{lem:OneStep} can be used to show that in the situation of the previous lemma, $\mathcal{S}_\lifting(\kappa)$ actually consists of \emph{fully} spread out sets.

The following theorem explains the import of $\mathcal{S}_\lifting(\kappa)$. Of course, only the last part mentioning $\infSC$-\SRP is new.
For the definition of $\FP{\kappa}$, see Definition \ref{def:FPkappa}.

\begin{thm}
\label{thm:infSCSRPimpliesFPkappa}
Let $\kappa>\omega_1$ be regular.
\begin{enumerate}[label=(\arabic*)]
\item
\label{item:SRPSliftingImpliesFPkappa}
$\SRP(\mathcal{S}_\lifting(\kappa))$ implies $\FP{\kappa}$.
\item
\label{item:SRPImpliesSRP(Slifting)}
$\SRP(\kappa)$ implies $\SRP(\mathcal{S}_\lifting(\kappa))$ and hence $\FP{\kappa}$.
\item
\label{item:SCSRPImpliesSRP(Slifting)}
If $\kappa>2^\omega$, then $\infSC$-$\SRP(\kappa)$ implies $\SRP(\mathcal{S}_\lifting(\kappa))$ and hence $\FP{\kappa}$.
\end{enumerate}
\end{thm}

\begin{proof}
For \ref{item:SRPSliftingImpliesFPkappa}, let $A\sub S^\kappa_\omega$ be stationary. Then the set
\[S=\{X\in[H_\kappa]^\omega\st\sup(X\cap\kappa)\in A\}\]
is in $\mathcal{S}_\lifting(\kappa)$. By $\SRP(\mathcal{S}_\lifting(\kappa))$, let $\seq{M_i}{i<\omega_1}$ be a continuous elementary chain through $S$. Define $f:\omega_1\To A$ by $f(i)=\sup M_i\cap\kappa$. Then $f$ is a normal function, verifying the instance of $\FP{\kappa}$ given by $A$.

Now \ref{item:SRPImpliesSRP(Slifting)} follows from Fact \ref{fact:OldNews} and \ref{item:SRPSliftingImpliesFPkappa}, and \ref{item:SCSRPImpliesSRP(Slifting)} 
from Lemma \ref{lem:LiftingsAreSpreadOut} and \ref{item:SRPSliftingImpliesFPkappa}.
\end{proof}

In item \ref{item:SCSRPImpliesSRP(Slifting)} of the previous theorem, $\infSC$-$\SRP(\kappa)$ can be replaced with $\SC$-$\SRP(\kappa)$, since the relevant sets are fully spread out, as pointed out before.

It is well-known that for a cardinal $\kappa$, $\FP{\kappa^+}$ implies the failure of Jensen's principle $\square_\kappa$, so as a consequence of the previous theorem, one obtains:

\begin{cor}
\label{cor:FailureOfSquare}
$\infSC$-\SRP implies that for every cardinal $\kappa\ge 2^\omega$, $\square_\kappa$ fails.
\end{cor}

Again, $\SC$-\SRP is sufficient here. The following terminology expands on \cite{Fuchs:HierarchiesOfForcingAxioms}.

\begin{defn}
\label{def:SFPkappa}
Let $\kappa$ be an uncountable regular cardinal.

The \emph{strong Friedman Property} at $\kappa$, denoted $\SFP{\kappa}$, says that for any partition $\seq{D_i}{i<\omega_1}$ of $\omega_1$ into stationary sets, and  for any sequence $\seq{S_i}{i<\omega_1}$ of stationary subsets of $S^\kappa_\omega$, there is a normal function $f:\omega_1\To\bigcup_{i<\omega_1}S_i$ such that for every $i<\omega_1$, $f``D_i\sub S_i$.
\end{defn}

This is a somewhat technical concept, but I will extract from it something more natural. The following expands on \cite[Def.~8.17]{ST3}.

\begin{defn}
\label{def:Trace}
Let $\kappa$ be a cardinal of uncountable cofinality, and let $S\sub\kappa$ be stationary. Then an ordinal $\delta<\kappa$ is a \emph{reflection point} of $S$ if $\delta$ has uncountable cofinality and $S\cap\delta$ is stationary in $\delta$. It is an $\emph{exact reflection point}$ of $S$ if $S\cap\delta$ contains a club in $\delta$.

If $\vS=\seq{S_i}{i<\lambda}$ is a sequence of stationary subsets of $\kappa$, then an ordinal $\delta<\kappa$ is a \emph{simultaneous reflection point} of $\vS$ if for every $i<\lambda$, $\delta$ is a reflection point of $S_i$. It is an \emph{exact simultaneous reflection point} of $\vS$ if it is a simultaneous reflection point of $\vS$ and $\delta\cap(\bigcup_{i<\lambda}S_i)$ contains a club in $\delta$.

The \emph{trace} of $S$, denoted $\Tr(S)$ is the set of reflection points of $S$, and similarly, the trace of $\vS$, denoted $\Tr(\vS)$, is the set of simultaneous reflection points of $\vS$. Similarly, the \emph{exact trace} of $S$, denoted $\eTr(S)$, is the set of exact reflection points of $S$, and $\eTr(\vS)$ is the set of exact reflection points of $\vS$.
\end{defn}

Clearly, $\FP{\kappa}$ implies not only that every stationary subset $A$ of $S^\kappa_\omega$ has a reflection point, but that it has an exact reflection point. $\SFP{\kappa}$ has a similar effect on $\omega_1$-sequences of stationary subsets of $S^\kappa_\omega$, as the following observation shows - note that it obviously implies \ref{item:SFPkappa} if $S=S^\kappa_\omega$, and hence each of the equivalent conditions stated. In fact, the observation shows that \ref{item:SFPkappa} is maybe a more natural version of $\SFP{\kappa}$.

\begin{obs}
\label{obs:CharacterizationOfE-versionOfSFPkappa}
Let $\kappa>\omega_1$ be regular and fix a stationary subset $S$ of $\kappa$. The following are equivalent:
\begin{enumerate}[label=\textnormal{(\alph*)}]
\item
\label{item:SFPkappa}
Whenever $\vS=\seq{S_i}{i<\omega_1}$ is a sequence stationary subsets of $S$, there are a partition $\seq{D_i}{i<\omega_1}$ of $\omega_1$ into stationary sets and a normal function $f:\omega_1\To\kappa$ such that for all $i<\omega_1$, $f``D_i\sub S_i$.
\item
\label{item:NonemptyExactTrace}
Whenever $\vS=\seq{S_i}{i<\omega_1}$ is a sequence of stationary subsets of $S$, then $\eTr(\vS)\neq\leer$.
\item
\label{item:StationaryExactTrace}
Whenever $\vS=\seq{S_i}{i<\omega_1}$ is a sequence of stationary subsets of $S$, then $\eTr(\vS)$ is stationary.
\end{enumerate}
\end{obs}

\begin{proof}
\ref{item:SFPkappa}$\implies$\ref{item:StationaryExactTrace}:
Let $\vS$ be as in \ref{item:StationaryExactTrace}, and let $C\sub\kappa$ be club. Then let $\vec{D}$, $f$ be as in \ref{item:SFPkappa} with respect to the sequence $\seq{S_i\cap C}{i<\omega_1}$. Letting $\delta=\sup\ran(f)$, it follows that $\delta\in C\cap\eTr(\vS)$.

\ref{item:StationaryExactTrace}$\implies$\ref{item:NonemptyExactTrace}: trivial.

\ref{item:NonemptyExactTrace}$\implies$\ref{item:SFPkappa}:
Let $\vS=\seq{S_i}{i<\omega_1}$ be a sequence of stationary subsets of $S^\kappa_\omega$.
Let $\vS'=\seq{S'_i}{i<\omega_1}$ be a refinement of $\vS$ into a sequence of pairwise disjoint stationary sets. Such a sequence $\vS'$ exists because $\kappa>\omega_1$, so that the nonstationary ideal on $\kappa$ is ``nowhere $\omega_2$-saturated'', see Baumgartner-Hajnal-M\'{a}t\'{e} \cite[Lemma 2.1]{BHM:WeakSaturationPropertiesOfIdeals} for details.
By \ref{item:NonemptyExactTrace}, let $\delta$ be an exact reflection point of $\vS'$. Let $C\sub(\delta\cap\bigcup_{i<\omega_1}S_i)$ be club, and let $f:\omega_1\To C$ be the monotone enumeration of $C$. For
$i<\omega_1$, let $D_i=f^{-1}``S'_i$.
Then $\seq{D_i}{i<\omega_1}$ is a partition of $\omega_1$ into stationary sets such that for every $i<\omega_1$, $f``D_i\sub S'_i\sub S_i$.
\end{proof}

The following fact is usually stated assuming some variation of $\SFP{\kappa}$ in place of my assumption, but it can be filtered through the ``exact'' reflection property of the previous observation. See Foreman-Magidor-Shelah \cite{FMS:MM1}, or Jech \cite[p.~686, proof of Theorem 37.13]{ST3}. Note that if $\delta$ is an exact reflection point of a sequence $\vS$, then $\delta$ is a reflection point of each $S_i$, but if $T$ is a stationary subset of $\kappa$ that's disjoint from $\bigcup\vS$, then $\delta$ is \emph{not} a reflection point of $T$, and this is why I refer to it as an \emph{exact} reflection point. This property is used in the proof of the following fact.

\begin{fact}
\label{fact:SFPandArithmetic}
Let $\kappa>\omega_1$ be a regular cardinal, and let $S$ be a stationary subset of $\kappa$ such that any $\omega_1$-sequence of stationary subsets of $S$ has an exact simultaneous reflection point. Then $\kappa^{\omega_1}=\kappa$.
\end{fact}

\begin{proof}
Fix a sequence $\seq{S_i}{i<\kappa}$ of pairwise disjoint stationary subsets of $S$. If $x\in[\kappa]^{\omega_1}$ and $\delta$ is an exact simultaneous reflection point for $\vS\rest x$, then $x=R_\delta=\{i<\kappa\st\delta\ \text{is a reflection point of}\ S_i\}$. Thus, $[\kappa]^{\omega_1}\sub\{R_\delta\st\delta<\kappa\}$.
\end{proof}

The following is due to Feng \& Jech \cite[Example 2.3]{FengJech:ProjectiveStationarityAndSRP}.

\begin{lem}
\label{lem:FengJechSFPpreliminary}
Let $\kappa>\omega_1$ be regular.
Let $\vD=\seq{D_i}{i<\omega_1}$ be a partition of $\omega_1$ into stationary sets which is \emph{maximal} in the sense that for every stationary $T\sub\omega_1$, there is an $i<\omega_1$ such that $D_i\cap T$ is stationary.
Let $\vS=\seq{S_i}{i<\omega_1}$ be a sequence of stationary subsets of $S^\kappa_\omega$. Then the set
\[S=\{X\in[H_\kappa]^\omega\st
\forall i<\omega_1\ (X\cap\omega_1\in D_i\to \sup(X\cap\kappa)\in S_i)\}.\]
is projective stationary.
\end{lem}

The status of the maximality assumption on $\vD$ here is interesting. First, let me note:

\begin{remark}
Maximal partitions $\vD$ of $\omega_1$ into stationary sets as in the previous lemma are easy to construct:
start with an arbitrary partition $\vD'$ of $\omega_1$ into stationary sets. 
It can be identified with the function $f':\omega_1\To\omega_1$ defined by $i\in D'_{f'(i)}$. 
Modify $f'$ to the regressive function $f:\omega_1\To\omega_1$ defined by $f(i)=f'(i)$ if $f'(i)<i$, and $f(i)=0$ otherwise. The corresponding partition $\seq{D_i}{i<\omega_1}$ with $D_i=\{\alpha\st f(\alpha)=i\}$ is then as wished: each $D_i$ is stationary, because it contains $D'_i\ohne(i+1)$, and if $T\sub\omega_1$ is stationary, then $f\rest T$ is constant on a stationary subset of $T$, say with value $i_0$, so $T\cap D_{i_0}$ is stationary.
\end{remark}

Following is the version of Feng \& Jech's Lemma \ref{lem:FengJechSFPpreliminary}, with ``spread out'' in place of ``projective stationary.'' I have to strengthen the assumption that $\kappa>\omega_1$ to $\kappa>2^\omega$, but I can drop the maximality assumption on the partition $\vD$.

\begin{lem}
\label{lem:SFPpreliminary}
Let $\kappa>2^\omega$ be regular.
Let $\vD=\seq{D_i}{i<\omega_1}$ be a partition of $\omega_1$ into stationary sets and $\vS=\seq{S_i}{i<\omega_1}$ a sequence of stationary subsets of $S^\kappa_\omega$. Let
\[S=\{X\in[H_\kappa]^\omega\st
\forall i<\omega_1\ (X\cap\omega_1\in D_i\to \sup(X\cap\kappa)\in S_i)\}.\]
Then $S$ is spread out.
\end{lem}

\begin{proof}
Let $\theta$ be sufficiently large, $H_\theta\sub L_\tau^A=N\models\ZFCm$, $X\prec N$ countable and full with $S\in X$, and fix $a\in X$. Let $i<\omega_1$ be such that $\delta=X\cap\omega_1\in D_i$. We can now use Lemma \ref{lem:OneStep} as in the proof of Lemma \ref{lem:LiftingsAreSpreadOut}, showing that there are a $Y\prec N$ and an isomorphism $\pi':X\To Y$ fixing $a$ such that $\kappa_0=\sup(Y\cap\kappa)\in S_i$. Since this isomorphism has to fix the countable ordinals, it follows that $Y\cap\omega_1=X\cap\omega_1\in S_i$, and hence that $Y\cap H_\kappa\in S$.
\end{proof}

Again, using properties \ref{item:SameSupremaAbove} and
\ref{item:HeightMovesCorrectly} of Lemma \ref{lem:OneStep}, one obtains that in the previous lemma, $S$ is \emph{fully} spread out.

Since spread out sets are also projective stationary (see Observation \ref{obs:SpreadOutImpliesProjectiveStationary}), the previous lemma shows that if $\kappa>2^\omega$, then it is not necessary to assume $\vD$ is maximal in Lemma \ref{lem:FengJechSFPpreliminary}. This seems to be new, and I am not sure how one would prove this without using something along the lines of Lemma \ref{lem:OneStep}.

\begin{thm}
\label{thm:infSCSRPimpliesSFPkappa}
Let $\kappa>2^\omega$ be regular. Then $\infSC$-$\SRP(\kappa)$ implies $\SFP{\kappa}$.
\end{thm}

\begin{proof}
Let $\vD$, $\vS$ be as in Definition \ref{def:SFPkappa}. By Lemma \ref{lem:SFPpreliminary}, the set
\[S=\{X\in[H_\kappa]^\omega\st
\forall i<\omega_1\ (X\cap\omega_1\in D_i\to \sup(X\cap\kappa)\in S_i)\}\]
is \infSC-projective stationary. By \infSC-\SRP, let $\seq{M_i}{i<\omega_1}$ be a continuous $\in$-chain through $S$. Let $C=\{\alpha<\omega_1\st M_\alpha\cap\omega_1=\alpha\}$. Clearly, $C$ is closed and unbounded in $\omega_1$. Define $f:C\To\kappa$ by $f(i)=\sup(M_i\cap\kappa)$. Then $f$ is strictly increasing and continuous in the sense that if $\alpha$ is a countable limit point of $C$, then $f(\alpha)=\sup_{\beta\in C\cap\alpha}f(\beta)$. Moreover, for $j\in C$, if $j\in D_i$, then $f(j)\in S_i$, since $j=M_j\cap\omega_1\in D_i$ and hence $\sup(M_j\cap\kappa)\in S_i$, as $M_j\in S$. So $f$ is almost like the function postulated to exist by $\SFP{\kappa}$, except that it is only defined on $C$, a club subset of $\omega_1$, rather than on all of $\omega_1$. This form of $\SFP{\kappa}$ is enough for the applications, so let me just sketch how to obtain the full version from this.

All that needs to be done is fill in the gaps.
Let $\seq{\zeta_i}{i<\omega_1}$ be the monotone enumeration of $C$. Fixing $i<\omega_1$, consider the forcing notion $\P$ consisting of all functions $h:(\zeta_i,\alpha]\To\kappa$, continuous on their domains, such that (1), $\zeta_i<\alpha<\omega_1$, (2), for all $\xi\in\dom(h)$, if $j$ is such that $\xi\in D_j$, then $h(\xi)\in S_j$, and (3), if $\zeta_i+1\in\dom(h)$, then $h(\zeta_i+1)>f(\zeta_i)$. This forcing is stationary set preserving, and for every $\alpha\in(\zeta_i,\omega_1)$, the set $D_\alpha$ of conditions in $\P$ whose domain contains $\alpha$ is dense in $\P$. The latter property is what I need here, and an argument establishing it can be found in \cite[p.~686]{ST3}. Now let $G_i$ be $M_{\zeta_{i+1}}$-generic for $\P$ (we may assume that each element of the chain is an elementary submodel of $\kla{H_\kappa,\in,\vec{D},\vec{S}}$, so that $\P$ belongs to every model in the chain), and let $g_i=\bigcup G_i$. It then follows that $\dom(g_i)=[\zeta_i+1,\zeta_{i+1})$, $g_i(\zeta_i+1)>f(\zeta_i)$, $\sup g_i``\zeta_{i+1}=\sup(M_{i+1}\cap\kappa)=f(\zeta_{i+1})$, and for all $\alpha\in(\zeta_i,\zeta_{i+1})$, if $\alpha\in D_j$, then $g_i(\alpha)\in S_j$. Hence, if we define $f'=f\cup\bigcup_{i<\omega_1}g_i$, then $f'$ is as desired.
\end{proof}

Again, $\SC$-$\SRP(\kappa)$ is sufficient in this theorem.
It has been well-known for a long time that if $\kappa>\omega_1$ is regular, then $\SRP(\kappa)$ implies the version of $\SFP{\kappa}$ in which it is assumed that the partition $\vD$ used (see Definition \ref{def:SFPkappa}) is maximal in the sense of Lemma \ref{lem:FengJechSFPpreliminary}. The previous theorem shows that already the subcomplete fragment of $\SRP(\kappa)$ implies the full $\SFP{\kappa}$ principle, provided that $\kappa>2^\omega$. In this corollary, $\SC$-\SRP is enough.

\begin{cor}
\label{cor:SCSRPimpliesCardArithmeticAboveContinuum}
$\infSC$-\SRP implies that for regular $\kappa>2^\omega$, $\kappa^{\omega_1}=\kappa$.
\end{cor}

\begin{proof}
This is by the previous theorem, Observation \ref{obs:CharacterizationOfE-versionOfSFPkappa} and Fact \ref{fact:SFPandArithmetic}.
\end{proof}

Results of \cite{CummingsMagidor:MMandWeakSquare} can be used to derive further consequences of Theorem \ref{thm:infSCSRPimpliesSFPkappa} in terms of the failure of weak square principles. I will not go into the details here, but I would like to keep track of the fragment of \SRP responsible for the latest consequences mentioned. First, one could ask:

\begin{question}
\label{question:SFPfromSRP(Liftings)}
Assume $\kappa>\omega_1$ is regular. Does $\SRP(\mathcal{S}_\lifting(\kappa))$ imply $\SFP{\kappa}$?
\end{question}

On the positive side, one can list the sets used to derive $\SFP{\kappa}$.

\begin{defn}
\label{def:Correspondences}
Given an uncountable regular cardinal $\kappa$, a partition $\vD$ of $\omega_1$ into stationary sets and an $\omega_1$-sequence $\vec{S}$ of stationary subsets of $S^\kappa_\omega$, call the pair $\kla{\vD,\vS}$ a \emph{$\kappa$-correspondence,} and define the \emph{lifting} of such a correspondence to any $X$ with $\kappa\sub X$ by
\[\lifting(\kla{\vD,\vS},[X]^\omega)=\{x\in[X]^\omega\st\forall i<\omega_1\quad(x\cap\omega_1\in D_i\To\sup(x\cap\kappa)\in S_i)\}\]
and then define the class of liftings of correspondences by letting
\begin{ea*}
\mathcal{S}_{\text{corr}}(\kappa)=\{\lifting(\kla{\vD,\vS},[H_\kappa]^\omega)\cap C&\st&\kappa>\omega_1,\ \kla{\vD,\vS}\ \text{is a $\kappa$-correspondence}\\ &&\text{and}\ C\sub[H_\kappa]^\omega\ \text{is club}\}.
\end{ea*}
\end{defn}

Clearly then, for $\kappa>2^\omega$, $\mathcal{S}_{\text{corr}}(\kappa)$ consists of spread out sets, by Lemma \ref{lem:SFPpreliminary}, and $\SRP(\mathcal{S}_{\text{corr}}(\kappa))$ implies $\SFP{\kappa}$, by the proof of Theorem \ref{thm:infSCSRPimpliesSFPkappa}. In fact, fixing one partition $\vD$ as in Definition \ref{def:Correspondences} would suffice.

\begin{cor}
\label{cor:infSCSRPimpliesSCH}
$\infSC$-$\SRP$ implies \SCH. Actually, $\SRP(\mathcal{S}_{\text{corr}}(\kappa))$, for all $\kappa>2^\omega$, suffices.
\end{cor}

\begin{proof}
We have to show that if $\lambda$ is a singular cardinal with $2^{\cf(\lambda)}<\lambda$, then $\lambda^{\cf(\lambda)}=\lambda^+$. By Silver's Theorem (see \cite[Theorem 8.13]{ST3}), it suffices to prove this in the case that $\lambda$ has countable cofinality.
In this case, we have that $\lambda>2^\omega$. Since $\SRP(\mathcal{S}_\infSC)$ holds, it follows from Corollary \ref{cor:SCSRPimpliesCardArithmeticAboveContinuum} that $(\lambda^+)^{\omega_1}=\lambda^+$.
Thus we have
\[\lambda^+\le\lambda^{\cf(\lambda)}=\lambda^\omega\le(\lambda^+)^{\omega_1}=\lambda^+\]
as wished.
\end{proof}

Of course, $\SC$-$\SRP$ is sufficient in this corollary.

\subsection{Mutual stationarity}
\label{subsec:MutualStationarity}

The ideas of the previous subsection can be carried a little further. Let me recall the notion of mutual stationarity, introduced by Foreman \& Magidor \cite{ForemanMagidor:MutualStationarity}.

\begin{defn}
\label{def:MutualStationarity}
Let $K$ be a collection of regular cardinals with supremum $\delta$, and let $\vec{S}=\seq{S_\kappa}{\kappa\in K}$ be a sequence such that for every $\kappa\in K$, $S_\kappa$ is a subset of $\kappa$. Then $\vec{S}$ is \emph{mutually stationary} if for every algebra $\mathcal{A}$ on $\delta$, there is an $N\prec\mathcal{A}$ such that for all $\kappa\in N\cap K$, $\sup(N\cap\kappa)\in S_\kappa$.
\end{defn}

It is easy to see that if $\vec{S}$ is mutually stationary, then for all $\kappa\in K$, $S_\kappa$ is a stationary subset of $\kappa$. The following beautiful and fundamental fact on mutual stationarity was proved in the article in which the concept was introduced and gives a condition under which the converse is also true.

\begin{fact}[{Foreman \& Magidor \cite[Thm.~7]{ForemanMagidor:MutualStationarity}}]
\label{fact:MutualStationarity}
Let $K$ be a set of uncountable regular cardinals, and let $\vec{S}=\seq{S_\kappa}{\kappa\in K}$ be a sequence such that for every $\kappa\in K$, $S_\kappa$ is a stationary subset of $S^\kappa_\omega$. Then, $\vec{S}$ is mutually stationary: for any algebra $\mathcal{A}$ on $\sup K$, there is a countable $N\prec\mathcal{A}$ such that for all $\kappa\in N\cap K$, $\sup(N\cap\kappa)\in S_\kappa$.
\end{fact}

It seems as though the following connection has not been made before:

\begin{cor}
\label{cor:ExpressingItInTermsOfMutualStationarity}
Let $K$ be a set of regular cardinals with $\min(K)>\omega_1$, and let $\vec{S}=\seq{S_\kappa}{\kappa\in K}$ be a sequence such that for every $\kappa\in K$, $S_\kappa\sub S^\kappa_\omega$ is stationary in $\kappa$. Let $\delta=\sup K$. Then the set
\[S=\{M\in[H_\delta]^\omega\st\forall\kappa\in M\cap K\quad\sup(M\cap\kappa)\in S_\kappa\}\]
is projective stationary.
\end{cor}

\begin{proof}
This is because if $A\sub\omega_1$ is stationary, and if we let $K'=K\cup\{\omega_1\}$ and $\vec{S}'=\vec{S}\cup\{\kla{\omega_1,A}\}$, then we can apply Fact \ref{fact:MutualStationarity} to these objects, showing that the set
\[S_A=\{M\in S\st M\cap\omega_1\in A\}\]
is a stationary subset of $[H_\delta]^\omega$.
\end{proof}

This can be readily improved as follows by adding in a ``correspondence'' as before, but this time between a partition of $\omega_1$ into stationary sets and a corresponding list of $\omega_1$-sequences of stationary sets.

\begin{cor}
\label{cor:AProjectiveStationarySet}
Let $K$ be a set of regular cardinals with $\min(K)>\omega_1$, and let $\vec{S}=\seq{S_{\kappa,i}}{\kappa\in K, i<\omega_1}$ be a sequence such that for every $\kappa\in K$ and every $i<\omega_1$, $S_{\kappa,i}\sub S^\kappa_\omega$ is stationary in $\kappa$. Let $\delta=\sup K$. Let $\seq{D_i}{i<\omega_1}$ be a maximal partition of $\omega_1$ into stationary sets (in the sense of Lemma \ref{lem:FengJechSFPpreliminary}.)
Then the set
\[S=\{M\in[H_\delta]^\omega\st\forall\kappa\in M\cap K\forall i<\omega_1\quad(M\cap\omega_1\in D_i\implies\sup(M\cap\kappa)\in S_{\kappa,i})\}\]
is projective stationary.
\end{cor}

\begin{proof}
Let $B\sub\omega_1$ be stationary, and let $i<\omega_1$ be such that $D_i\cap B$ is stationary. Let $K'=K\cup\{\omega_1\}$, and consider the sequence $\seq{S'_\kappa}{\kappa\in K'}$ defined by letting $S'_\kappa=S_{\kappa,i}$ for $\kappa\in K$ and $S'_{\omega_1}=D_i\cap B$. 
By Fact \ref{fact:MutualStationarity}, the set $\bar{S}=\{M\in [H_\delta]^\omega\st\forall\kappa\in M\cap K'\quad\sup(M\cap\kappa)\in S'_\kappa\}$ is stationary. But $\omega_1\in M$, for a club $C$ of $M$, and $\bS\cap C\sub S_B=\{M\in S\st M\cap\omega_1\in B\}$, showing that $S_B$ is stationary.
\end{proof}

This connection to mutual stationarity gives rise to a somewhat ``diagonal'' reflection principle for sequences of stationary sets of ordinals which follows from \SRP, in contrast to others that I will discuss in Section \ref{sec:SRPLimitations}. It is a kind of simultaneous reflection principle for sequences of stationary sets that live on different regular cardinals. To motivate it, recall that Observation \ref{obs:CharacterizationOfE-versionOfSFPkappa} shows that a weak version of $\SFP{\kappa}$ that follows from \SRP for regular $\kappa>\omega_1$ implies that whenever $\{S_i\st i<\omega_1\}$ is a collection of stationary subsets of $S^\kappa_\omega$, then the set of exact simultaneous reflection points of this collection, is stationary in $\kappa$. The following theorem says that if we are given such collections of stationary sets, living on different regular cardinals, then the sequence consisting of the sets of the exact reflection points of these different collections is mutually stationary.

\begin{thm}
\label{thm:SRPimpliesMutuallyStationarySimultaneousReflection}
Assume \SRP. Let $K$ be a set of regular cardinals with $\min(K)>\omega_1$. Let $\vec{S}=\seq{S_{\kappa,i}}{\kappa\in K, i<\omega_1}$ be such that for every $\kappa\in K$ and $i<\omega_1$, $S_{\kappa,i}$ is a subset of $S^\kappa_\omega$ stationary in $\kappa$. For $\kappa\in K$, let $\vec{S}_\kappa=\seq{S_{\kappa,i}}{i<\omega_1}$. Then the sequence
\[\vec{T}=\seq{\eTr(\vS_\kappa)}{\kappa\in K}\]
is mutually stationary.
\end{thm}

\begin{proof}
Let $\delta=(\sup(K))^+$. Fix a partition $\seq{A_i}{i<\omega_1}$ of $\omega_1$ into stationary sets which is a maximal antichain. Let $S$ be the set of countable $M\prec H_\delta$ such that if $M\cap\omega_1\in A_i$, then for all $\kappa\in M\cap K$, $\sup(M\cap\kappa)\in S_{\kappa,i}$.
By Corollary \ref{cor:AProjectiveStationarySet}, this set is projective stationary. By \SRP, let $\seq{M_\alpha}{\alpha<\omega_1}$ be a continuous elementary chain through $S$. Let $M=\bigcup_{\alpha<\omega_1}M_\alpha$. Then $M\prec H_\delta$. We claim that $M$ verifies that $\vec{T}$ is mutually stationary. To see this, suppose $\kappa\in M\cap K$. We have to show that for every $i<\omega_1$, $s_\kappa=\sup(M\cap\kappa)\in\eTr(\vS_{\kappa})$. That is, we have to show that for every $i<\omega_1$, $s_\kappa$ is a reflection point of $S_{\kappa,i}$, and that $\bigcup_{i<\omega_1}S_{\kappa,i}\cap s_\kappa$ contains a club.

For the first part, fix any countable ordinal $i$. To see that $S_{\kappa,i}\cap s_\kappa$ is stationary, let $D\sub s_\kappa$ be club. We have to show that $S_{\kappa,i}\cap D\neq\leer$. Let $\beta<\omega_1$ be least such that $\kappa\in M_\beta$, and define, for $j\in[\beta,\omega_1)$, $\xi_j=\sup(M_j\cap\kappa)$. Then $\vec{\xi}$ is strictly increasing, continuous, and cofinal in $s_\kappa$. That is, $C=\{\xi_j\st\beta\le j<\omega_1\}$ is club in $s_\kappa$. Since $\cf(s_\kappa)=\omega_1$, $C\cap D$ is club in $s_\kappa$, and hence, $\bar{D}=\{j\in[\beta,\omega_1)\st\xi_j\in D\}$ is club in $\omega_1$.
Similarly, the sequence $\seq{M_j\cap\omega_1}{j<\omega_1}$ is strictly increasing and continuous, and so the set $E=\{j<\omega_1\st j=M_j\cap\omega_1\}$ is club in $\omega_1$. Now since $T_i$ is stationary in $\omega_1$, we can pick $\alpha\in T_i\cap\bar{D}\cap E$. Then $\xi_\alpha\in D$, since $\alpha\in\bar{D}$, and $\alpha=M_\alpha\cap\omega_1\in T_i$. Since $M_\alpha\in S$ and $\alpha\ge\beta$, that is, $\kappa\in M_\alpha$, it follows that $\xi_\alpha=\sup(M_\alpha\cap\kappa)\in S_{\kappa,i}$. Thus, $\xi_\alpha\in S_{\kappa,i}\cap D$, as wished.

For the second part, note that the club $C$ defined in the previous paragraph is contained in $\bigcup_{i<\omega_1}S_{\kappa,i}$
\end{proof}

I will show in the following that the same conclusion can be drawn from \infSC-\SRP under the additional assumption that \CH holds. Essentially, this amounts to showing a version of Corollary \ref{cor:ExpressingItInTermsOfMutualStationarity} with ``spread out'' in place of ``projective stationary''. To this end, I will use the following strengthening of Lemma \ref{lem:OneStep}. The argument will be a construction that proceeds in $\omega$ many steps, each of which will be an application of Lemma \ref{lem:OneStep}.

\begin{lem}
\label{lem:OmegaStep}
Let $K$ be a set of regular cardinals such that $\min(K)>2^\omega$ and such that whenever $\kappa<\lambda$, $\kappa,\lambda\in K$, then $\kappa^\omega<\lambda$. Let $\vec{S}=\seq{S_\kappa}{\kappa\in K}$ be such that for every $\kappa\in K$, $S_\kappa\sub S^\kappa_\omega$ is stationary in $\kappa$.
Let $\theta>2^{2^{\sup(K)}}$ be regular, $N$ a transitive model of $\ZFCm$ that has a definable well-order, with $H_\theta\sub N$. Let $\sigma:\bN\prec N$, where $\bN$ is countable and full and $K\in\ran(\sigma)$. Let $\bar{K}$ be the preimage of $K$, and let $\bar{a}$ be some element of $\bN$.

Then there is an embedding $\sigma':\bN\prec N$ with the following properties:
\begin{enumerate}[label=\textnormal{(\alph*)}]
  \item
  \label{item:LimitEmbeddingMovesStuffCorrectly}
  $\sigma(\bar{a})=\sigma'(\bar{a})$ and $\sigma'(\bar{K})=K$,
  \item
  \label{item:LimitEmbeddingHasTheRightSuprema}
  for every $\bar{\kappa}\in\bar{K}$, $\sup\sigma'``\bkappa\in S_{\sigma'(\bkappa)}$.
\end{enumerate}
\end{lem}

\begin{proof}
We may assume that $K$ is infinite. Let $\seq{\bkappa_n}{n<\omega}$ enumerate $\bar{K}$. Let us also fix an enumeration $\seq{\bar{a}_n}{n<\omega}$ of $\bN$. Also, letting $\rho=\sup(K)$, let us fix, for every $\alpha\in S^\rho_\omega$, an increasing and cofinal function $f_\alpha:\omega\To\alpha$.

We will construct sequences $\seq{\sigma'_n}{n<\omega}$, $\seq{\kappa_n}{n<\omega}$, $\seq{\tkappa_n}{n<\omega}$ and $\seq{\beta^n_{m,\ell}}{m\le\ell\le n<\omega}$ by simultaneous recursion on $n$, satisfying the following properties, for every $n<\omega$:
\begin{enumerate}[label=(\roman*)]
  \item
  \label{item:ElementaryEmbeddings}
  $\sigma'_n:\bN\prec N$.
  \item
  \label{item:SupremaCorrect}
  Let $\tkappa_n=\sup\sigma'_n``\bkappa_n$ and $\kappa_n=\sigma'_n(\bkappa_n)$. Then
  $\tkappa_n\in S_{\kappa_n}$.
  \item
  \label{item:EverybodyMovesBasicsTheSameWay}
  $\sigma'_n\rest\{\bar{K},\bar{a}\}=\sigma\rest\{\bar{K},\bar{a}\}$.
  \item
  \label{item:DiagonalUnionPrep}
  For $m\le n$, $\sigma'_n(\bar{a}_m)=\sigma'_{n+1}(\bar{a}_m)$.
  \item
  \label{item:SupremaPreserved}
  For $m<n$, $\sigma'_m(\bkappa_m)=\sigma'_n(\bkappa_m)$ and $\sup\sigma'_m``\bkappa_m=\sup\sigma'_n``\bkappa_m$.
  \item
  \label{item:InsuringSupremaOfLimitAreCorrect}
  For $m\le\ell\le n$, let $\beta^n_{m,\ell}<\bkappa_m$ be the least ordinal $\beta$ such that $\sigma'_n(\beta)> f_{\tkappa_m}(\ell)$. Then $\sigma'_n(\beta^n_{m,\ell})=\sigma'_{n+1}(\beta^n_{m,\ell})$. Moreover, for all $k<n$ and all $m\le\ell\le k$, $\sigma'_{n+1}(\beta^k_{m,\ell})=\sigma'_n(\beta^k_{m,\ell})$.
\end{enumerate}

To get started, let $\kappa_0=\sigma(\bkappa_0)$ and
apply Lemma \ref{lem:OneStep} to $\kappa_0$ and $S_{\kappa_0}$. We don't need the full strength of the lemma in step 0 of the construction, but just conditions \ref{item:MoveBasicParametersTheSameWay} and \ref{item:CorrectSupremum} (so we can let $\eta=0$). This gives us a $\sigma'_0:\bN\prec N$ that moves $\bar{K}$, $\bar{a}$, $\bkappa_0$ the same way $\sigma$ does, and such that letting $\tkappa_0=\sup\sigma'_0``\bkappa_0$, we have that $\tkappa_0\in S_{\kappa_0}$, since there is an $\omega$-club in $\kappa_0$ of possibilities for $\tkappa$ and $S_{\kappa_0}$ is a stationary subset of $\kappa_0$ consisting of ordinals of countable cofinality.
Thus, conditions \ref{item:ElementaryEmbeddings}-\ref{item:EverybodyMovesBasicsTheSameWay} are satisfied at $n=0$, and the remaining conditions are vacuous at $n=0$. Define $\beta^0_{0,0}$ as in condition \ref{item:InsuringSupremaOfLimitAreCorrect}. This is possible because $\sigma'``\bkappa_0$ is cofinal in $\tkappa_0$.

Now let us assume that $\seq{\sigma'_m}{m\le n}$ have been constructed, and $\seq{\kappa_m}{m\le n}$, $\seq{\tkappa_m}{m\le n}$ and $\seq{\beta^k_{m,\ell}}{m\le\ell\le k\le n}$ have been defined accordingly, so that all of the conditions are satisfied at each $m\le n$. Let $\kappa_{n+1}=\sigma'_n(\bkappa_{n+1})$. Let $\eta=\max(\{0\}\cup(\{\kappa_m\st m\le n\}\cap\kappa_{n+1}))$. Since $\kappa_{n+1}\in K$, $\eta<\kappa_{n+1}$ and either $\eta=0$ or $\eta\in K$, it follows from our assumption on $K$ that $\eta^\omega<\kappa_{n+1}$. Moreover, $\eta\in\ran(\sigma'_n)$, so we can let $\bar{\eta}$ be the preimage of $\eta$ under $\sigma'_n$. Let $\blambda_0,\ldots,\blambda_r$ enumerate the finite set $\{\bkappa_m\st m\le n\land\bkappa_m>\bkappa_{n+1}\}$. Let
\[p=\{\bar{K},\bar{a}\}\cup\{\bar{a}_m\st m\le n\}\cup\{\bkappa_m\st m\le n\}\cup\{\beta^{k}_{m,\ell}\st m\le\ell\le k\le n\}\]
Now apply Lemma \ref{lem:OneStep} to $\sigma'_n:\bN\prec N$, $S_{\kappa_{n+1}}$, $\kappa_{n+1}$, $\eta$, $\vec{\blambda}$ and $p$. This gives us an embedding $\sigma'_{n+1}:\bN\prec N$ such that $\sigma'_{n+1}\rest p=\sigma'_n\rest p$, $\sup\sigma'_{n+1}``\blambda_i=\sup\sigma'_n``\blambda_i$ for $i<r$, $\sigma'_{n+1}\rest\bar{\eta}=\sigma'_n\rest\bar{\eta}$ and, letting $\tkappa_{n+1}=\sup\sigma'_{n+1}``\bkappa_n$ and $\kappa_{n+1}=\sigma'_{n+1}(\bkappa_{n+1})$, we have that  $\tkappa_{n+1}\in S_{\kappa_{n+1}}$.

Let us check that the conditions are satisfied at $n+1$. This is immediate for conditions \ref{item:ElementaryEmbeddings}, \ref{item:SupremaCorrect}, \ref{item:EverybodyMovesBasicsTheSameWay} and \ref{item:DiagonalUnionPrep}. To check condition \ref{item:SupremaPreserved}, let $m<n+1$. We have to show that $\sigma'_m(\bkappa_m)=\sigma'_{n+1}(\bkappa_m)$ and $\sup\sigma'_m``\bkappa_m=\sup\sigma'_{n+1}``\bkappa_m$. Since inductively, the conditions are satisfied at $n$, we know that $\sigma'_m(\bkappa_m)=\sigma'_n(\bkappa_m)$ and
$\sup\sigma'_m``\bkappa_m=\sup\sigma'_n``\bkappa_m$. Hence, it suffices to show that $\sigma'_n(\bkappa_m)=\sigma'_{n+1}(\bkappa_m)$ and
$\sup\sigma'_n``\bkappa_m=\sup\sigma'_{n+1}``\bkappa_m$. The first part of this is clear, because we put $\bkappa_m$ into $p$, and $\sigma'_n\rest p=\sigma'_{n+1}\rest p$. For the second part, we consider two cases. If $\bkappa_m<\bkappa_{n+1}$, then $\bkappa_m\le\bar{\eta}$, and $\sigma'_{n+1}\rest\bar{\eta}=\sigma'_n\rest\bar{\eta}$, so in particular, $\sup\sigma'_n``\bkappa_m=\sup\sigma'_{n+1}``\bkappa_m$. If $\bkappa_m>\bkappa_{n+1}$, then $\bkappa_m=\blambda_i$, for some $i<r$, and hence, we have the desired equality because $\sup\sigma'_{n+1}``\blambda_i=\sup\sigma'_n``\blambda_i$. Finally, regarding Condition \ref{item:InsuringSupremaOfLimitAreCorrect}, observe that since $\sigma``_{n+1}\bkappa_{n+1}$ is cofinal in $\tkappa_{n+1}$, $\beta^{n+1}_{m,\ell}$ is well-defined for $m\le\ell\le n+1$. The remainder of this condition is again clear because $\sigma'_{n+1}\rest p=\sigma'_n\rest p$ and the relevant ordinals are in $p$.

This finishes the recursive construction $\seq{\sigma'_n}{n<\omega}$, $\seq{\kappa_n}{n<\omega}$, $\seq{\tkappa_n}{n<\omega}$ and $\seq{\beta^n_{m,\ell}}{m\le\ell\le n<\omega}$.

Observe that for every $\bar{x}\in\bN$, $\seq{\sigma'_n(\bar{x})}{n<\omega}$ is eventually constant, so we can define $\sigma':\bN\To N$ by letting $\sigma'(\bar{x})$ be the eventual value of this sequence, in other words,
\[\sigma'(\bar{a}_n)=\sigma'_n(\bar{a}_n)\]
for every $n<\omega$. We claim that $\sigma'$ is the desired embedding.

First, note that $\sigma':\bN\prec N$, because if $\phi(\vx)$ is a formula in the language of $N$ with $j<\omega$ free variables and $\bar{a}_{n_0},\ldots,\bar{a}_{n_{j-1}}$ are parameters in $\bN$, then if we choose $n\ge\max\{n_0,\ldots,n_{j-1}\}$, we have that $\sigma'(\vec{\bar{a}})=\sigma'_n(\vec{\bar{a}})$, and so
\[\bN\models\phi(\vec{\bar{a}})\iff N\models\phi(\sigma'_n(\vec{\bar{a}}))\iff N\models\phi(\sigma'(\vec{\bar{a}}))\]
since $\sigma'_n$ is an elementary embedding.
Obviously, we have that $\sigma'(\bkappa_n)=\kappa_n$, for every $n<\omega$, and $\sigma'(\bar{a})=\sigma(\bar{a})$. Thus, condition \ref{item:LimitEmbeddingMovesStuffCorrectly} is satisfied. Let us check the remaining condition \ref{item:LimitEmbeddingHasTheRightSuprema}. Let $\bkappa\in\bar{K}$, say $\bkappa=\bkappa_n$. We have to show that $\sup\sigma'``\bkappa_n\in S_{\sigma'(\bkappa_n)}$. Since $\sigma'(\bkappa_n)=\kappa_n$ and $\tkappa_n\in S_{\kappa_n}$, it will suffice to show that $\sup\sigma'``\bkappa_n=\tkappa_n$. Clearly, $\sigma'``\bkappa_n\sub\tkappa_n$, since for every $\xi<\bkappa_n$, there is a $k<\omega$ with $k\ge n$ such that $\sigma'(\xi)=\sigma'_k(\xi)$, but by condition \ref{item:SupremaPreserved}, $\sup\sigma'_k``\bkappa_n=\sup\sigma'_n``\bkappa_n=\tkappa_n$, so $\sigma'(\xi)<\tkappa_n$. Thus, $\sigma'``\bkappa_n\le\tkappa_n$. For the other inequality, we show that $\sigma'``\bkappa_n$ is unbounded in $\tkappa_n$. Since $\ran(f_{\tkappa_n})$ is unbounded in $\tkappa_n$, it suffices to show that for every $\ell<\omega$, there is a $\beta<\bkappa_n$ such that $\sigma'(\beta)>f_{\tkappa_n}(\ell)$. We may clearly assume that $\ell\ge n$.
Let $k<\omega$, $k\ge\max\{n,\ell\}$. Consider $\beta=\beta^k_{n,\ell}$. By definition, $\beta<\bkappa_n$ and $\sigma'_k(\beta)>f_{\tkappa_n}(\ell)$. Moreover, by the same condition, we have that $\sigma'_{k'}(\beta)=\sigma'_k(\beta)$ whenever $k<k'<\omega$. Hence, $\sigma'(\beta)=\sigma'_k(\beta)>f_{\tkappa_n}(\ell)$, as claimed.
\end{proof}

We thus obtain the following version of Foreman \& Magidor's mutual stationarity Fact \ref{fact:MutualStationarity}, or rather, the equivalent Corollary \ref{cor:ExpressingItInTermsOfMutualStationarity}.

\begin{cor}
\label{cor:ASimpleSpreadOutSet}
Let $K$ be a set of regular cardinals such that $\min(K)>2^\omega$ and such that whenever $\kappa<\lambda$, $\kappa,\lambda\in K$, then $\kappa^\omega<\lambda$. Let $\vec{S}=\seq{S_{\kappa}}{\kappa\in K}$ be such that for every $\kappa\in K$, $S_\kappa\sub S^\kappa_\omega$ is stationary in $\kappa$. Let $\rho\ge\sup(K)$. Then the set
\[S=\{M\in[H_\rho]^\omega\st\forall\kappa\in M\cap K\quad\sup(M\cap\kappa)\in S_{\kappa})\}\]
is spread out.
\end{cor}

Instead of providing a proof of this, let me build in a correspondence as before, and prove the following more general statement. This is the version of Corollary \ref{cor:AProjectiveStationarySet}.

\begin{cor}
\label{cor:ASpreadOutSet}
Let $K$ be a set of regular cardinals such that $\min(K)>2^\omega$ and such that whenever $\kappa<\lambda$, $\kappa,\lambda\in K$, then $\kappa^\omega<\lambda$. Let $\vec{S}=\seq{S_{\kappa,i}}{\kappa\in K, i<\omega_1}$ be such that for every $\kappa\in K$ and every $i<\omega_1$, $S_{\kappa,i}\sub S^\kappa_\omega$ is stationary in $\kappa$. Let $\seq{D_i}{i<\omega_1}$ be a partition of $\omega_1$ into stationary sets, and let $\rho\ge\sup(K)$ be regular. Then the set
\[S=\{M\in[H_\rho]^\omega\st\forall\kappa\in M\cap K\forall i<\omega_1\quad(M\cap\omega_1\in D_i\implies\sup(M\cap\kappa)\in S_{\kappa,i})\}\]
is spread out.
\end{cor}

\begin{proof}
This is an immediate application of Lemma \ref{lem:OmegaStep}. Namely, to show that $S$ is spread out, let $X\prec L_\theta^A$ as usual, and let $\sigma:\bN\To X$ be the inverse of the Mostowski collapse of $X$, where we assume that $\bN$ is full. As usual, we may assume that $X$ contains the parameters we care about; in this case we choose $K$. Let $\bar{K}=\sigma^{-1}(K)$. Fix some $a\in X$, and let $\bar{a}=\sigma^{-1}(a)$. Let $\delta=X\cap\omega_1=\omega_1^\bN$, and let $i<\omega_1$ be such that $\delta\in D_i$. Applying Lemma \ref{lem:OmegaStep} to the sequence $\seq{S_{\kappa,i}}{\kappa\in K}$, there is a $\sigma':\bN\prec N$ with $\sigma'(\bar{K})=K$, and $\sigma'(\bar{a})=\sigma(\bar{a})$, such that for every $\bar{\kappa}\in\bar{K}$, $\sup\sigma'``\bkappa\in S_{\sigma'(\bkappa),i}$. Now let $Y=\ran(\sigma')$. Then $\pi=\sigma'\compose\sigma^{-1}:X\To Y$ is an isomorphism fixing $\sigma(\bar{a})$, and $Y\cap H_\rho\in S$: since neither $\sigma$ nor $\sigma'$ move countable ordinals, we have that $Y\cap\omega_1=X\cap\omega_1=\delta$. Now let $\kappa\in Y\cap K$. Let $\bkappa=\sigma'^{-1}(\kappa)$. Since $\sigma'(\bar{K})=K$, we have that $\bkappa\in\bar{K}$. Let $\tkappa=\sup\sigma'``\bkappa$. Then $\sigma'``\bkappa=Y\cap\kappa$, and so, $\tkappa=\sup(Y\cap\kappa)\in S_{\kappa,i}$, as wished.
\end{proof}

It is easy to see that the construction in the proof of Lemma \ref{lem:OmegaStep} can be modified so as to obtain the version of the previous corollary in which ``spread out'' is replaced with ``fully spread out.'' The corollary can be made to be closer to Corollary \ref{cor:AProjectiveStationarySet} by adding the assumption of $\infSC$-\SRP, because then the cardinal arithmetic requirements on $K$ are automatically satisfied:

\begin{cor}
\label{cor:ASpreadOutSetAssumingSCSRP}
Assume $\infSC$-\SRP. Let $K$ be a set of regular cardinals such that $\min(K)>2^\omega$. Let $\vec{S}=\seq{S_{\kappa,i}}{\kappa\in K, i<\omega_1}$ be such that for every $\kappa\in K$ and every $i<\omega_1$, $S_{\kappa,i}\sub S^\kappa_\omega$ is stationary in $\kappa$. Let $\seq{D_i}{i<\omega_1}$ be a partition of $\omega_1$ into stationary sets, and let $\rho\ge\sup(K)$ be regular. Then the set
\[S=\{M\in[H_\rho]^\omega\st\forall\kappa\in M\cap K\forall i<\omega_1\quad(M\cap\omega_1\in D_i\implies\sup(M\cap\kappa)\in S_{\kappa,i})\}\]
is spread out.
\end{cor}

\begin{proof}
The point is that under $\infSC$-$\SRP$, we have that for any regular cardinal $\kappa>2^\omega$, $\kappa^{\omega_1}=\kappa$, by Theorem \ref{thm:infSCSRPimpliesSFPkappa} and Fact \ref{fact:SFPandArithmetic}. Thus, Corollary \ref{cor:ASpreadOutSet} applies, completing the proof.
\end{proof}

Thus, under $\infSC\text{-}\SRP+\CH$, an even stronger version of Corollary \ref{cor:AProjectiveStationarySet} holds with ``spread out,'' and even ``fully spread out,'' in place of ``projective stationary,'' since it does not assume the partition of $\omega_1$ into stationary sets to be maximal.
Here is the promised version of Theorem \ref{thm:SRPimpliesMutuallyStationarySimultaneousReflection} for \infSC-\SRP, and even \SC-\SRP:

\begin{thm}
\label{thm:infSC-SRP+CHimpliesMutuallyStationarySimultaneousReflection}
Assume $\infSC$-$\SRP$. Let $K$ be a set of regular cardinals with $\min(K)>2^\omega$. Then the conclusions of Theorem \ref{thm:SRPimpliesMutuallyStationarySimultaneousReflection} hold:
let $\vec{S}=\seq{S_{\kappa,i}}{\kappa\in K, i<\omega_1}$ be such that for every $\kappa\in K$ and $i<\omega_1$, $S_{\kappa,i}$ is a subset of $S^\kappa_\omega$ stationary in $\kappa$.
For $\kappa\in K$, let $\vec{S}_\kappa=\seq{S_{\kappa,i}}{i<\omega_1}$. Then the sequence
\[\vec{T}=\seq{\eTr(\vS_\kappa)}{\kappa\in K}\]
is mutually stationary.

Thus, under the additional assumption of \CH, we get the full conclusion of Theorem \ref{thm:SRPimpliesMutuallyStationarySimultaneousReflection} from \infSC-\SRP.
\end{thm}

\begin{proof}
The point is that under $\infSC$-$\SRP$, we have that for any regular cardinal $\kappa>2^\omega$, $\kappa^{\omega_1}=\kappa$, by Theorem \ref{thm:infSCSRPimpliesSFPkappa} and Fact \ref{fact:SFPandArithmetic}. Thus, the assumptions on $K$ in Corollary \ref{cor:ASpreadOutSet} are satisfied. The theorem now follows by an argument exactly as in the proof of Theorem \ref{thm:SRPimpliesMutuallyStationarySimultaneousReflection}.
\end{proof}

The following result of Jensen \cite{Jensen:FAandCH} fits in here well. I recast it as a statement about mutual stationarity.

\begin{thm}[Jensen]
\label{thm:JensenOnMutualStationarity}
Assume that \SCFA holds, and \GCH holds below $\lambda$, an uncountable cardinal. Let $K\sub\lambda$ be a set of regular cardinals greater than $\omega_1$, and let $f:K\To\{\omega,\omega_1\}$. Then the sequence $\seq{S^\kappa_{f(\kappa)}}{\kappa\in K}$ is mutually stationary.
\end{thm}

The proof of this theorem uses the subcompleteness of an intricate forcing notion, developed in \cite{Jensen:ExtendedNamba}, that changes the cofinality of some regular cardinals to be countable, while preserving that others have uncountable cofinality. Clearly, this forcing is not countably distributive. It seems unlikely that in this theorem, \SCFA can be replaced with \SC-\SRP, as the forcing notions of the form $\P_S$ are countably distributive. This question is a fitting segue into the next section, in which I will explore stationary reflection principles that do not follow from fragments of \SRP, and consequences of \SCFA  that don't follow from \SC-\SRP, thus separating these assumptions.

%
%
%

\section{Limitations and separations}
\label{sec:SRPLimitations}

I will now explore limitations on the extent to which the subcomplete fragment of \SRP implies certain principles of stationary reflection, and I will develop some results going in the direction of separating the subcomplete fragment of \SRP from \SCFA. In the first subsection, I will focus on achieving such results in a general setting, but the results strongly suggest that the addition of \CH should be made. Consequently, the last two subsections deal with this scenario.

\subsection{The general setting}
\label{subsec:GeneralSetting}

Here, I will explore a framework for obtaining limiting results, following the approach of Larson \cite{Larson:SeparatingSRP}, originally in the setting of the full \SRP. The idea is to start in a model of set theory $\V$ in which an indestructible version of \SRP holds, namely, \SRP holds in $\V$ and in any forcing extension of $\V$ obtained by $\omega_2$-directed closed forcing. This indestructible form of \SRP follows from \MM. To show that \SRP does not imply a statement $\phi$, one then forces with a poset $\P$ to add a counterexample to $\phi$. Let's say the forcing extension is $\V[G]$. One then argues that there is in $\V[G]$ a further forcing $\bbT=\dot{\bbT}^G$ such that $\P*\dot{\bbT}$ is $\omega_2$-distributive. Letting $H$ be $\V[G]$-generic for $\Q$, one has then that $\V[G][H]$ satisfies \SRP. This is then used to argue that $\V[G]$ also satisfies $\SRP$. The crucial step here is to show that if $S$ is projective stationary in $\V[G]$, then this remains true in $\V[G][H]$. The main technical problem is thus to prove the \emph{preservation of projective stationarity.}

The property $\phi$ in this approach is usually a statement about stationary reflection, and so, the forcing $\P$ usually adds a sequence of stationary sets that does not reflect in a certain way. The forcing $\bbT$ is usually a forcing that destroys this counterexample to $\phi$ by destroying the stationarity of some of the sets in the sequence, that is, by shooting a club through the complement of some of these sets. Given a stationary and costationary subset $A$ of some regular uncountable cardinal $\kappa$, the forcing used is called $\bbT_A$ (the forcing to ``kill'' $A$), and consists of closed bounded subsets of $\kappa$ disjoint from $A$, and ordered by end-extension. Larson \cite[Lemma 4.5]{Larson:SeparatingSRP} has analyzed when this forcing notion preserves generalized stationarity; see  \cite[Lemma 4.3]{Fuchs-LambieHanson:SeparatingDSR} for a proof of this exact formulation.

\begin{lem}
\label{lem:WhenT_APreservesGenStationarity}
Let $\gamma>\omega_1$ be regular, $X\supseteq\gamma$ a set, $A\subseteq \gamma$ stationary and $S\subseteq[X]^\omega$ also stationary, such that $\gamma\setminus A$ is unbounded in $\gamma$ and $\bbT_A$ is countably distributive (this is the case, for example, if $A\subseteq S^\gamma_\omega$ and $S^\gamma_\omega \setminus
A$ is stationary). Then the following are equivalent:
\begin{enumerate}[label=\textnormal{(\arabic*)}]
\item
\label{item:SminusAisStationary}
$S\setminus\lifting(A,[X]^\omega)$ is stationary.
\item
\label{item:StationarityOfSisPreserved}
$\bbT_A$ preserves the stationarity of $S$.
\item
\label{item:SomebodyPreservesStationarityOfS}
There is a condition $p\in\bbT_A$ that forces that $\check{S}$ is  stationary.
\end{enumerate}
\end{lem}

This leads to a characterization of when $\bbT_A$ preserves \emph{projective} stationarity as follows.

\begin{lem}
\label{lem:WhenT_APreservesProjectiveStationarity}
Let $\gamma>\omega_1$ be regular, $X\supseteq\gamma$ a set, $A\subseteq \gamma$ stationary and $S\subseteq[X]^\omega$ projective stationary, such that $\gamma\setminus A$ is unbounded in $\gamma$ and $\bbT_A$ is countably distributive. Then the following are equivalent:
\begin{enumerate}[label=\textnormal{(\arabic*)}]
\item
\label{item:SminusAisProjectiveStationary}
$S\setminus\lifting(A,[X]^\omega)$ is projective stationary.
\item
\label{item:ProjectiveStationarityOfSisPreserved}
$\bbT_A$ preserves the projective stationarity of $S$.
\item
\label{item:SomebodyPreservesStationarityOfS}
There is a condition $p\in\bbT_A$ that forces that $\check{S}$ is  projective stationary.
\end{enumerate}
\end{lem}

\begin{proof}
\ref{item:SminusAisProjectiveStationary}$\implies$\ref{item:ProjectiveStationarityOfSisPreserved}:
Suppose $G\sub\bbT_A$ is generic and $S$ is not projective stationary in $\V[G]$. Then there is some stationary $B\sub\omega_1$ such that in $\V[G]$, $S\cap\lifting(B,[X]^\omega)$ is not stationary. Since $\bbT_A$ is countably distributive, $B\in\V$ is stationary in $\V$. Since in $\V$, $S\setminus\lifting(A,[X]^\omega)$ is projective stationary, we know that
$(S\setminus\lifting(A,[X]^\omega))\cap\lifting(B,[X]^\omega)$ is stationary. Note that this intersection is the same as $(S\cap\lifting(B,[X]^\omega))\ohne\lifting(A,[X]^\omega)$. So by Lemma \ref{lem:WhenT_APreservesGenStationarity}, $S\cap\lifting(B,[X]^\omega)$ is stationary in $\V[G]$, a contradiction.

\ref{item:ProjectiveStationarityOfSisPreserved}$\implies$
\ref{item:SomebodyPreservesStationarityOfS} is trivial.

\ref{item:SomebodyPreservesStationarityOfS}$\implies$\ref{item:SminusAisProjectiveStationary}: Let $p\in\bbT_A$ force that $\check{S}$ is projective stationary. Then, for every stationary $B\sub\omega_1$, $p$ forces that $S\cap\lifting(B,[X]^\omega)$ is stationary. By Lemma \ref{lem:WhenT_APreservesGenStationarity}, this implies that
$(S\cap\lifting(B,[X]^\omega))\setminus\lifting(A,[X]^\omega)$ is stationary. That is, for every stationary $B\sub\omega_1$, $(S\setminus\lifting(A,[X]^\omega)\cap\lifting(B,[X]^\omega)$ is stationary, which means that $S\setminus\lifting(A,[X]^\omega)$ is projective stationary, as wished.
\end{proof}

To formulate some sample applications of these methods, let me recall
the following principles of stationary reflection. The simulatenous reflection principles are well-known, see \cite{SquaresScalesStationaryReflection}. The first diagonal version was originally introduced in Fuchs \cite{Fuchs:DiagonalReflection}, and its variations come from Fuchs \& Lambie-Hanson \cite{Fuchs-LambieHanson:SeparatingDSR}.

\begin{defn}
\label{defn:DiagonalReflection}
Let $\lambda$ be a regular cardinal, let $S\subseteq\lambda$ be stationary, and let $\kappa<\lambda$.

The \emph{simultaneous reflection principle} $\refl(\kappa,S)$ says that whenever $\vS=\seq{S_i}{i<\kappa}$ is a sequence of stationary subsets of $S$, then $\vS$ has a simultaneous reflection point (see Def.~\ref{def:Trace}).

The \emph{diagonal reflection principle} $\DSR({<}\kappa,S)$ says that whenever $\langle S_{\alpha,i}\st\alpha<\lambda,i<j_\alpha\rangle$ is a sequence of stationary subsets of $S$, where $j_\alpha<\kappa$ for every $\alpha<\lambda$, then there are an ordinal $\gamma<\lambda$ of uncountable cofinality and a \emph{club} $F\subseteq\gamma$ such that for every $\alpha\in F$ and every $i<j_\alpha$, $S_{\alpha,i}\cap\gamma$ is stationary in $\gamma$.

The version of the principle in which $j_\alpha\le\kappa$ is denoted $\DSR(\kappa,S)$.

If $F$ is only required to be \emph{unbounded,} then the resulting principle is called $\uDSR({<}\kappa,S)$, and if it is required to be \emph{stationary,} then it is denoted $\sDSR({<}\kappa,S)$.
\end{defn}

These principles can be viewed as different ways of generalizing the following reflection principle, due to Larson \cite{Larson:SeparatingSRP}.

\begin{defn}
\label{def:OSR}
The principle \OSR{\omega_2} (for ``ordinal stationary reflection'') states that whenever $\vS=\seq{S_i}{i<\omega_2}$ is a sequence of stationary subsets of $S^{\omega_2}_\omega$, there is an ordinal $\rho<\omega_2$ of uncountable cofinality which is a simultaneous reflection point of $\vS\rest\rho$.
\end{defn}

The following fact shows that the principles introduced above indeed generalize $\OSR{\omega_2}$.

\begin{fact}[{special case of \cite[Lemma 2.4]{Fuchs-LambieHanson:SeparatingDSR}}]
\label{fact:DSREquivalences}
The following are equivalent:
\begin{enumerate}[label=\textnormal{(\arabic*)}]
\item $\uDSR(\omega_1,S^{\omega_2}_\omega)$
\item $\DSR(\omega_1,S^{\omega_2}_\omega)$
\item $\OSR{\omega_2}$
\end{enumerate}
\end{fact}

The result of Larson \cite[Thm.~4.7]{Larson:SeparatingSRP} most relevant here is that if Martin's Maximum is consistent, then it is consistent that $\SRP$ holds but $\OSR{\omega_2}$ fails. On the other hand, $\MM$ implies $\OSR{\omega_2}$. Hence, $\OSR{\omega_2}$ is a consequence of $\MM$ which is not captured by $\SRP$, thus separating $\MM$ from $\SRP$. The following lemma shows that, under the appropriate assumptions, the $\infty$-subcomplete fragment of \SRP does not even imply $\Refl(1,S^{\omega_2}_\omega)$, giving a much stronger failure of reflection.

\begin{lem}
\label{lem:SRP+nonCHisWeak}
Assume that \SRP holds and continues to hold after $\omega_2$-directed closed forcing. Then there is an $\omega_2$-strategically closed forcing of size $\omega_2$ in whose extensions the following hold:
\begin{enumerate}[label=(\alph*)]
  \item
  \label{item:SCSRPstillHolds}
  $\infSC$-\SRP,
  \item $2^\omega=\omega_2$,
  \item there is a nonreflecting stationary subset of $S^{\omega_2}_\omega$.
\end{enumerate}
\end{lem}

\begin{proof}
Let $\P$ be the forcing to add a nonreflecting stationary subset of $S^{\omega_2}_\omega$ by closed initial segments. Since $\SRP$ holds, we have that $\omega_2^{\omega_1}=\omega_2$, and hence, the cardinality of $\P$ is $\omega_2$, and it is well-known that $\P$ is $\omega_2$-strategically closed. A generic filter for $\P$ may be identified with its union, which is a nonreflecting stationary subset of $S^{\omega_2}_\omega$. Let $A$ be such a generic set. $\V[A]$ then has the desired properties, and only property \ref{item:SCSRPstillHolds} needs to be verified. To see this, let me prove a general claim, which may be useful in other contexts as well.

\begin{NumberedClaim}
\label{claim:AlmostPreservingSpreadOutSets}
Suppose $\omega_1<\lambda\le 2^\omega$, $\lambda$ regular, and $B$ is a stationary subset of $S^\lambda_\omega$. Let $\kappa\ge\lambda$ be regular, and let $S\sub[H_\kappa]^\omega$ be spread out. Then $S\cap\lifting(B,[H_\kappa]^\omega)$ is projective stationary.
\end{NumberedClaim}

\begin{proof}[Proof of \ref{claim:AlmostPreservingSpreadOutSets}.]
Let $D\sub\omega_1$ be stationary, and let $f:(H_\kappa)^{{<}\omega}\To H_\kappa$. We have to find an $x\in S$ such that $S\cap\omega_1\in D$, $x$ is closed under $f$ and $\sup(x\cap\lambda)\in B$.
To this end, let $\theta$ witness that $S$ is spread out.
Let $H_\theta\sub N=L_\tau^I$, where $S,\theta,f\in N$ and $\tau$ is a cardinal in $L[I]$. Let $\tau'=(\tau^+)^{L[I]}$, and let $N'=L_{\tau'}^I$. Let $x'$ be countable with $N'|x'\prec N$ and $S,\theta,f,\lambda\in x'$. Moreover, let $x'\cap\omega_1\in D$ and $\sup(x'\cap\lambda)\in B$. This is possible by Fact \ref{fact:MutualStationarity}. Let $x=x'\cap L_\tau[I]$. Then $x$ is full, and hence, since $S$ is spread out, there is a $y$ such that $N|x$ is isomorphic to $N|y$ via an isomorphism $\pi$ which fixes $f$, and such that $\bar{y}=y\cap H_\kappa\in S$.
But since $\lambda\le 2^\omega$, $\pi$ is the identity on $\lambda$ (see the paragraph before Observation \ref{obs:SCSRPtrivialuptocontinuum}), so that $\bar{y}$ is as wished: $\bar{y}\cap\omega_1=x\cap\omega_1\in D$, $\sup(\bar{y}\cap\lambda)=\sup(x\cap\lambda)\in B$, $\bar{y}\in S$, and since $f=\pi(f)\in y$, $\bar{y}$ is closed under $f$.
\end{proof}

Now, working in $\V[A]$, $A$ is a nonreflecting stationary subset of $S^{\omega_2}_\omega$. It follows that $B=S^{\omega_2}_\omega\ohne A$ is stationary. Hence, by Claim \ref{claim:AlmostPreservingSpreadOutSets}, $S\cap\lifting(B,[H_\kappa]^\omega)=S\ohne\lifting(A,[H_\kappa]^\omega)$ is projective stationary. By Lemma \ref{lem:WhenT_APreservesProjectiveStationarity}, it follows that the forcing $\bbT_A$ preserves the projective stationarity of $S$.

Now it is well-known that $\P*\dot{\bbT}_{\dot{A}}$ is forcing equivalent to an $\omega_2$-directed closed forcing of size $\omega_2$, so that if we let $H$ be generic for $\bbT_A$ over $\V[G]$, then $\SRP$ holds in $\V[G][H]$, where $S$ is projective stationary. Working in $\V[G][H]$, let $\seq{y_\alpha}{\alpha<\omega_1}$ be a continuous $\in$-chain through $S\projectup[H_\kappa^{\V[G][H]}]^\omega$. Then letting $x_\alpha=y_\alpha\cap H_\kappa^{\V[G]}$, it follows that $\vec{x}$ is a continuous $\in$-chain through $S$, and $\vx\in\V[G]$, since $\bbT_A$ is $\omega_2$-distributive in $\V[G]$.
\end{proof}

Note that it was crucial in the proof of (1) that $\lambda\le 2^\omega$, and the case $\lambda=\omega_2$ is what was used in the remainder of the proof.
Thus, we are in the funny situation that if we want $\infSC$-\SRP to be \emph{similar} to \SRP, in that it should yield consequences like $\Refl(\omega_1,S^{\omega_2}_\omega)$, for example, then we should add the assumption that \CH holds, which \emph{contradicts} \SRP. In fact, already Observation \ref{obs:SCSRPtrivialuptocontinuum} was a clear indication that one should add \CH to $\infSC$-\SRP in order to make it stronger.

I also have to point out that the previous lemma does \emph{not} separate $\infSC$-\SRP from $\infSC\FA$. In fact, Hiroshi Sakai observed that methods of Sean Cox can be used to show that $\infSC\FA+2^\omega=\omega_2$ is consistent with the failure of $\Refl(1,S^{\omega_2}_\omega)$ as well, and more, for example, $\infty$-\SCFA, if consistent, is consistent with $\square_{\omega_1}$. The details of the argument, and variations of it, have been worked out by Corey Switzer (forthcoming).

Thus, Lemma \ref{lem:SRP+nonCHisWeak} provides a fairly strong \emph{limitation} of $\infSC$-$\SRP$, but fails to \emph{separate} it from $\infty$-\SCFA. However, it is possible to obtain a separating result as follows. The first fact I need is that for a regular cardinal $\kappa$, $\DSR(\omega_1,S^\kappa_\omega)$ follows from Martin's Maximum if $\kappa>\omega_1$, and also from \SCFA, for $\kappa>2^\omega$, see \cite{Fuchs:DiagonalReflection}. Second, the principle does not follow from \SRP. In fact:

\begin{thm}[{Fuchs \& Lambie-Hanson \cite[Thm.~4.4]{Fuchs-LambieHanson:SeparatingDSR}}]
\label{thm:FLH-SeparatingSRPfromuDSRhigherUp}
Let $\kappa>\omega_2$ be regular, and suppose that $\SRP$ holds and continues to hold in any forcing extension obtained by $\kappa$-directed closed forcing. Then there is a $\kappa$-strategically closed forcing notion that produces forcing extensions in which
\begin{enumerate}[label=\textnormal{(\arabic*)}]
  \item \SRP continues to hold, but
  \item $\uDSR(1,S^\kappa_\omega)$ fails.
\end{enumerate}
\end{thm}

Thus, a model as in the previous theorem will satisfy $\infSC$-\SRP (since it even satisfies the full \SRP), but not \SCFA, since $\kappa$ is regular and greater than $\omega_2=2^\omega$ in that model, so if \SCFA held, then this would imply $\DSR(\omega_1,S^\kappa_\omega)$ by the abovementioned fact, contradicting that even $\uDSR(1,S^\kappa_\omega)$ fails. Thus:

\begin{fact}
\label{fact:SeparatingSC-SRPfromSCFA}
Assuming the consistency of \MM, $\infSC$-\SRP does not imply \SCFA.
\end{fact}

While this is in a sense a satisfying separation result, I view the fact that $\infSC$-$\SRP(\omega_2)$ is trivial if $2^\omega\ge\omega_2$, and that $\infSC$-\SRP, and even $\infty$-\SCFA, do not imply $\Refl(1,S^{\omega_2}_\omega)$ if \CH fails, as strong indications that these axioms should be augmented with the assumption of \CH. But in the model that witnesses the separation fact just stated, we have that $2^\omega=\omega_2$, so I have not separated $\infSC$-$\SRP+\CH$ from $\infty$-$\SCFA$. Since these are the axioms of the main interest, I will focus on them for the remainder of this article. The goal result would be that the answer to the following question is no:

\begin{question}
\label{question:SeparatingSCSRPfromuDSR}
Does $\infSC$-\SRP $+$ $\CH$ imply $\uDSR(1,S^\kappa_\omega)$, for any regular $\kappa>\omega_1$, or any of the diagonal reflection principles?
\end{question}

Note that answering this question in the negative would strengthen the conclusion of Theorem \ref{thm:FLH-SeparatingSRPfromuDSRhigherUp}. The reason why I am hopeful that this might be the case is that the assumption of \CH might replace the assumption that the nonstationary ideal in $\omega_1$-dense in the following.

\begin{cor}[{Fuchs \& Lambie-Hanson \cite[Cor.~4.7]{Fuchs-LambieHanson:SeparatingDSR}}]
\label{cor:SeparatingSRPfromuDSR}
Suppose that
{\sf SRP} holds and continues to hold in any forcing extension obtained by an $\omega_2$-directed closed forcing notion. Assume furthermore that the density of the nonstationary ideal on $\omega_1$ is $\omega_1$. Then
there is an $\omega_2$-strategically closed forcing notion which produces forcing extensions where
\begin{enumerate}[label=\textnormal{(\arabic*)}]
  \item {\sf SRP} continues to hold, but
  \item $\uDSR(1,S^{\omega_2}_\omega)$ fails.
\end{enumerate}
\end{cor}

However, there are obstacles to obtaining versions of
Theorem \ref{thm:FLH-SeparatingSRPfromuDSRhigherUp} or the corollary just mentioned for $\infSC$-\SRP in the context of \CH. The main problem is to find a version of the Preservation Lemma \ref{lem:WhenT_APreservesProjectiveStationarity} for spread out sets rather than for projective stationarity.

\subsection{The \CH setting: a less canonical separation}
\label{subsec:CHsetting-LessCanonical}

I will now move on to a partial separation result which may be cause for optimism that the answer to Question \ref{question:SeparatingSCSRPfromuDSR} is negative.
It concerns the fragment of $\infSC$-\SRP that says that \SRP holds for all the spread out sets shown to be spread out in this article. These are the ones from Subsection \ref{subsec:MutualStationarity}. Since this fragment of \SRP does not correspond to a forcing class, I don't consider it to be canonical, hence the title of the present subsection.

\begin{defn}
\label{def:MS-correspondences}
Let $\rho$ be regular, and let $K\sub\rho$ be a set of regular cardinals such that $\min(K)>\omega_1$. Let $\vec{S}=\seq{S_{\kappa,i}}{\kappa\in K, i<\omega_1}$ be such that for every $\kappa\in K$ and every $i<\omega_1$, $S_{\kappa,i}\sub S^\kappa_\omega$ is stationary in $\kappa$. Let $\seq{D_i}{i<\omega_1}$ be a partition of $\omega_1$ into stationary sets, and let $\rho\ge\sup(K)$ be regular.
Let's call $\kla{\vec{D},\vec{S}}$ a $K$-correspondence, and
define
\begin{ea*}
\lifting(\vec{D},\vec{S},[H_\rho]^\omega)&=&\{x\in[H_\rho]^\omega\st\forall\kappa\in x\cap K\forall i<\omega_1\\
&&\qquad(x\cap\omega_1\in D_i\implies\sup(x\cap\kappa)\in S_{\kappa,i})\}.
\end{ea*}
Let
\begin{ea*}
\mathcal{S}^+(\rho)&=&\{\lifting(\kla{\vec{D},\vec{S}},[H_\rho]^\omega)\cap C\st\kla{\vD,\vS}\ \text{is a $K$-correspondence}\\
&&\qquad \text{for some $K\sub\rho$, and $C\sub[H_\rho]^\omega$ is club}\}.
\end{ea*}
Let's write $\SRP(\mathcal{S}^+)$ for the statement that $\SRP(\mathcal{S}^+(\rho))$ holds for every regular $\rho>\omega_1$.
\end{defn}

\begin{obs}
It follows from the results of the previous subsections that:
\begin{enumerate}[label=\textnormal{(\arabic*)}]
\item $\text{$\infSC$-\SRP}+\CH$ implies $\SRP(\mathcal{S}^+)$.
\item If $\SRP(\mathcal{S}^+)+\CH$ holds, then for all regular $\kappa>\omega_1$, $\kappa^{\omega_1}=\kappa$, and hence, every set in $\mathcal{S}^+(\rho)$, for $\rho>\omega_1$, is spread out.
\end{enumerate}
\end{obs}

The assumptions of the following theorem are implied by $\SCFA+\CH$; see Fuchs \& Rinot \cite{FuchsRinot:WeakSquareStationaryReflection}.

\begin{thm}
\label{thm:SRPliftingDoesNotImplyuDSR}
Assume that $\SRP(\mathcal{S}^+)$ holds in $\V$ and in any forcing extension obtained by a $\lambda$-directed closed forcing, and that \CH is true.
Then there is a $\lambda$-strategically closed forcing in whose forcing extensions
\begin{enumerate}[label=\textnormal{(\arabic*)}]
\item $\SRP(\mathcal{S}^+)$ holds, yet
\item $\uDSR(1,S^{\omega_2}_\omega)$ fails.
\end{enumerate}
\end{thm}

\begin{remark}
I recommend reading the proof with Question \ref{question:SeparatingSCSRPfromuDSR} in mind. The problem is the preservation of an arbitrary spread out set without knowing that it belongs to some $\mathcal{S}^+(\rho)$.
\end{remark}

\begin{proof}
The forcing notion in question is $\P$, the forcing to add a counterexample to $\uDSR(1,S^{\omega_2}_\omega)$ used in the proof of \cite[Theorem 3.10]{Fuchs-LambieHanson:SeparatingDSR}, where it is shown that $\P$ is $\omega_2$-strategically closed (see claim 3.11 in the proof). Let $G$ be generic for $\P$. In $\V[G]$, there is a sequence $\vA=\seq{A_\alpha}{\alpha<\omega_2}$ which witnesses the failure of $\uDSR(1,S^\lambda_\omega)$. That is, for every $\alpha<\omega_2$ of uncountable cofinality, the set $\{\xi<\alpha\st A_\xi\cap\alpha\ \text{is stationary in $\alpha$}\}$ is bounded in $\alpha$. Thus, I will be done if I can show that $\V[G]$ satisfies $\SRP(\mathcal{S}^+)$.

To this end, let $S\sub[H_\lambda]^\omega\in\mathcal{S}^+(\rho)$, for some regular $\rho>\omega_1$. Let $\vD$, $\vS$, $C$ witness this, that is, $\vD=\seq{D_i}{i<\omega_1}$ is a partition of $\omega_1$ into stationary sets, $\vS=\seq{S_{\kappa,i}}{\kappa\in K,\ i<\omega_1}$, where $K\sub\rho$ is a set of regular cardinals greater than $\omega_1$, $C\sub[H_\rho]^\omega$ is club and $S=\lifting(\vD,\vS,[H_\rho]^\omega)\cap C$.

For $\beta<\lambda$, let $A_{{\ge}\beta}=\bigcup_{\beta\le\alpha<\omega_2}A_\alpha$, and let $\bbT_\beta$ be the forcing to shoot a club through the complement of $A_{{\ge}\beta}$. By Claim 3.13 of the proof of \cite[Theorem 3.9]{Fuchs-LambieHanson:SeparatingDSR}, the forcing $\P*\dot{\bbT}_{{\ge}\beta}$ has a dense subset that's $\omega_2$-directed closed and has size $\omega_2$ in \V. It is used here that $\omega_2^{\omega_1}=\omega_2$, which follows from $\SRP(\mathcal{S}^+)+\CH$. Clearly then, $\P*\dot{\bbT}_{{\ge}\beta}$ preserves cofinalities.

Let me assume for a second that $\omega_2\in K$. By Claim 3.13 of the aforementioned proof, for every $i<\omega_1$, there is a $\beta_i<\omega_2$ such that $S_{\omega_2,i}\ohne A_{{\ge}\beta_i}$ is stationary. Let $\beta=\sup_{i<\omega_1}\beta_i$. Then $\beta<\omega_2$, and for every $i<\omega_1$, $S_{\omega_2,i}\ohne A_{{\ge}\beta}$ is stationary, which means that forcing with $\bbT_\beta$ preserves the stationarity of each $S_{\omega_2,i}$, see \cite[Lemma 3.7.(2)]{Fuchs-LambieHanson:SeparatingDSR}.

If $\omega_2\notin K$, then let $\beta=0$.

Now let $H$ be $\bbT_\beta$-generic over $\V[G]$. Since $\P*\dot{\bbT}_\beta$ has a dense subset that's $\omega_2$-directed closed, it follows from our indestructible $\SRP(\mathcal{S}^+)$ assumption in \V that $\SRP(\mathcal{S}^+)$ holds in $\V[G][H]$. Since $\P*\dot{\bbT}_\beta$ has an $\omega_2$-directed closed subset, it adds no sequences of ground model sets of length less than $\omega_2$.
In $\V[G][H]$, it is still the case that $C$ is a club subset of $[H_\rho^{\V[G]}]^\omega$, and it is still the case that $\kla{\vec{D},\vec{S}}$ is a correspondence: $\vD$ obviously is still a partition of $\omega_1$ into stationary sets. Since $\P*\bbT_{{\ge}\beta}$ preserves cofinalities, $K$ is still a set of regular cardinals greater than $\omega_1$. Now, for $\kappa\in K$ and $i<\omega_1$, $S_{\kappa,i}$ is still a stationary subset of $S^\kappa_\omega$ - in the case where $\kappa=\omega_2$, this is by our choice of $\beta$, and if $\kappa>\omega_2$, then $\bbT_\beta$ preserves the stationarity of $S_{\kappa,i}$ because it is $\kappa$-c.c.; it has size $\omega_2<\kappa$.

Let $C^*=C\projectup[H_\rho^{\V[G][H]}]^\omega$ - so $C^*$ contains a club $C'$. In $\V[G][H]$, let $S^*=\lifting(\vec{D},\vec{S},[H_\rho^{\V[G][H]}]^\omega)\cap C'$. Then $S^*\in\mathcal{S}^+(\rho)$ in $\V[G][H]$, so by $\SRP(\mathcal{S}^+)$ in $\V[G][H]$, there is a continuous $\in$-chain $\seq{x'_i}{i<\omega_1}$ through $S^*$. But then, for every $i<\omega_1$, $x_i=x'_i\cap H_\rho^\V\in S$. This is because $x'_i\in C'\sub C\projectup [H_\rho^{\V[G][H]}]^\omega$, so that $x_i\in C$, and because $x'_i\in\lifting(\vec{D},\vec{S},[H_\rho^{\V[G][H]}]^\omega)$, if we let $j=x'_i\cap\omega_1$, then $j=x_i\cap\omega_1$, and further, if $\kappa\in K\cap x_i$, then $\kappa\in K\cap x'_i$, so that $\sup x_i\cap\kappa=\sup x'_i\cap\kappa\in S_{\kappa,j}$. Now the sequence $\seq{x_i}{i<\omega_1}$ already exists in $\V[G]$, since $\bbT_\beta$ is $\omega_2$-distributive in $\V[G]$.

Thus, $\V[G]$ satisfies $\SRP(\mathcal{S}^+)$ but not $\uDSR(1,S^\lambda_\omega)$, as desired.
\end{proof}

In particular, it is not the case that $\SFP{\lambda}+\CH$ implies $\uDSR(1,S^\lambda_\omega)$, for regular $\lambda>\omega_1$. This is because $\SRP(\mathcal{S}^+)$ does imply $\SFP{\lambda}$, and because of the previous theorem.

\subsection{The \CH setting: a canonical separation at $\omega_2$}
\label{subsec:CHsetting:CanonicalButOnlyUpToOmega2}

In the present subsection, I work with subcompleteness rather than $\infty$-subcompleteness. The goal is to separate, to some degree, the subcomplete forcing axiom from the subcomplete fragment of \SRP in the presence of the continuum hypothesis. The idea is to take a consequence of \SCFA that is obtained using a subcomplete forcing that is not countably distributive (as suggested by the discussion after Theorem \ref{thm:JensenOnMutualStationarity}), and to argue that it does not follow from the subcomplete fragment of \SRP. The forcing in question is going to be Namba forcing, changing a regular cardinal to have countable cofinality. The first ingredient I will need is an iteration theorem for the subclass of subcomplete forcing notions consisting of those that \emph{preserve} uncountable cofinalities. Thus, this class excludes Namba forcing.

\begin{defn}
\label{def:UncountableCofinalityPreserving}
A forcing notion $\P$ \emph{is uncountable cofinality preserving} if for every ordinal $\eps$ of uncountable cofinality, $\forces_\P$ ``$\eps$ has uncountable cofinality.'' 
\end{defn}

In the proof of the iteration theorem, I will use the Boolean algebraic approach to forcing iterations and to revised countable support, as described in \cite{Jensen2014:SubcompleteAndLForcingSingapore}. An extensive account of the method is the manuscript \cite{VialeEtAl:BooleanApproachToSPiterations}, which also contains a proof of the following fact, originally due to Baumgartner, see \cite[Theorem 3.13]{VialeEtAl:BooleanApproachToSPiterations}:

\begin{fact}
\label{fact:lambdacc}
Let $\seq{\B_i}{i<\lambda}$ be an iteration such that for every $\alpha<\lambda$, $\B_\alpha$ is ${<}\lambda$-c.c., and such that the set of $\alpha<\lambda$ such that $\B_\alpha$ is the direct limit of $\vec{\B}\rest\alpha$ is stationary. Then the direct limit of $\vec{\B}$ is ${<}\lambda$-c.c.
\end{fact}

I will use the following iteration theorem in the separation result I am working towards, but it may be of independent interest as well.

\begin{thm}
\label{thm:IteratingSubcompleteUncountableCofinalityPreserving}
Let $\seq{\B_i}{i\le\delta}$ be an RCS iteration of complete Boolean algebras such that for all $i+1\le\alpha$, the following hold:
\begin{enumerate}[label=\textnormal{(\arabic*)}]
  \item
  \label{item:IterationIsDirect}
  $\B_i\neq\B_{i+1}$,
  \item
  \label{item:EveryIterandIsInGamma}
  $\forces_{\B_i} (\check{\B}_{i+1}/\dot{G}_{\B_i}\ \text{is subcomplete and uncountable cofinality preserving})$,
  \item
  \label{item:IntermediateCollapses}
  $\forces_{\B_{i+1}} (\delta(\check{\B}_i)\ \text{has cardinality at most}\ \omega_1)$.
\end{enumerate}
Then for every $i\le\delta$, $\B_i$ is subcomplete and uncountable cofinality preserving.
\end{thm}

\begin{proof}
I will have to use some basic facts about iterated forcing with complete Boolean algebras and revised countable support. I will state these facts as I need them. The basic setup is that $\seq{\B_i}{i\le\delta}$ is a tower of complete Boolean algebras, with projection functions $h_{j,i}:\B_j\To\B_i$ for $i\le j\le\delta$ defined by $h_{j,i}(b)=\bigwedge_{\B_i}\{a\in\B_i\st b\le_{\B_j}a\}$. Since for $i\le j_0\le j_1\le\delta$, $h_{j_1,i}\rest\B_{j_0}=h_{j_0,i}$, I'll just write $h_i$ for $h_{\delta,i}$, so that for all $i\le j\le\delta$, $h_{j,i}=h_i\rest\B_j$.

I claim that:
\begin{quote}
for every $h<\delta$, if $G_h$ is $\B_h$-generic over $\V$, then in $\V[G_h]$, for every $i\in[h,\delta]$, $\B_i/G_h$ is subcomplete and uncountable cofinality preserving.
\end{quote}

Jensen's proofs show that in this situation, $\B_i/G_h$ is subcomplete (see \cite[\S2, Thm.~1, pp.~3-24]{Jensen:IterationTheorems}), so that I can focus on the preservation of uncountable cofinality.

If not, let $\delta$ be a minimal such that there is an iteration of length $\delta$ that forms a counterexample.

Since the successor case is trivial, let me focus on the case that $\delta$ is a limit ordinal. Let $\eps$ be an ordinal of uncountable cofinality in $\V[G_h]$. Note that by minimality of $\delta$, this is equivalent to saying that $\eps$ has uncountable cofinality in $\V$.
I will show that in $\V[G_h]$, $\B_\delta/G_h$ still forces that $\check{\eps}$ has uncountable cofinality.

\noindent{\bf Case 1:} there is an $i<\delta$ such that $\cf^{\V[G_h]}(\eps)\le\delta(\B_i)$.

Then by \ref{item:IntermediateCollapses}, in $\V[G_h]$, $\forces_{\B_{i+1}/G_h}\cf(\check{\eps})\le\omega_1$. But by minimality of $\delta$, we also know that $\forces_{\B_{i+1}}\cf(\check{\eps})\ge\omega_1$. Thus, $\B_{i+1}/G_h$ forces over $\V[G_h]$ that the cofinality of $\eps$ is $\omega_1$. But the tail of the iteration is subcomplete, and hence it preserves $\omega_1$. It follows that the cofinality of $\eps$ stays uncountable in $\V[G_h]^{\B_\delta/G_h}$.

\noindent{\bf Case 2:} case 1 fails. So for all $i<\delta$, $\cf^{\V[G_h]}(\eps)>\delta(\B_i)$.

Let $\kappa=\cf^{\V[G_h]}(\eps)$.
Note that for $i<\delta$, $\B_i$ is $\delta(\B_i)^+$-c.c., so since $\kappa>\delta(\B_i)$, it follows that $\kappa$ is a regular uncountable cardinal greater than $\delta(\B_i)$ in $\V^{\B_i}$.
We may also assume that $\kappa>\omega_1$, for if $\kappa=\omega_1$, then it will still have uncountable cofinality in $\V[G_h]^{\B_\delta/G_h}$, since $\B_\delta/G_h$ is subcomplete in $\V[G_h]$ and hence preserves $\omega_1$, as in case 1.

\noindent{\bf Case 2.1:} there is an $i<\delta$ such that $\cf^{\V[G_h]}(\delta)\le\delta(\B_i)$.

Then, since $\B_{i+1}$ collapses $\delta(\B_i)$ to $\omega_1$, it follows that in $\V[G_h]^{\B_{i+1}/G_h}$, $\delta$ has cofinality at most $\omega_1$. For notational simplicity, let me pretend that $h=0$, so that I can write $\V^{\B_{i+1}}$ in place of  $\V[G_h]^{\B_{i+1}/G_h}$, for example. So from what I said above, $\kappa$ is a regular uncountable cardinal in what we call $\V$ now, and it is greater than each $\delta(\B_j)$. So It suffices to show that $\B_\delta$ preserves the fact that $\kappa$ has uncountable cofinality. I'm using here that the tail of the iteration in $\V[G_h]$ is a revised countable support iteration.

Towards a contradiction, let $\dot{f}$ be a $\B_\delta$-name, and let $a_0\in\B_\delta$ be a condition such that $a_0$ forces that $\dot{f}$ is a function from $\omega$ to $\kappa$ whose range is unbounded in $\kappa$. I will find a condition extending $a_0$ that forces that the range of $\dot{f}$ is bounded in $\kappa$, a contradiction.

Note that if $\cf(\delta)=\omega_1$, then for every $i<\delta$, $\forces_{\B_i}\cf(\check{\delta})=\check{\omega}_1$, as $\B_i$ preserves $\omega_1$. Thus, $\B_\delta$ is the direct limit of $\vec{\B}\rest\delta$, that is, $\bigcup_{i<\delta}\B_i$ is dense in $\B_\delta$ in this case. Let me denote this dense set by $X$.

If, on the other hand, $\cf(\delta)=\omega$, then since $\vec{\B}$ is an RCS iteration, the set $\{\bigwedge_{i<\delta}t_i\st\seq{t_i}{i<\delta}\ \text{is a thread in}\ \vec{\B}\rest\delta\}$ is dense in $\B_\delta$ -- here, $\vec{t}$ is a thread in $\vec{\B}\rest\delta$ if for all $i\le j<\delta$, $0\neq t_i=h_i(t_j)$.  In case $\cf(\delta)=\omega$, let $X$ be that dense subset of $\B_\delta$. For more background on threads in RCS iterations, I refer the reader to \cite[p.~124]{Jensen2014:SubcompleteAndLForcingSingapore}.

Let $\pi:\omega_1\To\delta$ be cofinal, with $\pi(0)=0$.

Let $N=L_\tau^A$ with $H_\theta\cup\kappa+1\sub N$, where $\theta$ verifies the subcompleteness of each $\B_i$, for $i\le\delta$ and for each $\B_j/G_i$ whenever $i<j<\delta$ and $G_i$ is $\B_i$-generic.
Let $S=\kla{\kappa,\eps,\delta,\vec{\B},\dot{f},a_0,\pi}$.
Let $\sigma_0:\bN\prec N$, where $\bN$ is transitive, countable and full, and $S\in\ran(\sigma_0)$. Let $\sigma_0(\bS)=S$, and let $\bS=\kla{\bkappa,\bar{\eps},\bdelta,\vec{\bar{\B}},\dot{\barf},\bar{a}_0,\bpi}$.
Let $\tilde{\kappa}=\sup\sigma_0``\kappa<\kappa$, and let us also fix an enumeration $\bN=\{e_n\st n<\omega\}$. Let $\bG\sub\bar{\B}_{\bdelta}$ be generic over $\bN$, with $\bar{a}_0\in\bG$.

Let $\seq{\nu_n}{n<\omega}$ be a sequence of ordinals $\nu_n<\omega_1^\bN$ such that if we let $\bgamma_n=\bpi(\nu_n)$, it follows that $\seq{\bgamma_n}{n<\omega}$ is cofinal in $\bdelta$, and such that $\nu_0=0$, so that $\bgamma_0=0$. Hence, letting $\gamma_n=\sigma_0(\bgamma_n)$, we have that $\sup_{n<\omega}\gamma_n=\sup\ran(\sigma_0\cap\delta)=\tdelta$. Moreover, whenever $\sigma':\bN\prec N$ is such that $\sigma'(\bpi)=\pi$, it follows that for every $n<\omega$, $\sigma'(\bgamma_n)=\gamma_n=\pi(\nu_n)$, since $\sigma'(\bgamma_n)=\sigma'(\bpi(\nu_n)=\sigma'(\bpi)(\nu_n)=\pi(\nu_n)$.

By induction on $n<\omega$, construct sequences $\seq{\dot{\sigma}_n}{n<\omega}$, $\seq{c_n}{n<\omega}$, 
with $c_n\in\B_{\gamma_n}$, $\dot{\sigma}_n\in\V^{\B_{\gamma_n}}$, 
such that for every $n<\omega$, $c_n$ forces the following statements with respect to $\B_{\gamma_n}$:
\begin{enumerate}
\item $\dot{\sigma}_n:\check{\bN}\prec\check{N}$,
\item $\dot{\sigma}_n(\check{\bS})=\check{S}$, and for all $k<n$, 
    $\dot{\sigma}_n(\check{e}_k)=\dot{\sigma}_k(\check{e}_k)$,
\item $\dot{\sigma}_n``\check{\bG}\cap\check{\bar{\B}}_{\bgamma_n}\sub\dot{G}_{\B_{\gamma_n}}$,
\item $\sup\dot{\sigma}_n``\check{\bar{\kappa}}=\check{\tkappa}$,
\item $c_{n-1}=h_{\gamma_{n-1}}(c_{n})$ (for $n>0$).
\end{enumerate}
To start off, in the case $n=0$, we set $c_0=\eins$, $\dot{\sigma}_0=\check{\sigma}_0$. 

Now suppose $\seq{\dot{\sigma}_n}{n\le n}$, $\seq{c_m}{m\le n}$ have been constructed, so that the above conditions are satisfied so far. Let $G_{\gamma_n}$ be $\B_{\gamma_n}$-generic with $c_n\in G_{\gamma_n}$, and work in $\V[G_n]$ temporarily. Let $\sigma_n=\dot{\sigma}_n^{G_{\gamma_n}}$. Then $\sigma_n$ extends to $\sigma^*_n:\bN[\bG_{\bgamma_n}]\prec N[G_{\gamma_n}]$ such that $\sigma^*_n(\bG_{\bgamma_n})=G_n$. Since $\theta$ verifies the subcompleteness of $\B'=\B_{\gamma_{n+1}}/G_{\gamma_n}$, and since $\bN[\bG_{\bgamma_n}]$ is full, there is a condition $c'\in(\B')^+$ such that whenever $G'$ is $\B'$-generic over $\V[G_n]$ with $c'\in G'$, there is a $\sigma':\bN[\bG_{\bgamma_n}]\prec N[G_n]$ such that $\sigma'(\bG_{\bgamma_n})=G_{\gamma_n}$, $\sigma'(\bar{S})=S$ and $\sigma'(e_k)=\sigma^*_n(e_k)=\sigma_n(e_k)$ for $k\le n$, $(\sigma')``\bar{G}\cap\bar{\B}_{\bar{G}}\sub G'$, where $\bG'=\bG_{\bgamma_{n+1}}/\bar{G}_{\bgamma_n}$, and $\sup(\sigma')``\bkappa=\tkappa=\sup(\sigma^*``)\bkappa$. The point is that $\kappa>\delta(\B_{\gamma_{n+1}})$, so that the suprema condition can be employed in order to ensure this last point.

Since the situation described in the previous paragraph arises whenever $G_{\gamma_n}$ is generic with $c_n\in G_{\gamma_n}$, it is forced by $c_n$, and there are $\B_{\gamma_n}$-names $\dot{c}'$, $\dot{\sigma}_{n+1}$ for the condition $c'$ and the restriction of the embedding $\sigma'$ to $\bN$.
We may choose the name $\dot{c}'$ in such a way that $\forces_{\B_{\gamma_n}}\dot{c}'\in\check{\B}_{\gamma_{n+1}}/\dot{G}_{\B_{\gamma_n}}$ and $c_n=\BV{\dot{c}'\neq 0}_{\B_{\gamma_n}}$. Namely, given the original $\dot{c}'$ such that $c_n$ forces that $\dot{c}'\in(\check{\B}_{\gamma_{n+1}}/\dot{G}_{\B_{\gamma_n}})^+$ and all the other statements listed above, there are two cases: if $c_n=\eins_{\B_{\gamma_n}}$, then since $c_n\le\BV{\dot{c}'\neq 0}$, it already follows that $c_n=\BV{\dot{c}'\neq 0}$. If $c_n<\eins_{\B_{\gamma_n}}$, then let $\dot{e}\in\V^{\B_{\gamma_n}}$ be a name such that $\forces_{\B_{\gamma_n}}\dot{e}=0_{\check{\B}_{\gamma_{n+1}}/\dot{G}_{\B_{\gamma_n}}}$, and mix the names $\dot{c}'$ and $\dot{e}$ to get a name $\dot{d}'$ such that $c_n\forces_{\B_{\gamma_n}}\dot{d}'=\dot{c}'$ and $\neg c_n\forces_{\B_{\gamma_n}}\dot{d}'=\dot{e}$. Then $\dot{d}'$ is as desired. Clearly, $\forces_{\B_{\gamma_n}}\dot{d}'\in\check{\B}_{\gamma_{n+1}}/\dot{G}_{\B_{\gamma_n}}$. Since $c_n\forces_{\B_{\gamma_n}}\dot{d}'=\dot{c}'$, it follows that $c_n\le\BV{\dot{d}'\neq 0}$, and since $\neg c_n\forces_{\B_{\gamma_n}}\dot{d}'=\dot{e}$, it follows that $\neg c_n\le\BV{\dot{d}'=0}=\neg\BV{\dot{d}'\neq 0}$, so $\BV{\dot{d}'\neq 0}\le c_n$. So we could replace $\dot{c}'$ with $\dot{d}'$.

So let us assume that $\dot{c}'$ already has this property, that is, $c_n=\BV{\dot{c}'\neq 0}$.
Then, by \cite[\S0, Fact 4]{Jensen2014:SubcompleteAndLForcingSingapore},%
\footnote{That fact should read: ``Let $\mathbb{A}\sub\B$, and let $\forces_{\mathbb{A}}\dot{b}\in\check{\B}/\dot{G}_{\mathbb{A}}$, where $\dot{b}\in\V^{\mathbb{A}}$. There is a unique $b\in\B$ such that $\forces_{\mathbb{A}}\dot{b}=\check{b}/\dot{G}_{\mathbb{A}}$.'' That's what the proof given there shows.} %
there is a unique $c_{n+1}\in\B_{\gamma_{n+1}}$ such that $\forces_{\B_{\gamma_n}}\check{c}_{n+1}/\dot{G}_{\B_{\gamma_n}}=\dot{c}'$, and it follows by \cite[\S0, Fact 3]{Jensen2014:SubcompleteAndLForcingSingapore} that
\[h(c_{n+1})=\BV{\check{c}_{n+1}/\dot{G}_{\B_{\gamma_n}}\neq 0}_{\B_{\gamma_n}}=\BV{\dot{c}'\neq 0}_{\B_{\gamma_n}}=c_n\]
as wished.

This finishes the construction of $\seq{\dot{\sigma}_n}{n<\omega}$ and $\seq{c_n}{n<\omega}$.

Now, the sequence $\seq{c_n}{n<\omega}$ is a thread, and so, $c=\bigwedge_{n<\omega}c_n\in\B_\delta^+$.
Let $G$ be $\B_\delta$-generic with $c\in G$. In $\V[G]$, let
\[\sigma=\bigcup_{n<\omega}\dot{\sigma}_n^{G\cap\B_{i_n}}\rest\{e_0,\ldots,e_n\}.\]
Jensen's arguments then show that $\sigma:\bN\prec N$, and $\sigma``\bG\sub G$. For the latter, we can argue as in \cite[p.~141, proof of (d)]{Jensen2014:SubcompleteAndLForcingSingapore}: clearly, $\sigma``\bG\cap\bar{\B}_{\bar{\gamma}_n}\sub G$, for every $n<\omega$, since if $\bar{a}\in\bG\cap\bar{\B}_{\bar{\gamma}_n}$, then for some $m\ge n$, $\sigma(a)=\sigma_m(a)\in G\cap\B_{\gamma_m}\sub G$ (letting $a=e_k$, this is true whenever $m\ge\max(n,k)$).
If $\cf(\delta)=\omega_1$, then this implies directly that $\sigma``\bG\sub G$, because then, in $\bN$, $\bar{\B}_{\bdelta}$ is the direct limit of $\vec{\B}\rest\bdelta$, so $\bigcup_{i<\bdelta}\bar{\B}_i$ is dense in $\bar{\B}_\delta$. So if $\bar{a}\in\bG$, me may assume that $\bar{a}\in\bG\cap\B_{\bgamma_n}$, for some $n<\omega$, so that $a\in G\cap\B_{\gamma_n}$.

If $\cf(\delta)=\omega$, that is, in $\bN$, $\cf(\bdelta)=\omega$, then $\B_{\bdelta}$ is the inverse limit of $\seq{\bar{\B}_i}{i<\bdelta}$. Hence, letting $\bar{a}\in\bG$, we may assume that $\bar{a}=\bigwedge_{i<\bdelta}\bar{a}_i$, where $\seq{\bar{a}_i}{i<\bdelta}$ is a thread in $\seq{\bar{\B}_i}{i<\bdelta}$, since the set of such conditions is dense in $\B_{\bdelta}^+$. Let $\sigma(\vec{\bar{a}})=\vec{a}=\seq{a_i}{i<\delta}$. Then $\vec{a}$ is a thread in $\seq{\B_i}{i<\delta}$. Moreover, for each $n<\omega$, $\sigma(\bar{a}_{\bgamma_n})=a_{\gamma_n}\in G$. Thus, $\sigma(\bar{a})=\bigwedge_{n<\omega}a_{\gamma_n}\in G$, by the completeness of $G$ (since $\seq{a_{\gamma_n}}{n<\omega}\in\V$, as $\vec{a}\in N\sub\V$ and $\seq{\gamma_n}{n<\omega}\in\V$.)

Thus, $\sigma$ lifts to an elementary embedding $\sigma^*:\bN[\bG]\prec N[G]$. Note that $\bar{a}_0\in\bG$, and so, $a_0=\sigma^*(\bar{a}_0)\in G$. Moreover, $\sigma``\bkappa\sub\tkappa$, because if $\xi<\bkappa$, then for some $n<\omega$, $\xi=e_n$, and so, $\sigma(\xi)=\sigma(e_n)=\sigma_n(e_n)=\sigma_n(\xi)<\tkappa$, since $\sigma_n``\bkappa\sub\tkappa$. But then, it follows that $\ran(\dot{f}^G)\sub\tkappa<\kappa$, because $\ran(\dot{f}^G)=\sigma``\ran(\dot{\barf}^\bG)\sub\sigma``\bkappa\sub\tkappa<\kappa$. This contradicts the fact that $a_0\in G$ and $a_0$ forces that the range of $\dot{f}$ is unbounded in $\kappa$.

\noindent{\bf Case 2.2:} for all $i<\delta$, $\cf(\delta)>\delta(\B_i)$.

We may also assume that $\cf(\delta)>\omega_1$, for otherwise, $\cf(\delta)\le\omega_1$ and the argument of case 2.1 goes through.

It follows as in \cite[p.~143, claim (2)]{Jensen2014:SubcompleteAndLForcingSingapore} that for $i<\delta$, $\card{i}\le\delta(\B_i)$. But then, $\delta$ must be regular, for otherwise, if $i=\cf(\delta)<\delta$, it would follow that $\cf(\delta)=i\le\delta(\B_i)<\cf(\delta)$.

Thus, $\delta$ is a regular cardinal, and $\delta\ge\omega_2$. Hence, $S^\delta_{\omega_1}$, the set of ordinals less than $\delta$ with cofinality $\omega_1$, is stationary in $\delta$. For $\gamma\in S^\delta_{\omega_1}$, since $\B_\gamma$, being subcomplete, preserves $\omega_1$, it follows that for every $i<\gamma$, $\forces_{\B_i}\cf(\check{\gamma})>\omega$. Thus, $\B_\gamma$ is the direct limit of $\vec{\B}\rest\gamma$.
Moreover, since for $i<\delta$, $\delta(\B_i)<\delta=\cf(\delta)$, it follows that $\B_i$ is $\delta(\B_i)^+$-c.c., and hence $\delta$-c.c.

It follows by Fact \ref{fact:lambdacc} that the direct limit of $\vec{\B}\rest\delta$ is $\delta$-c.c.

Again, since for all $i<\delta$, $\B_i$ is $\delta$-c.c., $\B_i$ forces that the cofinality of $\delta$ is uncountable, so that $\B_\delta$ is the direct limit of $\vec{\B}\rest\delta$, which is $\delta$-c.c., as we have just seen. So $\B_\delta$ is $\delta$-c.c.

Let $G$ be $\B_\delta$-generic over $\V$, and suppose that in $\V[G]$, $\kappa$ has countable cofinality. Since $\B_\delta$ is $\delta$-c.c., it must be that $\kappa<\delta$. But letting $f:\omega\To\kappa$ be cofinal, $f\in\V[G]$, it then follows that $f\in\V[G\cap\B_i]$, for some $i<\delta$. This contradicts the minimality of $\delta$ and completes the proof.
\end{proof}

Looking back, it turns out that the theorem shows iterability with countable support.

\begin{cor}
\label{cor:IteratingSubcompleteUncountableCofinalityPreservingWithCS}
Let $\seq{\B_i}{i\le\delta}$ be a countable support iteration of complete Boolean algebras such that for all $i+1\le\alpha$, the following hold:
\begin{enumerate}[label=\textnormal{(\arabic*)}]
  \item
  \label{item:IterationIsDirect}
  $\B_i\neq\B_{i+1}$,
  \item
  \label{item:EveryIterandIsInGamma}
  $\forces_{\B_i} (\check{\B}_{i+1}/\dot{G}_{\B_i}\ \text{is subcomplete and uncountable cofinality preserving})$,
  \item
  \label{item:IntermediateCollapses}
  $\forces_{\B_{i+1}} (\delta(\check{\B}_i)\ \text{has cardinality at most}\ \omega_1)$.
\end{enumerate}
Then for every $i\le\delta$, $\B_i$ is subcomplete and uncountable cofinality preserving.
\end{cor}

\begin{proof}
It is easy to see by induction on $i\le\delta$ that $\B_i$ is subcomplete and uncountable cofinality preserving, and that if $i$ is a limit ordinal, then $\B_i$ is (isomorphic to) the rcs limit of $\vec{\B}\rest i$.
The successor case of the induction is trivial, so let $i\le\delta$ be a limit ordinal. If we take $\B'_i$ to be the revised countable support limit of $\vec{\B}\rest i$, then the resulting iteration $\vec{\B}\rest i\verl\B'_i$ is an rcs iteration, because inductively, $\vec{\B}\rest i$ is. Hence, by the theorem, $\B'_i$ is subcomplete and uncountable cofinality preserving. But moreover, by the proof of the theorem, whenever $h\le k<i$ and $G_h$ is $\B_h$-generic, it follows that $\B_k/G_h$ is uncountable cofinality preserving in $\V[G_h]$. It follows that the rcs limit of $\vec{\B}\rest i$ is the same as the countable support limit, and hence that $\B_i=\B'_i$ (modulo isomorphism). So $\B_i$ is subcomplete and uncountable cofinality preserving.
\end{proof}

The second set of ingredients I need is centered around bounded forcing axioms and their consistency strengths. These axioms were introduced in \cite{GoldsternShelah:BPFA} as follows, albeit with a different notation.

\begin{defn}
\label{def:BoundedForcingAxioms}
Let $\Gamma$ be a class of forcings, and let $\kappa,\lambda$ be cardinals. Then $\BFA(\Gamma,{\le}\kappa,{\le}\lambda)$ is the statement that if $\P$ is a forcing in $\Gamma$, $\B$ is its complete Boolean algebra, and $\mathcal{A}$ is a collection of at most $\kappa$ many maximal antichains in $\B$, each of which has size at most $\lambda$, then there is a $\mathcal{A}$-generic filter in $\B$, that is, a filter that intersects each antichain in $\mathcal{A}$. When $\kappa=\omega_1$, then I usually don't mention $\kappa$, that is, $\BFA(\Gamma,{\le}\lambda)$ is short for $\BFA(\Gamma,{\le}\omega_1,{\le}\lambda)$. And when $\kappa=\lambda=\omega_1$, then the resulting principle is abbreviated to $\BFA(\Gamma)$.
If $\Gamma$ is the class of subcomplete forcings, then I write $\BSCFA$ for $\BFA(\SC)$ and $\BSCFA({\le}\lambda)$ for $\BFA(\SC,{\le}\lambda)$. Similarly, if $\Gamma$ is the class of proper forcing, then $\BPFA$ denotes $\BFA(\Gamma)$, and similarly for $\BPFA({\le}\lambda)$.
\end{defn}

Bounded forcing axioms can be expressed as generic absoluteness principles as follows.

\begin{thm}[Bagaria {\cite[Thm.~5]{Bagaria:BFAasPrinciplesOfGenAbsoluteness}}]
\label{thm:BagariaCharacterizationOfBFA}
Let $\kappa$ be a cardinal of uncountable cofinality, and let $\P$ be a poset. Then $\BFA(\{\P\},{\le}\kappa,{\le}\kappa)$ is equivalent to $\Sigma_1(H_{\kappa^+})$-absoluteness for $\P$. The latter means that whenever $g$ is $\P$-generic, $\phi(x)$ is a $\Sigma_1$-formula and $a\in H_{\kappa^+}$, then $\V\models\phi(a)$ iff $\V[g]\models\phi(a)$.
\end{thm}

For any class $\Gamma$ of forcings, the principles $\BFA({\le}\kappa)$ give closer and closer approximations to $\FA(\Gamma)$, as $\kappa$ increases; in fact, $\FA(\Gamma)$ is $\BFA(\Gamma,{\le}\infty)$, or, for all $\kappa$, $\BFA({\le}\kappa)$.
The following characterization of these axioms is easily seen to be equivalent to the one given in \cite[Thm.~1.3]{ClaverieSchindler:AxiomStar}, see also \cite{BagariaGitmanSchindler:RemarkableWeakPFA}.

\begin{fact}
\label{fact:CharacterizationOfBFAkappa}
  $\BFA(\{\Q\},{\le}\kappa)$ is equivalent to the following statement:
  if $M=\langle|M|,\in,\langle R_i\st i<\omega_1\rangle\rangle$ is a transitive model for the language of set theory with $\omega_1$ many predicate symbols $\seq{\dot{R}_i}{i<\omega_1}$, of size $\kappa$, and $\phi(x)$ is a $\Sigma_1$-formula, such that $\forces_\Q\phi(\check{M})$, then there are in $\V$ a transitive $\bM=\kla{|\bM|,\in,\seq{\bar{R}_i}{i<\omega_1}}$ and an elementary embedding $j:\bM\prec M$ such that $\phi(\bM)$ holds.
\end{fact}

Miyamoto has analyzed the strength of these principles for proper forcing and introduced the following large cardinal concept, with slightly different terminology.

\begin{defn}[{\cite[Def.~1.1]{Miyamoto:WeakSegmentsOfPFA}}]
\label{def:LocalizedReflectingCardinals}
Let $\kappa$ be a regular cardinal, $\alpha$ an ordinal, and $\lambda=\kappa^{+\alpha}$. Then $\kappa$ is $H_\lambda$-reflecting, or I will say ${+\alpha}$-reflecting, iff for every $a\in H_\lambda$ and any formula $\phi(x)$, the following holds: if there is a cardinal $\theta$ such that $H_\theta\models\phi(a)$, then the set of $N\prec H_\lambda$ such that
\begin{enumerate}
\item $N$ has size less than $\kappa$,
\item $a\in N$,
\item if $\pi_N:N\To H$ is the Mostowski-collapse of $N$, then there is a cardinal $\btheta<\kappa$ such that $H_\btheta\models\phi(\pi_N(a))$
\end{enumerate}
is stationary in $\power_\kappa(H_\lambda)$.
\end{defn}

The concept of a reflecting cardinal was previously introduced in \cite{GoldsternShelah:BPFA}, and is contained in this definition, as it is not hard to see that being reflecting is equivalent to being $+0$-reflecting. The $+1$-reflecting cardinals are also known as strongly unfoldable cardinals, introduced independently in \cite{Villaveces:StrongUnfoldability}. In the context of bounded forcing axioms, it seems to make the most sense to emphasize that they generalize reflecting cardinals, as it was shown in \cite{GoldsternShelah:BPFA} that the consistency strength of $\BPFA$ is precisely a reflecting cardinal, and it was shown in \cite[Def.~1.1]{Miyamoto:WeakSegmentsOfPFA} that the consistency strength of $\BPFA({\le}\omega_2)$ is a ${+}1$-reflecting cardinal. I showed that the same consistency strength results hold for $\BSCFA$ and $\BSCFA({\le}\omega_2)$ as well, in \cite{Fuchs:HierarchiesOfForcingAxioms}. It will be important for the upcoming argument that ${+}1$-reflecting cardinals may exist in $L$; in fact, I showed in \cite{Fuchs:HierarchiesOfForcingAxioms} that if $\BSCFA({\le}\omega_2)$ holds, then the cardinal $\omega_2^\V$ is ${+}1$-reflecting in $L$.

Given this background, it is easy to observe that $\SC$-$\SRP(\omega_2)+\neg\CH$ does not imply $\BSCFA({\le}\omega_2)$, since the consistency strength of $\SC$-$\SRP(\omega_2)+\neg\CH$ is equal to that of $\ZFC$ (in fact, $\neg\CH$ implies $\SC$-$\SRP(\omega_2)$ by Observation \ref{obs:SCSRPtrivialuptocontinuum}), while the consistency strength of $\BSCFA({\le}\omega_2)$ is a ${+}1$-reflecting cardinal. I have now assembled the main tools needed to show that such a separation can also be arranged when \CH holds. This will be achieved by constructing a model of $\CH+\SC\text{-}\SRP(\omega_2)$ in which the consequence of $\BSCFA({\le}\omega_2)$ stated in the following lemma fails (in an extreme way).

\begin{lem}
\label{lem:StationarilyManyL-RegularsHaveCountableCofinality}
Assume $\BSCFA({\le}\omega_2)$. Then the set
\[\{\alpha<\omega_2\st\cf(\alpha)=\omega\ \text{and $\alpha$ is regular in $L$}\}\]
is stationary in $\omega_2$.
\end{lem}

\begin{proof}
If $0^\#$ exists, then, letting $I$ be the class of Silver indiscernibles, $I\cap\omega_2$ is club in $\omega_2$ and consists of $L$-regular cardinals. Therefore, $I\cap S^{\omega_2}_\omega$ is a stationary subset of the set in question.

So let me assume that $0^\#$ does not exist.
The following argument traces back to \Todorcevic (unpublished), but see \cite[Lemma 2.4]{Aspero:MaximalBFA}. A variant of the argument was used in \cite[Lemma 4.11]{Fuchs:HierarchiesOfForcingAxioms}.

Let $\kappa=\omega_2$, and let $\gamma$ be some singular strong limit cardinal, and let $\theta=\gamma^+=(\gamma^+)^L$, by Jensen's covering lemma. Let $E\sub\kappa$ be some club subset. Let $X\sub H_{\omega_2}$ have cardinality $\omega_2$, $\omega_2\sub X$, and $H_{\omega_2}|X\prec H_{\omega_2}$. Let   $M=\kla{X,\in,E,0,\ldots,\xi,\ldots,\omega_1}_{\xi<\omega_1}$. So the universe of $M$ has size $\omega_2$.

Let $\seq{C_\xi}{\xi\ \text{is a singular ordinal in $L$}}$ be the canonical global $\square$ sequence for $L$ of Jensen \cite{FS}. It is $\Sigma_1$-definable in $L$ and has the properties that for every $L$-singular ordinal $\xi$, the order type of $C_\xi$ is less than $\xi$, and if $\zeta$ is a limit point of $C_\xi$, then $\zeta$ is singular in $L$ and $C_\zeta=C_\xi\cap\zeta$.

Let $B=\{\xi<\theta\st\kappa<\xi<\theta\ \text{and}\ \cf(\xi)=\omega\}$. Note that by covering, every $\xi\in B$ is singular in $L$, since a countable cofinal subset of $\xi$ in $\V$ can be covered by a set in $L$ of cardinality at most $\omega_1$, so that its order type will be less than $\kappa$, and hence less than $\xi$. So $C_\xi$ is defined for every $\xi\in B$, and since the function $\xi\mapsto\otp(C_\xi)$ is regressive, there is a stationary subset $A$ of $B$ on which this function is constant.

Let $g\sub\kappa$ be an $\omega$-sequence cofinal in $\kappa$, added by Namba forcing, which is subcomplete. Since Namba forcing certainly has cardinality less than $\theta$, $A$ remains stationary in $\V[g]$. Working in $\V[g]$ now, since $A$ consists of ordinals of cofinality $\omega$ and is stationary in a regular cardinal greater than $2^\omega$, the forcing $\P_A$, which adds a subset $F$ of $A$ that's closed and unbounded in $\theta$ and has oder type $\omega_1$, by forcing with closed initial segments, is subcomplete -- this follows from Lemma \ref{lem:OneStep}, see also \cite[p.~134ff., Lemma 6.3]{Jensen2014:SubcompleteAndLForcingSingapore}, where the assumption that $\theta>2^\omega$ is omitted. Let $h$ be generic over $\V[g][F]$ for $\Col{\omega_1,M}$. Clearly, the composition of Namba forcing, $\P_A$ and $\Col{\omega_1,M}$ is subcomplete.

In $\V[g][F][h]$, the $\Sigma_1$-statement $\Phi(M)$ saying ``there are an ordinal $\theta'>\On\cap M$, sets $g'$ and $F'$, and
a function $h'$ such that $h'$ is a surjection from $\omega_1^M$ onto the universe of $M$, $F'$ is a club in $\theta'$ of order type $\omega_1^M$ such that for all $\xi,\zeta\in F'$, $\otp(C_{\xi})=\otp(C_{\zeta})$, and $g'$ is a cofinal subset of $\On\cap M$ of order type $\omega$, and in $L_{\theta'}$, $\On\cap M$ is a regular cardinal greater than $\omega_1^M$'' holds, as witnessed by $\theta$, $g$, $F$ and $h$. It is important here that the definition of the canonical global $\square$ sequence is $\Sigma_1$. This does not depend on the particular choice of the generics $g$, $F$ and $h$, which means that it is forced by the trivial condition in the composition of these subcomplete forcings that $\Phi(\check{M})$ holds. So according to the characterization of  $\BSCFA({\le}\omega_2)$ given by Fact \ref{fact:CharacterizationOfBFAkappa}, there are a transitive $\bM=\kla{|\bM|,\in,\bar{E},\seq{\xi}{\xi<\omega_1}}$ such that $\Phi(\bM)$ holds, and an elementary embedding $j:\bM\prec M$. It follows that $j$ is the identity. This is because we have constants for the countable ordinals, so that $\omega_1^\bM=\omega_1^M=\omega_1$, and since $M$ believes that the transitive closure of any set has cardinality at most $\omega_1$.

Let $\bkappa=\On^\bM$, and let $\btheta$, $\bar{g}$, $\bar{F}$, $\bar{h}$ witness that $\Phi(\bM)$ holds. Then $\bar{h}:\omega_1\To|\bM|$ is onto, so $\bkappa<\kappa$.
Moreover, since $\bM\in H_{\omega_2}$, $\btheta$ may be chosen to be less than $\omega_2$.

Note that $\bkappa\in E$. This is because $\bar{E}=\bkappa\cap E$, and by elementarity, $\bar{E}$ is unbounded in $\bkappa$, so since $E$ is closed in $\kappa$, it follows that $\bkappa\in E$. Moreover, $\bkappa$ has countable cofinality, since $\bar{g}$ is a cofinal subset of $\bkappa$ of order type $\omega$.
I claim that $\bkappa$ has uncountable cofinality in $L$, thus completing the proof.

The key point is that $\btheta$ is a regular cardinal in $L$. To see this, assume that $\btheta$ is singular in $L$. Then $C_\btheta$ is defined.
Note that $\cf(\btheta)=\omega_1$, since $\otp(\bar{F})=\omega_1$. So, letting $C_\btheta'$ be the set of limit points of $C_\btheta$, $C_\btheta'\cap\bar{F}$ is club in $\btheta$. Now take $\xi<\zeta$, both in $C_\btheta'\cap\bar{F}$. Then, since $\xi,\zeta\in\bar{F}$, $C_\xi$ and $C_\zeta$ have the same order type, but since both $\xi$ and $\zeta$ are limit points of $C_\btheta$, $C_\xi=C_\btheta\cap\xi$, which is a proper initial segment of $C_\zeta=C_\btheta\cap\zeta$, a contradiction.

So, since by $\Phi(\bM)$, $\bkappa$ is a regular cardinal in $L_\btheta$, it follows that $\bkappa$ is an uncountable regular cardinal in $L$, as wished.
\end{proof}

\begin{lem}
\label{lem:BFAimpliesSRPatOmega_2}
Let $\Gamma$ be a forcing class.
\begin{enumerate}[label=\textnormal{(\arabic*)}]
  \item
  \label{item:InGeneralNeedCardArithmetic}
  $2^{\omega_1}=\omega_2 + \BFA(\Gamma,{\le}\omega_2)$ implies $\Gamma$-$\SRP(\omega_2)$.
  \item
  \label{item:ForSCcapUCPdontNeedIt}
  Let $\Gamma$ be the class of all subcomplete, uncountable cofinality preserving forcing notions. Then $\BFA(\Gamma,{\le}\omega_2)$ implies $\SC$-$\SRP(\omega_2)$.
\end{enumerate}
\end{lem}

\begin{proof}
To prove \ref{item:InGeneralNeedCardArithmetic}, first note that $H_{\omega_2}$ has cardinality $\omega_2$. To see that $\Delta$-$\SRP(\omega_2)$ holds, let $S\sub[H_{\omega_2}]^\omega$ be $\Delta$-projective stationary. I will use the characterization of $\BFA(\Delta,{\le}\omega_2)$ given by Fact \ref{fact:CharacterizationOfBFAkappa}. Thus, let $M=\kla{H_{\omega_2},\in,S,0,1,\ldots,\xi,\ldots}_{\xi<\omega_1}$. Now the forcing $\P_S$ is in $\Gamma$, since $S$ is $\Delta$-projective stationary, and the size of the universe of $M$ is $\omega_2$, so that Fact \ref{fact:CharacterizationOfBFAkappa} is applicable. Namely, the $\Sigma_1$-statement expressing that there is a continuous $\omega_1^M$-chain of models through $\dot{S}^M$ is forced by $\P_S$. So, by the fact, there are a transitive model $\bM$ of the same language as $M$, and an elementary embedding $j:\bM\prec M$ such that the same $\Sigma_1$-statement is true of $\bM$. As in the proof of Lemma \ref{lem:StationarilyManyL-RegularsHaveCountableCofinality}, $j$ is the identity.

To see \ref{item:ForSCcapUCPdontNeedIt}, let $\Gamma$ be the class of all subcomplete, uncountable cofinality preserving forcing notions, and suppose that $\BFA(\Gamma,{\le}\omega_2)$ holds. If $\CH$ fails, then \SC-$\SRP(\omega_2)$ holds trivially, by Observation \ref{obs:SCSRPtrivialuptocontinuum}. So let me assume \CH. It then follow from $\BFA(\Gamma,{\le}\omega_2)$ that $\SFP{\omega_2}$ holds (see Definition \ref{def:SFPkappa}). To see this, let $\vD=\seq{D_i}{i<\omega_1}$ be a partition of $\omega_1$ into stationary sets, and let $\seq{S_i}{i<\omega_1}$ be a sequence of stationary subsets of $S^{\omega_2}_\omega$. By (the remark after) Lemma \ref{lem:SFPpreliminary}, the set
\[S=\{a\in[H_{\omega_2}]^\omega\st\forall i<\omega_1\ a\cap\omega_1\in D_i\To\sup(a\cap\omega_2)\in S_i\}\]
is fully spread out. This means that the forcing $\P_S$ is subcomplete, and by Fact \ref{fact:P_SisCountablyDistributiveAndExhaustive}, it is countably distributive and hence in $\Gamma$. Now let $N=\kla{H_{\omega_2},\in,\vD,\vec{S},0,1,\vec{\xi}}$, viewed as a model of the language which has a predicate symbol for each $D_i$, $S_i$ and $i$, for $i<\omega_1$. Let $\omega_2\sub X$ be such that $M=N|X\prec N$ and such that $X$ has cardinality $\omega_2$. Let $\vec{M}$ be a continuous $\in$-chain through $S$, added by $\P_S$, and let's assume that each $M_i$ is an elementary submodel of $\kla{H_\kappa,\in,\vec{D},\vec{S}}$. Using the argument of Theorem \ref{thm:infSCSRPimpliesSFPkappa}, it follows that in $\V[\vM]$, there is a normal function $f:\omega_1\To\omega_2^\V$, cofinal in $\On\cap M$, such that for every $i<\omega_1$, $f``D_i\sub S_i$. This can be expressed as a $\Sigma_1$-statement about $M$. By $\BFA(\Gamma,{\le}\omega_2)$, there are in $\V$ a transitive model $\bM$ and an elementary embedding $j:\bM\prec m$ such that this $\Sigma_1$-statement is true of $\bM$. Because of the availability of the constant symbols for the countable ordinals, it follows that $\omega_1\sub\ran(j)$, and hence $j\rest\omega_1=\id$. And since $\omega_1$ is the largest cardinal in $M$, the same is true in $\bM$, and hence, letting $\bkappa=\On\cap\bM$, it follows that $j\rest\bkappa$ is the identity. In fact, since the transitive closure of every set in $\bM$ has size at most $\omega_1$, it follows that $j$ is the identity. In any case, for $i<\omega_1$, we have that $\dot{S}_i^\bM=S_i\cap\bkappa$, and $\dot{D}_i^\bM=D_i$, for $i<\omega_1$. Now let $f$ witness that the above statement is true of $\bM$. Then $f:\omega_1\To\bkappa$ is cofinal, normal, and for all $i<\omega_1$, $f``D_i\sub S_i$. This shows that $\SFP{\omega_2}$ holds.

By Observation \ref{obs:CharacterizationOfE-versionOfSFPkappa} and Fact \ref{fact:SFPandArithmetic}, $\SFP{\omega_2}$ implies that $\omega_2^{\omega_1}=2^{\omega_1}=\omega_2$. Hence, by \ref{item:InGeneralNeedCardArithmetic}, we see that $\Gamma$-$\SRP(\omega_2)$ holds.
But $\Gamma$-$\SRP(\omega_2)$ implies $\SC$-$\SRP(\omega_2)$ (and by the way, the converse is also true, because $\Gamma\sub\SC$), because if $S\sub[H_\kappa]^\omega$ is fully spread out, then $\P_S$ is subcomplete, and $\P_S$ is always countably distributive, hence uncountable cofinality preserving. Thus, $\P_S$ is in $\Gamma$, making $S$ $\Gamma$-projective stationary. Therefore, by $\Gamma$-$\SRP(\omega_2)$, there is a continuous $\in$-chain through $S$.
\end{proof}

Now I'm ready to put the pieces together and  construct the model in which $\SC$-$\SRP(\omega_2)$ holds but $\BSCFA({\le}\omega_2)$ fails.

\begin{thm}
\label{thm:SeparatingSCSRPfromSCFAatOmega2}
Let $\Gamma$ be the class of subcomplete, uncountable cofinality preserving forcing notions. If $\ZFC$ is consistent with $\BFA(\Gamma,{\le}\omega_2)$, then $\ZFC$ is consistent with the conjunction of the following statements:
\begin{enumerate}[label=\textnormal{(\arabic*)}]
\item
\label{item:CH}
\CH,
\item
\label{item:BFAGamma}
$\BFA(\Gamma,{\le}\omega_2)$,
\item
\label{item:SC-SRP(omega2)}
$\SC$-$\SRP(\omega_2)$,
\item
\label{item:LcorrectAboutUncountableCF}
$L$ is correct about uncountable cofinalities, that is, for every ordinal $\alpha$, if $\cf^L(\alpha)>\omega$, then $\cf(\alpha)>\omega$,
\item
\label{item:NotBSCFA(omega_2)}
$\neg\BSCFA({\le}\omega_2)$.
\end{enumerate}
\end{thm}

\begin{proof}
The construction starts in a model of $\BFA(\Gamma,{\le}\omega_2)$. The argument of \cite[Lemma 3.10]{Fuchs:HierarchiesOfForcingAxioms} then shows that $\kappa=\omega_2$ is ${+}1$-reflecting in $L$. Indeed, going through the proof shows that only forcing notions in $\Gamma$ are used. So let us work in $L$ of that model, where $\kappa$ is ${+}1$-reflecting. By \cite[Lemma 4.9]{Fuchs:HierarchiesOfForcingAxioms}, $\kappa$ is remarkably ${\le}\kappa$-reflecting in $L$ (I do not want to go in the details here and explain what this means, but rather use the results of \cite{Fuchs:HierarchiesOfForcingAxioms} as a black box as much as possible).
Now the argument of \cite[Lemma 4.13]{Fuchs:HierarchiesOfForcingAxioms} shows that in $L$, there is a $\kappa$-c.c.~forcing notion $\P$ such that if $G$ is generic for $\P$, then in $L[G]$, a principle called $\mathsf{wBFA}(\Gamma,{\le}\kappa)$ holds (I don't want to define here what this principle says, since it will turn out to be equivalent to $\BFA(\Gamma,{\le}\kappa)$ in the present situation). $L[G]$ is the desired model. I will show that it satisfies \ref{item:CH}-\ref{item:NotBSCFA(omega_2)}.

The forcing $\P$ is of the form $\P_0*\dot{\P}_1$, where $\P_0$ is Woodin's fast function forcing at $\kappa$ and $\dot{\P}_1$ is a $\P_0$-name for an iteration of forcings in $\Gamma$ as in Theorem \ref{thm:IteratingSubcompleteUncountableCofinalityPreserving}. The forcing $\P_0$ is $\kappa$-c.c.~and (much more than) countably closed. It follows that the composition $\P=\P_0*\dot{\P}_1$ is in $\Gamma$, and hence, it follows that in $L[G]$, $L$ is correct about uncountable cofinalities, that is, \ref{item:LcorrectAboutUncountableCF} holds in $L[G]$. This implies, by Lemma \ref{lem:StationarilyManyL-RegularsHaveCountableCofinality}, that $\BSCFA({\le}\omega_2)$ fails, since otherwise, stationarily many ordinals below $\omega_2$ would have to be regular in $L$ yet of countable cofinality in $L[G]$ (but there is not a single ordinal like that). So \ref{item:NotBSCFA(omega_2)} holds in $L[G]$.

Let $G=G_0*G_1$, where $G_0$ is $\P_0$-generic over $L$ and $G_1$ is $\P_1=\dot{\P}_1^{G_0}$-generic over $L[G_0]$. By \cite[Lemma 4.13]{Fuchs:HierarchiesOfForcingAxioms}, $\kappa$ is still ${+}1$-reflecting in $L[G_0]$, and in particular inaccessible.

Working in $L[G_0]$ temporarily, let me analyze the iteration giving rise to $\P_1$. It is an iteration of length $\kappa$ such that each initial segment of the iteration is in $\V_\kappa$ (in the sense of $L[G_0]$). 
Due to the intermediate collapses in the iteration, it follows that $\kappa=\omega_2^{L[G]}$. Thus, in $L[G]$, we have $\wBFA(\Gamma,{\le}\omega_2)$, which, by \cite[Obs.~4.7]{Fuchs:HierarchiesOfForcingAxioms}, is equivalent to $\BFA(\Gamma,{\le}\omega_2)$. Thus, we have \ref{item:BFAGamma}. By part \ref{item:ForSCcapUCPdontNeedIt} of Lemma \ref{lem:BFAimpliesSRPatOmega_2}, this implies \SC-$\SRP(\omega_2)$, so that \ref{item:SC-SRP(omega2)} is satisfied.
The collapses in the iteration will force \CH, and once \CH is true, it remains true, since no reals are added, so we have \ref{item:CH}, completing the proof.
\end{proof}

In the follow-up article \cite{CoxFuchs:DSRP}, joint with Sean Cox, we will introduce a diagonal version of \SRP, which strengthens \SRP, and we will consider the canonical fragments of this principle. The principle is designed in such a way that it captures those $\MM$-consequences on diagonal reflection that \SRP fails to capture. Thus, the subcomplete fragment of the diagonal \SRP implies that for regular $\kappa>2^\omega$, $\DSR(\omega_1,S^\kappa_\omega)$ holds, and the full principle (that is, the stationary set preserving fragment) implies this for regular $\kappa>\omega_1$. Thus, neither Larson's result separating \SRP from \MM, nor Theorem \ref{thm:FLH-SeparatingSRPfromuDSRhigherUp} serve to separate the subcomplete fragment of the diagonal \SRP from \SCFA. However, the previous result, Theorem \ref{thm:SeparatingSCSRPfromSCFAatOmega2}, does provide such a separation at the level $\omega_2$ because, using the terminology used in its statement, in the model constructed, we have $\BFA(\Gamma,{\le}\omega_2)+\CH+2^{\omega_1}=\omega_2$, and this implies the subcomplete fragment of the diagonal \SRP. Since $\BSCFA({\le}\omega_2)$ fails in the model, this shows that the subcomplete fragment of the diagonal \SRP at $\omega_2$ does not imply $\BSCFA({\le}\omega_2)$.

\section{Questions}
\label{sec:Questions}

I will list questions by the related section in this article.

Section \ref{sec:Gamma-ProjectiveStationarity}: in subsection \ref{subsec:RelativizingToForcingClass}, the present formulation of the $\Gamma$-fragment of \SRP is given, postulating that if the \emph{natural} forcing $\P_S$ to add a continuous $\in$-chain through a stationary set $S$ is in $\Gamma$, then such a sequence exists. A potentially stronger formulation would ask that if there is \emph{any} forcing in $\Gamma$ that adds such a sequence, then such a sequence should exist. The question is whether this can be a stronger principle. More broadly, can there be forcings in $\Gamma$ that add such a sequence when $\P_S$ is not in $\Gamma$? Subsection \ref{subsec:SCfragment} introduced the subcomplete fragment of \SRP, and an early observation was that $\SC$-$\SRP(\kappa)$ holds trivially if $\kappa\le 2^\omega$. It is then natural to ask whether \SC-\SRP is consistent with $2^\omega>\omega_2$. The same can be asked about \SCFA instead of \SC-\SRP.

Section \ref{sec:SRPandConsequences}: there is a lot of room for questions here. Many consequences of \SRP obviously don't follow from its subcomplete fragment, but many others might. For example, the weak reflection principle $\mathsf{WRP}$, or the strong Chang conjecture, would be candidates. One may ask the same questions about the full forcing axiom \SCFA. In fact, these principles follow from $\SCFA^+$, so assuming the consistency of a supercompact cardinal,  they are consistent with $\SCFA^+$, together with $\diamondsuit$, say. It would also be interesting to explore consequences of Theorem \ref{thm:infSC-SRP+CHimpliesMutuallyStationarySimultaneousReflection} on the mutual stationarity of sequences of sets of exact simultaneous reflection points. Another question is whether \MM implies the full principle $\SFP{\omega_2}$ in which it is not assumed that the sequence $\vD$ is a \emph{maximal} partition of $\omega_1$ into stationary sets (see Definition \ref{def:SFPkappa}. If not, then this would be a consequence of the subcomplete fragment of $\SRP$ with \CH that does not follow from \MM.

Section \ref{sec:SRPLimitations}: the first main question for this section concerns subsections \ref{subsec:GeneralSetting} and \ref{subsec:CHsetting-LessCanonical}, and asks whether the combination of \CH with the subcomplete fragment of \SRP implies $\OSR{\omega_2}$, or $\uDSR(\lambda,S^\kappa_\omega)$, for any regular $\kappa\ge\omega_2$ and any $\lambda$ with $1\le\lambda\le\omega_1$. A negative answer would separate $\text{\SC-\SRP}+\CH$ from $\SCFA$, in fact, it would even separate $\text{\SC-\SRP}+\CH$ from the subcomplete fragment of the diagonal strong reflection principle mentioned at the end of Section \ref{subsec:CHsetting:CanonicalButOnlyUpToOmega2}, to be introduced in \cite{CoxFuchs:DSRP}.
The underlying question here is how to guarantee the preservation of spread out sets. Regarding subsection \ref{subsec:CHsetting:CanonicalButOnlyUpToOmega2}, there is a fundamental problem concerning the difference between $\infty$-subcompleteness and subcompleteness: can a forcing be found that is $\infty$-subcomplete but not subcomplete? Does the Iteration Theorem \ref{thm:IteratingSubcompleteUncountableCofinalityPreserving} go through for $\infty$-subcomplete forcing? Can the separation result be modified to show that $\text{$\infSC$-$\SRP(\omega_2)$}+\CH$ does not imply $\infty$-$\SCFA(\omega_2)$? Finally, is there a global version of the result, separating \SCFA from the combination of the subcomplete fragment of \SRP with \CH? This would most likely also separate the combination of the subcomplete fragment of the diagonal strong reflection principle with \CH from \SCFA.

Questions abound.

\bibliographystyle{abbrv}


\begin{thebibliography}{10}

\bibitem{Aspero:MaximalBFA}
D.~Asper\'{o}.
\newblock A maximal bounded forcing axiom.
\newblock {\em Journal of Symbolic Logic}, 67(1):130--142, 2002.

\bibitem{Bagaria:BFAasPrinciplesOfGenAbsoluteness}
J.~Bagaria.
\newblock Bounded forcing axioms as principles of generic absoluteness.
\newblock {\em Archive for Mathematical Logic}, 39:393--401, 2000.

\bibitem{BagariaGitmanSchindler:RemarkableWeakPFA}
J.~Bagaria, V.~Gitman, and R.~Schindler.
\newblock Remarkable cardinals, structural reflection, and the weak proper
  forcing axiom.
\newblock {\em Archive for Mathematical Logic}, 56(1):1--20, 2017.

\bibitem{ASS}
J.~Barwise.
\newblock {\em Admissible Sets and Structures}.
\newblock Springer, Berlin, 1975.

\bibitem{BHM:WeakSaturationPropertiesOfIdeals}
J.~E. Baumgartner, A.~Hajnal, and A.~M{\'{a}}t{\'{e}}.
\newblock Weak saturation properties of ideals.
\newblock In A.~Hajnal, R.~Rado, and V.~T. S{\'{o}}s, editors, {\em Infinite
  and finite sets, vol.~I}, pages 137--158. North-Holland, Amsterdam, 1973.

\bibitem{Bekkali:TopicsInST}
M.~Bekkali.
\newblock {\em Topics in Set Theory. Lebesgue Measurability, Large Cardinals,
  Forcing Axioms, Rho Functions}.
\newblock Springer, 1991.
\newblock ISBN 978-3-540-47422-7.

\bibitem{ClaverieSchindler:AxiomStar}
B.~Claverie and R.~Schindler.
\newblock Woodin's axiom {$(*)$}, bounded forcing axioms, and precipitous
  ideals on {$\omega_1$}.
\newblock {\em Journal of Symbolic Logic}, 77(2):475--498, 2012.

\bibitem{CoxFuchs:DSRP}
S.~Cox and G.~Fuchs.
\newblock The diagonal strong reflection principle and its fragments.
\newblock {\em in preparation}.

\bibitem{SquaresScalesStationaryReflection}
J.~Cummings, M.~Foreman, and M.~Magidor.
\newblock Squares, scales and stationary reflection.
\newblock {\em Journal of Mathematical Logic}, 01(01):35--98, 2001.

\bibitem{CummingsMagidor:MMandWeakSquare}
J.~Cummings and M.~Magidor.
\newblock Martin's {M}aximum and weak square.
\newblock {\em Proceedings of the American Mathematical Society},
  139(9):3339--3348, 2011.

\bibitem{FengJech:ProjectiveStationarityAndSRP}
Q.~Feng and T.~Jech.
\newblock Projective stationary sets and a strong reflection principle.
\newblock {\em Journal of the London Mathematical Society}, 58(2):271--283,
  1998.

\bibitem{FengJechZapletal:StructureOfStationarySets}
Q.~Feng, T.~Jech, and J.~Zapletal.
\newblock On the structure of stationary sets.
\newblock {\em Science in China Series A: Mathematics}, (50):615--627, 2007.

\bibitem{ForemanMagidor:MutualStationarity}
M.~Foreman and M.~Magidor.
\newblock Mutually stationary sequences of sets and the non-saturation of the
  non-stationary ideal on on {$P_{\varkappa}(\lambda)$}.
\newblock {\em Acta Mathematica}, 186:271--300, 2001.

\bibitem{FMS:MM1}
M.~Foreman, M.~Magidor, and S.~Shelah.
\newblock Martin's maximum, saturated ideals, and non-regular ultrafilters.
  {P}art {I}.
\newblock {\em Annals of Mathematics}, 127(1):1--47, 1988.

\bibitem{HFriedman:OnClosedSetsOfOrdinals}
H.~Friedman.
\newblock On closed sets of ordinals.
\newblock {\em Proceedings of the American Mathematical Society},
  43(1):393--401, 1974.

\bibitem{Fuchs:ParametricSubcompleteness}
G.~Fuchs.
\newblock Closure properties of parametric subcompleteness.
\newblock {\em Archive for Mathematical Logic}, 57(7-8):829--852, 2018.

\bibitem{Fuchs:HierarchiesOfForcingAxioms}
G.~Fuchs.
\newblock Hierarchies of forcing axioms, the continuum hypothesis and square
  principles.
\newblock {\em Journal of Symbolic Logic}, 83(1):256--282, 2018.

\bibitem{Fuchs:HierarchiesOfRA}
G.~Fuchs.
\newblock Hierarchies of (virtual) resurrection axioms.
\newblock {\em Journal of Symbolic Logic}, 83(1):283--325, 2018.

\bibitem{Fuchs:SCprinciplesAndDefWO}
G.~Fuchs.
\newblock Subcomplete forcing principles and definable well-orders.
\newblock {\em Mathematical Logic Quarterly}, 64(6):487--504, 2018.

\bibitem{Fuchs:DiagonalReflection}
G.~Fuchs.
\newblock Diagonal reflections on squares.
\newblock {\em Archive for Mathematical Logic}, 58(1):1--26, 2019.

\bibitem{Fuchs:ATP}
G.~Fuchs.
\newblock Aronszajn tree preservation and bounded forcing axioms.
\newblock {\em Submitted}, 2020.
\newblock Preprint at arXiv:2001.03105 [math.LO].

\bibitem{Fuchs-LambieHanson:SeparatingDSR}
G.~Fuchs and C.~Lambie-Hanson.
\newblock Separating diagonal stationary reflection principles.
\newblock {\em Accepted for publication in the Journal of Symbolic L}, 2020.
\newblock Preprint at arXiv:2002.12862 [math.LO].

\bibitem{FuchsRinot:WeakSquareStationaryReflection}
G.~Fuchs and A.~Rinot.
\newblock Weak square and stationary reflection.
\newblock {\em Acta Mathematica Hungarica}, 155(2):393--405, 2018.

\bibitem{FuchsSwitzer:IterationTheorems}
G.~Fuchs and C.~Switzer.
\newblock Iteration theorems for subversions of forcing classes.
\newblock {\em In preparation}, 2019.

\bibitem{GoldsternShelah:BPFA}
M.~Goldstern and S.~Shelah.
\newblock The bounded proper forcing axiom.
\newblock {\em Journal of Symbolic Logic}, 60(1):58--73, 1995.

\bibitem{ST3}
T.~Jech.
\newblock {\em Set {T}heory: {T}he {T}hird {M}illenium {E}dition, {R}evised and
  {E}xpanded}.
\newblock Springer {M}onographs in {M}athematics. Springer, Berlin, Heidelberg,
  2003.

\bibitem{Jech:StationarySetsHST}
T.~Jech.
\newblock Stationary sets.
\newblock In M.~Foreman, A.~Kanamori, and M.~Magidor, editors, {\em Handbook of
  Set Theory}, volume~1, pages 93--128. Springer, 2009.

\bibitem{FS}
R.~B. Jensen.
\newblock The fine structure of the constructible hierarchy.
\newblock {\em Annals of Mathematical Logic}, 4:229--308, 1972.

\bibitem{Jensen:ExtendedNamba}
R.~B. Jensen.
\newblock The extended {N}amba problem.
\newblock Handwritten notes, available at
  https://www.mathematik.hu-berlin.de/\textasciitilde raesch/org/jensen.html,
  2009.

\bibitem{Jensen:FAandCH}
R.~B. Jensen.
\newblock Forcing axioms compatible with {CH}.
\newblock Handwritten notes, available at
  https://www.mathematik.hu-berlin.de/\textasciitilde raesch/org/jensen.html,
  2009.

\bibitem{Jensen:SPSCF}
R.~B. Jensen.
\newblock Subproper and subcomplete forcing.
\newblock 2009.
\newblock Handwritten notes, available at
  https://www.mathematik.hu-berlin.de/\textasciitilde raesch/org/jensen.html.

\bibitem{Jensen:AdmissibleSets}
R.~B. Jensen.
\newblock Admissible sets.
\newblock {\em Handwritten notes}, 2010.

\bibitem{Jensen2014:SubcompleteAndLForcingSingapore}
R.~B. Jensen.
\newblock Subcomplete forcing and {${\mathcal{L}}$}-forcing.
\newblock In C.~Chong, Q.~Feng, T.~A. Slaman, W.~H. Woodin, and Y.~Yang,
  editors, {\em {E}-recursion, forcing and {$C^*$}-algebras}, volume~27 of {\em
  Lecture Notes Series, Institute for Mathematical Sciences, National
  University of Singapore}, pages 83--182, Singapore, 2014. World Scientific.

\bibitem{Jensen:IterationTheorems}
R.~B. Jensen.
\newblock Iteration theorems for subcomplete and related forcings.
\newblock Handwritten notes, available at
  https://www.mathematik.hu-berlin.de/\textasciitilde raesch/org/jensen.html,
  2014-15.

\bibitem{Larson:SeparatingSRP}
P.~Larson.
\newblock Separating stationary reflection principles.
\newblock {\em Journal of Symbolic Logic}, 65(1):247--258, 2000.

\bibitem{Miyamoto:WeakSegmentsOfPFA}
T.~Miyamoto.
\newblock A note on weak segments of {PFA}.
\newblock In C.~Chong, Q.~Feng, D.~Ding, Q.~Huang, and M.~Yasugi, editors, {\em
  Proceedings of the Sixth Asian Logic Conference}, pages 175--197, 1998.

\bibitem{Todorcevic:NotesOnForcingAxioms}
S.~Todor\v{c}evi\'{c}.
\newblock {\em Notes on Forcing Axioms}.
\newblock World Scientific, 2014.

\bibitem{VialeEtAl:BooleanApproachToSPiterations}
M.~Viale, G.~Audrito, and S.~Steila.
\newblock A boolean algebraic approach to semiproper iterations.
\newblock 2014.
\newblock Preprint: arXiv:1402.1714 [math.LO].

\bibitem{Villaveces:StrongUnfoldability}
A.~Villaveces.
\newblock Chains of end elementary extensions of models of set theory.
\newblock {\em Journal of Symbolic Logic}, 63(3):1116--1136, 1998.

\end{thebibliography}
\end{document}